\documentclass[12pt]{article}

\usepackage{latexsym}
\usepackage{a4wide}
\usepackage{amscd}
\usepackage{graphics}
\usepackage{amsmath}
\usepackage{amssymb}
\usepackage{esint}
\usepackage{mathrsfs}
\usepackage{amsthm}
\usepackage{bbm}
\usepackage{color}
\usepackage{accents}
\usepackage{enumerate}
\usepackage[pdftex]{hyperref}

\input xy
\xyoption{all}
\usepackage{tikz-cd}

\newcommand{\N}{\mathbb{N}}                     
\newcommand{\Z}{\mathbb{Z}}                     
\newcommand{\R}{\mathbb{R}}                     
\newcommand{\C}{\mathbb{C}}                     
\newcommand{\T}{\mathbb{T}}                     
\newcommand{\coker}{\mathrm{coker\,}}           
\newcommand{\ind}{\mathrm{ind\,}}               
\newcommand{\supp}{\mathrm{supp\,}}             

\newtheorem{thm}{\sc Theorem}[section]               
\newtheorem*{thm*}{\sc Theorem}               
\newtheorem*{cor*}{\sc Corollary}        
\newtheorem{lem}[thm]{\sc Lemma}            
\newtheorem{prop}[thm]{\sc Proposition}     
\newtheorem{rem}[thm]{\sc Remark}           

\AtEndDocument{\bigskip{\footnotesize
\noindent\textsc{Ruhr Universit\"at Bochum, Fakult\"at f\"ur Mathematik, Geb\"aude NA 4/33, D-44801 Bochum, Germany} \par  
 \noindent \textit{E-mail address}: \texttt{\href{mailto:alberto.abbondandolo@rub.de}{alberto.abbondandolo@rub.de}} \par
  \addvspace{\medskipamount}
\noindent\textsc{Seoul National University, Department of Mathematical Sciences, Research Institute in Mathematics, Gwanak-Gu, 
	Seoul 08826, South Korea} \par  
\noindent  \textit{E-mail address}: \texttt{\href{mailto:jungsoo.kang@snu.ac.kr}{jungsoo.kang@snu.ac.kr}} \par
}}

\title{Symplectic homology of convex domains and\\ Clarke's duality} 

\author{Alberto Abbondandolo and Jungsoo Kang}

\date{}

\begin{document}

\maketitle

\begin{abstract}
We prove that the Floer complex that is associated with a convex Hamiltonian function on $\mathbb{R}^{2n}$  is isomorphic to the Morse complex of Clarke's dual action functional that is associated with the Fenchel-dual Hamiltonian. This isomorphism preserves the action filtrations. As a corollary, we obtain that the symplectic capacity from the symplectic homology of a convex domain with smooth boundary coincides with the minimal action of closed characteristics on its boundary.
\end{abstract}

\section*{Introduction}

Let $H: \T \times \R^{2n} \rightarrow \R$ be a smooth time-periodic Hamiltonian function on $\R^{2n}$, endowed with its standard symplectic structure 
\[
\omega_0 := \sum_{j=1}^n dp_j \wedge dq_j.
\]
Here $\T:= \R/\Z$ denotes the 1-torus. The corresponding time-periodic Hamiltonian vector field $X_H$ is defined as usual by 
\[
\imath_{X_H} \omega_0 = - dH,
\]
where $d$ denotes differentiation with respect to the spatial variables. The 1-periodic orbits of $X_H$ are precisely the critical points of the action functional
\[
\Phi_H(x) := \frac{1}{2} \int_{\T} J_0 \dot{x}(t)\cdot x(t)\, dt - \int_{\T} H_t(x(t))\, dt, \qquad x\in C^{\infty}(\T,\R^{2n}),
\]
where $J_0$ denotes the standard complex structure on $\R^{2n}$ mapping $(q,p)$ into $(-p,q)$, and we are using the notation $H_t(x):=H(t,x)$. Assuming non-degeneracy of all 1-periodic orbits of $X_H$ and suitable conditions on the behaviour of $H_t(x)$ for $|x|$ large, one can associate a Floer complex with $H$. This is a chain complex $F_*(H)$ of $\Z_2$-vector spaces that are generated by the 1-periodic orbits of $X_H$, whose boundary operator
\[
\partial : F_*(H) \rightarrow F_{*-1}(H)
\]
is obtained by a suitable count of the spaces of cylinders $u: \R \times \T \rightarrow \R^{2n}$ that satisfy the Floer equation
\begin{equation}
\label{floereq}
\partial_s u + J_t(u) ( \partial_t u - X_{H_t}(u)) = 0, \qquad (s,t)\in \R \times \T,
\end{equation}
and are asymptotic to pairs of periodic orbits for $s\rightarrow \pm \infty$. Here, $J$ is a generic time-periodic almost complex structure on $\R^{2n}$ that is compatible with $\omega_0$ and has a suitable behaviour at infinity. Actually, one can work with coefficients in an arbitrary abelian group instead of $\Z_2$, but in this paper we stick to $\Z_2$ coefficients, as this simplifies the presentation and the proofs.

The Floer complex is graded by the Conley-Zehnder index $\mu_{CZ}(x)$, an integer that counts the half-windings of the differential of the flow of $X_{H}$ along $x$ in the symplectic group. The homology of this chain complex, which is known as the Floer homology of $H$, is a considerably stable object and depends only on the behaviour at infinity of the Hamiltonian $H$. It is denoted by $HF_*(H)$. See e.g.\ \cite{sal99}, \cite{ad14}, \cite{as18} and references therein for more information on Hamiltonian Floer theory.

The Floer equation (\ref{floereq}) can be seen as a negative gradient equation for the action functional $\Phi_H$, and Floer homology should be thought as a kind of Morse theory for this functional, which does not have a Morse theory in the usual sense because all its critical points have infinite Morse index and co-index.

When $H$ is strongly convex (i.e.~has an everywhere positive definite second differential) and superlinear in the spatial variable, there is another way of deriving a Morse theory for the 1-periodic orbits of $X_H$. Indeed, following Clarke's \cite{cla79} one can introduce the following dual action functional
\[
\Psi_{H^*}(x) := - \frac{1}{2} \int_{\T} J_0 \dot{x}(t)\cdot x(t)\, dt +\int_{\T} H^*_t(J_0 \dot{x}(t))\, dt, 
\]
where $H^*$ denotes the Fenchel conjugate of $H$ in the spatial variable $z=(q,p)$, which is still a strongly convex and superlinear function (see also \cite{ce80, cla81, ce82}). This functional is invariant under translations, so it can be seen as a functional on the quotient space of smooth closed curves in $\R^{2n}$ modulo translations, that we identify with the space of closed curves with zero mean. The crucial observation of Clarke was that there is a natural one-to-one correspondence between the critical points of $\Phi_H$ and $\Psi_{H^*}$: A closed curve $x$ is a critical point of $\Psi_{H^*}$ if and only if $x+v_0$ is a critical point of $\Phi_H$, for a suitable translation vector $v_0\in \R^{2n}$. Moreover, the direct action functional and the dual one have the same value at their corresponding critical points:
\[
\Psi_{H^*}(x) = \Phi_{H}(x+v_0).
\]
See Ekeland's book \cite{eke90} for a general approach to Clarke's duality, with special emphasis on Hamiltonian systems.

Clarke's dual functional $\Psi_{H^*}$ has better analytical properties than the direct action functional $\Phi_H$: In $\Phi_H$, the indefinite quadratic form 
\begin{equation}
\label{quadratic}
\frac{1}{2} \int_{\T} J_0 \dot{x}(t)\cdot x(t)\, dt,
\end{equation}
is the leading part (the part involving derivatives of $x$, as opposed to the integral of $H_t(x)$, which does not),  whereas in $\Psi_{H^*}$ the leading term is the integral of $H^*_t(J_0 \dot x)$, which defines a convex functional. A consequence of this is that the critical points of $\Psi_{H^*}$ have finite Morse index. Moreover, under suitable assumptions on $H$ (e.g.~$H$ subquadratic, so that $H^*$ is superquadratic), one can find critical points of $\Psi_{H^*}$ just by minimization.

Here we are interested in the global properties of $\Psi_{H^*}$ and its critical points. These properties can be encoded in the Morse complex of $\Psi_{H^*}$. Constructing such a Morse complex presents some analytical difficulties, which we will mention in due time, but it is possible and the outcome is a Morse theory that, unlike Floer's theory for $\Phi_H$, is essentially finite dimensional. It is then a natural question to compare the Floer chain complex associated with $\Phi_H$ to the Morse complex induced by $\Psi_{H^*}$. It is to this question that this paper is devoted.

Before stating our main results, we need to clarify the class of convex Hamiltonians we are going to work with. We shall assume that $H\in C^{\infty}(\T\times \R^{2n})$ is non-degenerate, meaning that all the 1-periodic orbits of $X_H$ are non-degenerate, and quadratically convex, meaning that
\[
\underline{h} |u|^2 \leq d^2 H_t(x)[u,u] \leq \overline{h} |u|^2 \qquad \forall x,u\in \R^{2n},
\]
for suitable positive numbers $\underline{h}$ and $\overline{h}$. Moreover, we shall assume that $H$ is non-resonant at infinity. This means that there are positive numbers $\epsilon$ and $r$ such that every smooth curve $x:\T \rightarrow \R^{2n}$ satisfying 
\[
\|\dot{x}-X_H(x)\|_{L^2(\T)}< \epsilon
\]
has $L^2$-norm bounded by $r$. This is a kind of Palais-Smale condition for the action functional $\Phi_H$ and implies in particular that all 1-periodic orbits are contained in a compact set. By the non-degeneracy assumption, $X_H$ has then only finitely many 1-periodic orbits. 

Here is a concrete condition on the behaviour of $H$ at infinity that guarantees that it is non-resonant at infinity:
\[
H(t,z) = \eta |z|^2 + \xi \qquad \mbox{for } |z| \geq R, \; \forall t\in \T,
\]
where $R>0$, $\xi\in \R$, and $\eta\in (0,+\infty) \setminus \pi \N$. A more general class of non-resonant Hamiltonians is described in Lemma \ref{cresce} below.

Fix a Hamiltonian $H\in C^{\infty}(\T\times \R^{2n})$ that is non-degenerate, quadratically convex and non-resonant at infinity. The Floer complex $F_*(H)$ is a finitely generated chain complex. It is filtered by the action: $F^{<a}_*(H)$ denotes the subcomplex that is generated by all 1-periodic orbits $x$ of $X_H$ with $\Phi_H(x)<a$. 

As mentioned above, constructing a Morse complex for $\Psi_{H^*}$ presents some analytical difficulties. Indeed, under the above assumptions on $H$ a suitable space for studying $\Psi_{H^*}$ is the space $\mathbb{H}_1$ of closed curves $x: \T \rightarrow \R^{2n}$ of Sobolev class $H^1$ and zero mean. However, the function $\Psi_{H^*}$ is nowhere twice Fr\'echet-differentiable on $\mathbb{H}_1$, except for the very special case in which $H$ is a quadratic form in the spatial variable. It is however twice Gateaux-differentiable, so the Morse index and nullity of its critical points are well-defined, and as mentioned above both finite, but the construction of a Morse complex for $\Psi_{H^*}$ is problematic since this function is not $C^2$. Following \cite{vit89b}, we will overcome this difficulty by performing a saddle point reduction. Indeed, we shall construct a finite dimensional submanifold $M$ of $\mathbb{H}_1$ such that: 
\begin{enumerate}[(i)]
\item the restriction $\psi_{H^*}$ of $\Psi_{H^*}$ to $M$ is a smooth Morse function;
\item $M$ is diffeomorphic to $\R^N$ for some large $N$ and contains all critical points of $\Psi_{H^*}$;
\item a point $x\in M$ is a critical point of $\psi_{H^*}$ if and only if is a critical point of $\Psi_{H^*}$, and its Morse index and nullity with respect to these two functions  coincide;
\item $\psi_{H^*}$ satisfies the Palais-Smale compactness condition.
\end{enumerate}
 The manifold $M$ is defined by splitting $\mathbb{H}_1$ into a suitable finite dimensional space $\mathbb{H}_1^{N,+}$ and its orthogonal complement $\widehat{\mathbb{H}}_1^{N,+}$ and showing that if $N$ is large enough then the functions $y\mapsto \Psi_{H^*}(x,y)$ are strictly convex and have a unique minimum $Y(x)$ on $\widehat{\mathbb{H}}_1^{N,+}$, for every $x\in \mathbb{H}_1^{N,+}$. The manifold $M$ is then the graph of the map $Y: \mathbb{H}_1^{N,+} \rightarrow \widehat{\mathbb{H}}_1^{N,+}$.

 The above properties allow us to associate a Morse complex $M_*(\psi_{H^*})$ with the function $\psi_{H^*}$. This chain complex is graded by the Morse index, is filtered by the values of $\psi_{H^*}$, and its homology is isomorphic to the singular homology of the pair $(M,\{\psi_{H^*}<a\})$, where $a$ is any real number smaller than the minimum of $\psi_{H^*}$  on its critical set (in general, $\psi_{H^*}$ is unbounded from below).

The Conley-Zehnder index $\mu_{CZ}(x)$ of a critical point $x$ of $\Phi_H$ is related to the Morse index $\mathrm{ind}(\pi(x);\psi_{H^*})$ of $\pi(x)$ as a critical point of $\psi_{H^*}$ by the identity
\[
\mu_{CZ}(x) = \mathrm{ind}(\pi(x);\psi_{H^*}) + n.
\]
Here, $\pi$ is the standard projection onto the space $\mathbb{H}_1$ of curves with zero mean.
We can now state the main result of this paper in the following form.

\begin{thm*}
Assume that the Hamiltonian $H\in C^{\infty}(\T\times \R^{2n})$ is non-degenerate, quadratically convex and non-resonant at infinity.
Then there is a chain complex isomorphism
\[
\Theta: M_{*-n}(\psi_{H^*}) \rightarrow F_*(H)
\]
from the Morse complex of the reduced dual functional $\psi_{H^*}$ to the Floer complex of the direct action functional $\Phi_H$. This isomorphism preserves the action filtrations.
\end{thm*} 

The theorem is proven in Section \ref{sec:isom}. The main ideas in the construction of the chain isomorphism $\Theta$ will be sketched at the end of this introduction. See also Appendix \ref{appB} for a heuristic argument that shows that these two chain complexes should be isomorphic and suggests the way for the rigorous proof that is contained in this paper.

Our next result is a corollary of the above theorem and concerns the symplectic homology of convex domains in $\R^{2n}$ and the resulting $SH$-capacity of such domains. 

Symplectic homology is an algebraic invariant that is associated with certain symplectic compact manifolds with boundary. It was introduced by Floer and Hofer in \cite{fh94} and further developed in \cite{fhw94}, \cite{cfh95} and \cite{cfhw96}. See also \cite{vit99} for a somehow different and quite fruitful approach, and the surveys \cite{oan04} and \cite{sei08}.

In Section \ref{sec:SHcap}, we recall its definition for smooth starshaped domains $W$, that is, bounded open subsets of $\R^{2n}$ that are starshaped with respect to a point  and have a smooth boundary that is transverse to all the lines through this point. Being a smooth hypersurface, the boundary of $W$ carries a 1-dimensional foliation, that is called the characteristic foliation and is tangent to the kernel of the restriction of $\omega_0$ to the tangent spaces of $\partial W$. The closed leaves of this foliation are called closed characteristics, and their action is defined to be the absolute value of the integral of $\omega_0$ over a disk in $\R^{2n}$ capping them. The set of the actions of all closed characteristics - including their iterations - is called the action spectrum of $\partial W$. It is a closed subset of $\R$ consisting of positive numbers and having zero measure.

Given a smooth starshaped domain $W$, the filtered symplectic homology $SH^{<a}(W)$ is defined as the direct limit of the filtered Floer homologies of suitable Hamiltonian functions on $\R^{2n}$ adapted to $W$. This limit formalizes the idea of a Hamiltonian that is zero on $W$ and infinitely steep outside of it. There are natural homomorphisms
\begin{equation}
\label{connhom}
SH^{<a}_*(W) \rightarrow SH^{<b}_*(W)
\end{equation}
whenever $a\leq b$. By construction, 
\begin{equation}
\label{isobasso}
SH^{<\epsilon}_*(W) \cong H_{*+n}(W,\partial W),
\end{equation}
for every small enough positive number $\epsilon$, and (\ref{connhom}) may fail to be an isomorphism only when the interval $[a,b)$ contains elements of the spectrum of $\partial W$. Moreover, the full symplectic homology $SH_*(W)$ is zero. By (\ref{isobasso}), the space $SH^{<\epsilon}_n(W)$ is isomorphic to $\Z_2$, but since $SH_*(W)=0$ the homomorphism 
\[
SH^{<\epsilon}_n(W) \rightarrow SH^{<a}_n(W)
\]
must vanish for $a$ large enough. One defines the $SH$-capacity $c_{SH}(W)$ of $W$ as the infimum of all numbers $a$ for which this happens:
\[
c_{SH}(W) := \inf \{ a>\epsilon \mid SH^{<\epsilon}_n(W) \rightarrow  SH^{<a}_n(W) \mbox{ is zero}\}.
\]
This capacity is also known as Floer-Hofer capacity, as its definition is based on the seminal paper \cite{fh94}, or Floer-Hofer-Wysocki capacity, as it was first defined in \cite{fhw94}. The definition that we have sketched here requires $W$ to be a smooth starshaped domain, but $c_{SH}$ can actually be extended to arbitrary open subsets of $\R^{2n}$ and fulfills the axiomatic properties of a symplectic capacity, see \cite{fhw94}.

When $W$ is a smooth starshaped domain, the number $c_{SH}(W)$ is always an element of the spectrum of $\partial W$. Easy examples show that in general it might not coincide with the minimum of the spectrum of $\partial W$. See Section \ref{sec:SHcap} below for the description of such an example from \cite[Chapter 3.5]{hz94}.
Building on our main theorem stated above, we will show in Section \ref{sec:SHcap} that the symplectic capacity of a smooth convex domain $C$ coincides with the minimum of the action spectrum of $\partial C$.

\begin{cor*}
Let $C$ be a convex bounded open subset of $\R^{2n}$ with smooth boundary. Then $c_{SH}(C)$ coincides with the minimum of the action spectrum of $\partial C$.
\end{cor*}

The above result has been very recently proven also by Kei Irie in \cite{iri19}, by extending to arbitrary convex bodies in $\R^{2n}$ the approach that he had developed in \cite{iri14b} for the cotangent disk bundle of domains in $\R^n$. 

The above corollary adds some more evidence to the conjecture that all symplectic capacities should coincide on convex bodies. Indeed, it has been known for a long time that the above result holds for the Hofer-Zehnder capacity and the Ekeland-Hofer capacity, see \cite{eh89,hz90}. Moreover, Hermann has shown in \cite{her04} that the Viterbo capacity and the symplectic homology capacity $c_{SH}$ coincide on domains with contact type boundary. The above corollary allows us to conclude that these four capacities coincide on the set of convex bodies.

The proof of this corollary is contained in Section \ref{sec:SHcap}. The idea is to perturb $C$ to make it strongly convex and non-degenerate, and then to see $SH_*^{<\eta}(C)$ for $\eta$ just above the minimum of the spectrum of $\partial C$ as the Floer homology of a suitable Hamiltonian $H$ that is non-degenerate, quadratically convex and non-resonant at infinity. The corresponding Floer complex is generated by two generators for each closed characteristic of $\partial C$ of minimal action plus an extra generator $z$ corresponding to the global minimum of $H$. The generator $z$ defines a cycle in $F_n(H)$ whose homology class generates the image of the homomorphism
\[
SH_n^{<\epsilon}(C) \rightarrow  SH_n^{<\eta}(C).
\]
The generator $z$ corresponds to a local minimizer $\pi(z)$ of the reduced dual action functional $\psi_{H^*}$, while the pairs of generators given by the closed characteristics of minimal action correspond to pairs of critical points of $\psi_{H^*}$ of Morse index 1 and 2, respectively. Moreover, the function $\psi_{H^*}$ is unbounded from below, and this easily implies that $\pi(z)$ is the boundary of a critical point of index 1 in the Morse complex of $\psi_{H^*}$. Using the isomorphism of our main theorem, we deduce that $z$ is a boundary in $F_*(H)$, and hence vanishes in $SH_*^{<\eta}(C)$. 
This proves that $c_{SH}(C)$ does not exceed the minimum of the spectrum of $\partial C$. Being an element in this set, $c_{SH}(C)$ must then coincide with the minimum of the spectrum of $\partial C$.

We conclude this introduction by sketching the construction of the isomorphism $\Theta$ from our main theorem. As it is now customary in Floer homological theories, see in particular \cite{as06}, the chain isomorphism
\[
\Theta: M_{*-n}(\psi_{H^*}) \rightarrow F_*(H)
\]
is defined by counting solutions of a suitable hybrid problem that relates the negative gradient flow lines of $\psi_{H^*}$ and the solutions of the Floer equation (\ref{floereq}). We now wish to describe this hybrid problem.

The quadratic form (\ref{quadratic}) is continuous on the Sobolev space $\mathbb{H}_{1/2}$ of closed curves $x:\T \rightarrow \R^{2n}$ of Sobolev class $H^{1/2}$. The corresponding bounded self-adjoint operator on $\mathbb{H}_{1/2}$ is a Fredholm operator having a finite dimensional kernel - the space of constant curves, which we denote by $\R^{2n}$ - and two infinite dimensional positive and negative eigenspaces $\mathbb{H}_{1/2}^+$ and $\mathbb{H}_{1/2}^-$. So we have the orthogonal splitting
\[
\mathbb{H}_{1/2} = \mathbb{H}_{1/2}^+ \oplus \mathbb{H}_{1/2}^- \oplus \R^{2n}.
\]

Let $x$ and $y$ be two 1-periodic orbits of $X_H$. By what we have seen above, $\pi(x)$ belongs to $M$ and is a critical point of $\psi_{H^*}$. As before, $\pi$ is the standard projection on the space of closed curves with zero mean. We denote by $W^u(\pi(x))\subset M$ its unstable manifold with respect to the negative gradient flow of $\psi_{H^*}$. The hybrid problem we are interested in is the following: We look for smooth solutions 
\[
u:[0,+\infty) \times \T \rightarrow \R^{2n}
\]
of the Floer equation (\ref{floereq}) that converge to the periodic orbit $y$ for $s\rightarrow +\infty$ and satisfy the following boundary condition for $s=0$:
\begin{equation}
\label{bdrycondintro}
u(0,\cdot) \in \pi^{-1}(W^u(\pi(x))) + \mathbb{H}_{1/2}^-.
\end{equation}
The set appearing on the right-hand side turns out to be a submanifold of the Sobolev space $\mathbb{H}_{1/2}$ with infinite dimension and infinite codimension. Its tangent space at every point is a closed vector subspace that is a compact perturbation of $\mathbb{H}_{1/2}^{-}$. 

This is a non-local and somehow non-standard boundary condition for the Floer equation - standard boundary conditions would require $u(0,\cdot)$ to take values in a Lagrangian submanifold of $\R^{2n}$ - but conditions of this kind have been considered in \cite{hec12}, \cite{hec13} and  \cite{as15}. The Fredholm analysis for the linearization of the above hybrid problem builds on Hecht's work, see \cite[Section 4.2]{hec13}. It ultimately relies on the identity
\[
\int_{[0,+\infty)\times \T} |\nabla u|^2\, ds dt = \int_{[0,+\infty)\times \T} |\overline{\partial} u|^2\, ds dt + 2 \int_{\T} u(0,\cdot)^* \lambda_0,  \qquad \forall u\in C^{\infty}_c([0,+\infty)\times \T,\R^{2n}),
\]
where $\overline{\partial} = \partial_s + J_0 \partial_t$ is the Cauchy-Riemann operator, and on the fact that the second integral on the right-hand side is 
the quadratic form
\[
2 \int_{\T} x^* \lambda_0 =  \int_{\T} J_0 \dot{x}(t)\cdot x(t)\, dt,
\]
which is negative definite on $\mathbb{H}_{1/2}^-$. This analysis is carried out in Section \ref{fredsec}.

A good functional space for studying the above hybrid problem is the space of maps $u:[0,+\infty) \times \T \rightarrow \R^{2n}$ of Sobolev class $H^1$, as the trace at $s=0$ of these maps belongs to $\mathbb{H}_{1/2}$. The usual arguments from Floer theory for showing that solutions of (\ref{floereq}) are smooth require the solutions to be in $W^{1,p}_{\mathrm{loc}}$ for some $p>2$ (see \cite[Appendix B.4]{ms04}), or at least in $C^0\cap H^1_{\mathrm{loc}}$ (see \cite[Section 2.3]{is99} or \cite{is00}), and hence cannot be applied directly here. In Appendix \ref{appA} we show how interior regularity can be obtained also starting from $H^1_{\mathrm{loc}}$ solutions, while in Section \ref{funsetsec} we deal with regularity up to the boundary. 

The compactness of the spaces of solutions of the hybrid problem relies on the following inequality relating the direct and the dual action functionals
\begin{equation}
\label{ineq}
\Phi_H(x+y) \leq \Psi_{H^*}(\pi(x)) - \frac{1}{2} \|P^- y\|_{1/2}^2,
\end{equation}
where $x:\T\to\R^{2n}$ is of Sobolev class $H^1$ and $y\in\mathbb{H}^-_{1/2}\oplus\R^{2n}$. This inequality follows from Fenchel duality and is proven in Proposition \ref{confronto}. Here, $\|\cdot\|_{1/2}$ denotes the $H^{1/2}$-norm on $\mathbb{H}_{1/2}$ and $P^-$ is the orthogonal projection onto $\mathbb{H}_{1/2}^-$. The equality holds, in particular, when $y$ is a constant loop and $x+y$ is a critical point of $\Phi_H$. 

Once all these facts have been proven, the homomorphism
\[
\Theta : M_{*-n}(\psi_{H^*}) \rightarrow F_*(H)
\]
is defined in the usual way by counting the zero-dimensional spaces of solutions of the hybrid problem. The inequality (\ref{ineq}) implies that the space of solutions of the hybrid problem with asymptotic Hamiltonian orbits $x$ and $y$ can be non-empty only when $\Phi_H(x) \geq \Phi_H(y)$. This fact is used in the proofs of the fact that $\Theta$ is an isomorphism and of the fact that it preserves the action filtrations.

There is also an alternative way to compare the Morse theories of the direct action functional $\Phi_H$ and of the dual one $\Psi_{H^*}$. The idea, which was already used in \cite{as15} in order to compare the Hamiltonian and the Lagrangian action functionals for fiberwise convex Hamiltonians on cotangent bundles, is the following: The functionals $\Phi_H$ and $\Psi_{H^*}$ can be extended to larger domains, without modifying their critical points and gradient flow lines connecting them, in such a way that the extended functionals $\widetilde{\Phi}_H$ and $\widetilde{\Psi}_{H^*}$ are obtained one from the other by a change of variables: $\widetilde{\Psi}_{H^*}=\widetilde{\Phi}_H\circ \Gamma$, where $\Gamma$ is a diffeomorphism between the domains of the extended functionals. This argument is sketched rather informally in Appendix \ref{appB}. It gives an a posteriori explanation of the fact that (\ref{bdrycondintro}) is the right boundary condition to look at in order to define the isomorphism $\Theta$, and shows which boundary condition one would need to look at in order to define directly an isomorphism going in the opposite direction.

\paragraph{Outlook} The argument behind our main theorem is quite flexible and it should be possible to adapt it to several different situations. A natural direction is to extend our isomorphism to the $S^1$-equivariant setting. This is particularly interesting in view of some recent results of Gutt and Hutchings, who defined a sequence of symplectic capacities for a starshaped domain with extremely good properties using $S^1$-equivariant symplectic homology, see \cite{gh18}. This sequence of symplectic capacities is reminiscent of another sequence of symplectic capacities that was defined by Ekeland and Hofer \cite{eh90} using the direct functional $\Phi_H$ on the space $\mathbb{H}_{1/2}$ together with the Fadell-Rabinowitz index \cite{fr78}. 
On the other hand, using Clarke's duality and the Fadell-Rabinowitz index, one can also obtain a sequence of positive numbers belonging to the action spectrum of the boundary of a smooth convex domain as in \cite{eh87,eke90}, which is monotone with respect to symplectic embeddings between smooth convex domains. It should be also possible to build an isomorphism between  the Floer homology for the Hamiltonian $H$ and the Morse homology for  $\Phi_H$ on $\mathbb{H}_{1/2}$. Once the corresponding isomorphisms are upgraded to the $S^1$-equivariant setting, this could suggest how to compare Gutt and Hutchings' symplectic capacities with the sequence of actions defined using Clarke's duality and with Ekeland and Hofer's symplectic capacities.

\paragraph{Acknowledgments} We are very grateful to Urs Fuchs for explaining us how to prove interior regularity of solutions of the Floer equation of Sobolev class $H^1_{\mathrm{loc}}$, which is the content of Appendix \ref{appA}. We would like to thank Kei Irie for discussing with us his proof of the above corollary and the comparison of the two approaches. Our gratitude goes also to three anonymous referees, whose precious comments and corrections have allowed us to improve the quality of this paper.

The research of A.~Abbondandolo is supported by the SFB/TRR 191 ``Symplectic Structures in Geometry, Algebra and Dynamics'', funded by the Deutsche Forschungsgemeinschaft. The research of J.~Kang is supported by Samsung Science and Technology Foundation under Project Number SSTF-BA1801-01. Some part of this paper was written during several visits of the second author to the Ruhr-Universit\"at Bochum and the Universit\"at Heidelberg. He would like to thank A.~Abbondandolo, P.~Albers, and G.~Benedetti for their warm hospitality.

\numberwithin{equation}{section}

\tableofcontents

\section{The action functional and the relative Morse index of its critical points}

We equip $\R^{2n}$ with coordinates $(q_1,p_1,\dots,q_n,p_n)$, with the standard Liouville form 
\[
\lambda_0 := \frac{1}{2} \sum_{j=1}^n ( p_j \, dq_j - q_j \, dp_j)
\]
and with the standard symplectic form 
\[
\omega_0 := d\lambda_0 = \sum_{j=1}^n dp_j \wedge dq_j.
\]
Note that
\begin{equation}
\label{lambda-omega}
\lambda_0(u)[v] = \frac{1}{2} \omega_0(u,v)  \qquad \forall u,v\in \R^{2n}.
\end{equation}
The linear automorphism
\[
J_0: \R^{2n} \rightarrow \R^{2n}, \qquad (q,p) \mapsto (-p,q),
\]
is the standard complex structure on $\R^{2n}$, according to the identification $\R^{2n} \cong \C^n$ given by $(q,p)\mapsto q+ip$. The symplectic form $\omega_0$ and the complex structure $J_0$ are related to the standard Euclidean scalar product on $\R^{2n}$ by the identity
\[
u\cdot v = \omega_0(J_0u,v) \qquad \forall u,v\in \R^{2n}.
\]

The Hamiltonian vector field $X_H$ associated with a smooth Hamiltonian $H:\R^{2n} \rightarrow \R$ is defined by the identity
\[
\omega_0(X_H,\cdot) = - dH,
\]
or equivalently by
\[
X_H = -J_0 \nabla H,
\]
where $\nabla$ denotes the Euclidean gradient on $\R^{2n}$.

We now fix a time-periodic smooth Hamiltonian $H:\T \times \R^{2n} \rightarrow \R$, where $\T:= \R/\Z$ denotes the 1-torus, and use the notation $H_t(x):= H(t,x)$. 
We recall that a $1$-periodic orbit $x$ of $X_H$ is said to be non-degenerate if $1$ is not an eigenvalue of the linearization of the Hamiltonian flow along $x$:
\[
1 \notin \sigma \bigl( d\phi_{X_H}^1(x(0)) \bigr),
\]
where $\phi_{X_H}^t$ denotes the (possibly non-autonomous) local flow of $X_H$. When needed, the time-periodic Hamiltonian $H$ will be assumed to be non-degenerate:

\begin{description}
\item[Non-degeneracy:] The Hamiltonian $H\in C^{\infty}(\T\times \R^{2n})$ is said to be non-degenerate if all the 1-periodic orbits of $X_H$ are non-degenerate. 
\end{description}

The 1-periodic orbits of $X_H$ are exactly the critical points of the action functional $\Phi_H : C^{\infty}(\T,\R^{2n}) \rightarrow \R$ given by
\[
\begin{split}
\Phi_H (x) &:= \int_{\T} x^*\lambda_0 - \int_{\T} H_t(x(t))\, dt \\ &=
\frac{1}{2} \int_{\T} J_0 \dot{x}(t)\cdot x(t)\, dt - \int_{\T} H_t(x(t))\, dt.
\end{split}
\]
If the second differential of $H$ in the spatial variable $z=(q,p)\in \R^{2n}$ has polynomial growth, meaning that there are $c>0$ and $N>0$ such that
\[
|d^2 H_t(z)| \leq c (1 + |z|^N) \qquad \forall (t,z)\in \T\times\R^{2n},
\]
then $\Phi_H$ is twice continuously differentiable on the Sobolev space 
\[
\mathbb{H}_{1/2}:= H^{1/2}(\T,\R^{2n}).
\] 
This space consists of all $L^2$ curves $x:\T \rightarrow \R^{2n}$ such that the coefficients $(\hat{x}_k)_{k\in \Z}$ of the Fourier decomposition
\begin{equation}
\label{fourier}
x(t) = \sum_{k\in \Z} e^{-2\pi k J_0 t} \hat{x}_k, \qquad \hat{x}_k\in \R^{2n},
\end{equation}
satisfy
\[
\sum_{k\in \Z} |k|\, |\hat{x}_k|^2 < +\infty.
\]
See \cite[Section 3.3 and Appendix A.3]{hz94} for more information on the properties of the action functional on the Sobolev space $\mathbb{H}_{1/2}$. 

The non-degeneracy of the 1-periodic orbits of $X_H$ translates into the fact that $\Phi_H$ is a Morse functional on $\mathbb{H}_{1/2}$. The critical points of $\Phi_H$ have infinite Morse index, but one can associate with them a finite relative Morse index. Indeed, this is due to the fact that the leading part of this functional has the form
\[
\frac{1}{2} \int_{\T} J_0 \dot{x}(t)\cdot x(t)\, dt = \frac{1}{2} \bigl( \|P^+ x\|^2_{1/2} -  \|P^- x\|^2_{1/2} \bigr),
\]
where $\|\cdot\|_{1/2}$ denotes the $H^{1/2}$-Hilbert norm
\[
\|x\|_{1/2}^2 := |\hat{x}_0|^2 + 2\pi \sum_{k\in \Z} |k| \, |\hat{x}_k|^2, 
\]
and $P^+$ and $P^-$ are the orthogonal projectors onto the closed subspaces
\[
\begin{split}
\mathbb{H}^+_{1/2} &:= \bigl\{ x\in  \mathbb{H}_{1/2} \mid \hat{x}_k = 0 \; \forall k\leq 0 \bigr\}, \\
\mathbb{H}^-_{1/2} &:= \bigl\{ x\in  \mathbb{H}_{1/2} \mid \hat{x}_k = 0 \; \forall k\geq 0 \bigr\},
\end{split}
\]
defined by the Fourier decomposition (\ref{fourier}).
The Hilbert space $\mathbb{H}_{1/2}$ has the orthogonal splitting
\[
\mathbb{H}_{1/2}= \mathbb{H}^+_{1/2} \oplus \mathbb{H}^-_{1/2} \oplus \R^{2n},
\]
where $\R^{2n}$ denotes the space of constant curves. We also denote by $P^0$ the orthogonal projector onto $\R^{2n}$. It is convenient to fix a splitting of the space of constant curves $\R^{2n}$ 
\[
\R^{2n} = E^+ \oplus E^-
\]
where $E^+$ and $E^-$ are any orthogonal subspaces of dimension $n$.
The fact that the Hessian of the functional
\[
x\mapsto \int_{\T} H_t(x(t))\, dt
\]
is a compact operator on $\mathbb{H}_{1/2}$ implies that the negative eigenspace $V^-(\nabla^2 \Phi_H(x))$ of the Hessian $\nabla^2 \Phi_H(x)$ of $\Phi_H$ at a critical point $x$ is a compact perturbation of the space $\mathbb{H}^-_{1/2}\oplus E^-$, meaning that the difference of the orthogonal projectors onto these subspaces is compact.  Then we define the relative Morse index of $x$ as the
relative dimension of $V^-(\nabla^2 \Phi_H(x))$ with respect to $\mathbb{H}^-_{1/2}\oplus E^-$, that is, the integer
\begin{equation}\label{eq:relative_ind}
\begin{split}
\mathrm{ind}_{\mathbb{H}^-_{1/2} \oplus E^-} (x;\Phi_H) &:= \dim \bigl( V^-(\nabla^2 \Phi_H(x)), \mathbb{H}^-_{1/2} \oplus E^-) \\ &:= \dim V^-(\nabla^2 \Phi_H(x)) \cap (\mathbb{H}^+_{1/2} \oplus E^+) \\
&\quad - \dim \big(V^+(\nabla^2 \Phi_H(x))\oplus\ker \nabla^2 \Phi_H(x)\big) \cap (\mathbb{H}^-_{1/2} \oplus E^-).
\end{split}
\end{equation}
Here $V^+(\nabla^2 \Phi_H(x))$ denotes the positive eigenspace of $\nabla^2\Phi_H(x)$ at $x$. 
An equivalent definition is the following:~any maximal closed subspace $W$ of $\mathbb{H}_{1/2}$ on which the bilinear form $d^2 \Phi_H(x)$ is negative definite forms a Fredholm pair with the subspace $\mathbb{H}^+_{1/2} \oplus E^+$, and the relative Morse index of $x$ is the Fredholm index of this pair:
\[
\mathrm{ind}_{\mathbb{H}^-_{1/2} \oplus E^-} (x;\Phi_H) = \mathrm{ind} (W, \mathbb{H}^+_{1/2} \oplus E^+). 
\]
See \cite[Chapter 2]{abb01} for more details about relative dimensions and relative Morse indices. 
The nullity of $x$ is defined as usual as the dimension of the kernel of the Hessian of $\Phi_H$ at $x$:
\[
\mathrm{null}(x;\Phi_H) := \dim \ker \nabla^2 \Phi_H(x).
\]
We shall make use of the following fact, which is proven in \cite[Corollary 3.3.1]{abb01}.

\begin{prop}
\label{cz=relind}
Assume that $x$ is a 1-periodic orbit of $X_H$. Then the nullity of $x$ coincides with the geometric multiplicity of the eigenvalue 1 of the linearization of the flow along $x$:
\[
\mathrm{null}(x;\Phi_H) = \dim \ker \bigl(I - d\phi_{X_H}^1(x(0)) \bigr),
\]
and the relative Morse index of $x$ coincides with its Conley-Zehnder index:
\[
\mathrm{ind}_{\mathbb{H}^-_{1/2} \oplus E^-} (x;\Phi_H) = \mu_{CZ}(x).
\]
\end{prop}

The Conley-Zehnder index is an integer assigned to every element in the space
\begin{equation}\label{symplectic_path_space}
\mathcal{SP}(2n)=\big\{Z\in C^0([0,1],\mathrm{Sp}(2n))\mid Z(0)=I\textrm{ and } \det(I-Z(1))\neq0\big\}
\end{equation}
which we extend to degenerate paths, i.e.~$\det(I-Z(1))=0$, by lower semi-continuity as in \cite{lz90} and \cite{lon02}.
In the above proposition $\mu_{CZ}(x)$ is the Conley-Zehnder index of the symplectic path $t\mapsto d \phi_{X_H}^t(x(0))$, $t\in [0,1]$. Our sign convention is that the Conley-Zehnder index of the symplectic path
\[
t\mapsto e^{-\epsilon J_0 t}, \quad t\in [0,1],
\]
is $n$ for every $\epsilon\in (0,2\pi)$.

\section{Smooth starshaped domains}

We recall that the characteristic line bundle of a smooth hypersurface $\Sigma\subset \R^{2n}$ is given by the kernel of the restriction of $\omega_0$ to the tangent bundle of $\Sigma$. Its integral lines are called characteristics. The action of a closed characteristic $\gamma$ on $\Sigma$ is defined as the absolute value of the integral of a primitive of $\omega_0$ over $\gamma$. Stokes' theorem implies that this definition does not depend on the choice of the primitive of $\omega_0$. The set of all actions of closed characteristics of $\Sigma$ is called action spectrum of $\Sigma$, or just spectrum of $\Sigma$, and denoted by $\mathrm{spec}(\Sigma)$. Here, iterates of closed characteristics are also considered, so the spectrum of $\Sigma$ is a subset of $[0,+\infty)$ that is invariant under multiplication by positive integers.

We now restrict the attention to those hypersurfaces that are obtained as boundaries of starshaped domains. In this paper, by a smooth starshaped domain we mean a bounded open subset $W$ which is starshaped with respect to the origin and has a smooth boundary which is transverse to all lines through the origin. The restriction of the Liouville 1-form $\lambda_0$ to the boundary of $W$ is denoted by
\[
\alpha_W := \lambda_0|_{\partial W}.
\]
This is a contact form on $\partial W$, meaning that the restriction of the differential $d\alpha_W$ to the kernel of $\alpha_W$ is non-degenerate. The corresponding Reeb vector field on $\partial W$ is denoted by $R_{\alpha_W}$ and is defined by
\[
d\alpha_W(R_{\alpha_W},\cdot) = 0, \qquad \alpha_W(R_{\alpha_W})=1.
\]
This vector field is a smooth non-vanishing section of the characteristic line bundle of the hypersurface $\partial W$. Therefore, the orbits of $R_{\alpha_W}$ are parametrizations of the characteristic curves on $\partial W$. The action of a closed characteristic on $\partial W$ coincides with its period as a closed orbit of $R_{\alpha_W}$:~If $\gamma: \R/T\Z \rightarrow \partial W$ is a closed orbit of $R_{\alpha_W}$ of (not necessarily minimal) period $T$, then
\[
\int_{\R/T\Z} \gamma^* \lambda_0 =  \int_{\R/T\Z} \gamma^* \alpha_W = T.
\]
The spectrum of $\partial W$ is a measure zero nowhere dense closed subset of $\R$ consisting of positive numbers and invariant under the multiplication by positive integers. 
 
We denote by
\[
H_W : \R^{2n} \rightarrow \R
\]
the positively 2-homogeneous function that takes the value 1 on $\partial W$. This function is continuously differentiable on $\R^{2n}$ and smooth on $\R^{2n}\setminus \{0\}$. The restriction of the Hamiltonian vector field $X_{H_W}$ to the boundary of $W$ coincides with the Reeb vector field $R_{\alpha_W}$:
\[
R_{\alpha_W} = X_{H_W}|_{\partial W}.
\]
Indeed, this follows from the fact that for every $x\in \partial W$ the vector $X_{H_W}(x)$ spans $\ker d\alpha_W(x) = \ker \omega_0|_{T_x \partial W}$ and from the identity
\[
\alpha_W(x)[X_{H_W}(x)]  = \lambda_0(x)[X_{H_W}(x)] = \frac{1}{2} \omega_0 (x, X_{H_W}(x)) = \frac{1}{2} dH_W(x)[x] = H_W(x) = 1,
\]
where we have used the Euler identity for the positively 2-homogeneous function $H_W$. 

The symplectization of the contact manifold $(\partial W,\alpha_W)$ is the manifold $(0,+\infty) \times \partial W$ equipped with the Liouville form $\lambda_W:= r\alpha_W$ and the symplectic form $\omega_W := d\lambda_W$, where $r\in (0,+\infty)$ denotes the variable in the first factor. The symplectization of $(\partial W,\alpha_W)$ can be identified with $(\R^{2n}\setminus \{0\},\lambda_0)$ thanks to the following well-known fact:

\begin{lem}
\label{liouville}
The diffeomorphism
\[
\varphi: \R^{2n} \setminus \{0\} \longrightarrow (0,+\infty) \times \partial W, \qquad \varphi(x) = \left( H_W(x), \frac{x}{\sqrt{H_W(x)}} \right),
\]
satisfies $\varphi^* \lambda_W = \lambda_0$. In particular, $\varphi$ is a symplectomorphism from $( \R^{2n} \setminus \{0\} , \omega_0)$ to $( (0,+\infty) \times \partial W, \omega_W)$.
\end{lem}

\begin{proof}
We denote by
\[
\mu_W:\R^{2n} \rightarrow \R, \qquad \mu_W:= \sqrt{H_W},
\]
the Minkowski gauge function of $W$, which is continuous on $\R^{2n}$, smooth on $\R^{2n}\setminus \{0\}$ and positively 1-homogeneous. Let $x\in \R^{2n}\setminus \{0\}$. For every $v\in \ker d\mu_W(x)$ we have
\[
\begin{split}
(\varphi^* \lambda_W)(x)[v] &= \lambda_W(\varphi(x))\bigl[d\varphi(x)[v] \bigr] = H_W(x) \alpha_W \left( \frac{x}{\mu_W(x)} \right) \left[ \frac{v}{\mu_W(x)} \right] \\ &= \frac{H_W(x)}{\mu_W(x)^2} \lambda_0(x)[v] = \lambda_0(x)[v].
\end{split}
\]
There remains to check that the one-forms $(\varphi^* \lambda_W)(x)$ and $\lambda_0(x)$ coincide on a vector which is transverse to $\ker d\mu_W(x)$. The vector $x$ has this property, and we compute
\[
\begin{split}
(\varphi^* \lambda_W)(x)[x] &= \lambda_W(\varphi(x))\bigl[d\varphi(x)[x] \bigr] = H_W(x) \alpha_W \left( \frac{x}{\mu_W(x)} \right) \left[ \frac{x}{\mu_W(x)} - \frac{x}{\mu_W(x)^2} d\mu_W(x)[x] \right] \\ &= H_W(x) \alpha_W \left( \frac{x}{\mu_W(x)} \right) \left[ \frac{x}{\mu_W(x)} - \frac{x}{\mu_W(x)} \right] = 0 = \lambda_0(x)[x],
\end{split}
\]
where we have used the Euler identity for the 1-homogeneous function $\mu_W$.
\end{proof}
 
Any $T$-periodic Reeb orbit of $R_{\alpha_W}$ on $\partial W$ can be seen as a 1-periodic orbit of $X_{T H_W}$, after time reparametrization. More generally, it can be seen as a 1-periodic orbit of $X_{\varphi\circ H_W}$, where $\varphi$ is any smooth function on $\R$ such that $\varphi'(1)=T$. In the next proposition we study how the index and nullity of this orbit change, when it is seen as a critical point of $\Phi_{T H_W}$ or $\Phi_{\varphi\circ H_W}$.

\begin{prop}\label{index1}
Assume that $\gamma:\R/T\Z\to\partial W$ is a periodic orbit of $R_{\alpha_W}$. Let $\varphi:\R\to\R$ be a smooth function with $\varphi'(1)=T$. Then $x_\gamma(t):=\gamma(Tt)$ is a critical point of both $\Phi_{TH_W}$ and $\Phi_{\varphi\circ H_W}$ and
\[
\mathrm{ind}_{\mathbb{H}^-_{1/2} \oplus E^-}(x_\gamma,\Phi_{\varphi\circ H_W})=\left\{\begin{aligned}
&\mathrm{ind}_{\mathbb{H}^-_{1/2} \oplus E^-}(x_\gamma,\Phi_{TH_W}) & \textrm{ if } \varphi''(1)\leq 0\\[0.5ex]
&\mathrm{ind}_{\mathbb{H}^-_{1/2} \oplus E^-}(x_\gamma,\Phi_{TH_W})+1 & \textrm{ if } \varphi''(1)>0
\end{aligned}\right.\;
\]
\[
\mathrm{null}(x_\gamma,\Phi_{\varphi\circ H_W})=\left\{\begin{aligned}
&\mathrm{null}(x_\gamma,\Phi_{TH_W}) & \textrm{ if } \varphi''(1)= 0\\[0.5ex]
&\mathrm{null}(x_\gamma,\Phi_{TH_W})-1 & \textrm{ if } \varphi''(1)\neq0
\end{aligned}\right.\;.
\]
\end{prop}

\begin{proof}
From the hypothesis $\varphi'(H_W(x_\gamma))=\varphi'(1)=T$, we have for every $u\in\mathbb{H}_{1/2}$
\[
d\Phi_{\varphi\circ H_W}(x_\gamma)[u]=\int_{\T}\big(J_0\dot x_\gamma -\varphi'(H_W(x_\gamma))\nabla H_W(x_\gamma)\big)\cdot u\,dt=d\Phi_{TH_W}(x_\gamma)[u].
\]
Since $x_\gamma$ satisfies $J_0\dot x_\gamma = T\nabla H_W(x_\gamma)$, 
the above formula shows that $x_\gamma$ is a critical point of both $\Phi_{TH_W}$ and $\Phi_{\varphi\circ H_W}$.

Consider the following continuous symmetric bilinear forms on $\mathbb{H}_{1/2}$:
\[
\begin{split}
a(u,v) &:= \int_\T u\cdot\big(J_0\dot v- T\nabla^2H_W(x_\gamma)v\big)\, dt, \\
b(u,v) &:= \int_\T \big(\nabla H_W(x_\gamma)\cdot u\big)(\nabla H_W(x_\gamma)\cdot v\big)\, dt.
\end{split}
\]
The bilinear form $a$ is Fredholm, meaning that the corresponding self-adjoint operator representing it with respect to the scalar product of $\mathbb{H}_{1/2}$ is Fredholm (see e.g.\ \cite[Section 3.2]{abb01}). Then we have
\begin{equation}
\label{2nddiffs}
d^2\Phi_{TH_W}(x_\gamma) = a, \qquad d^2\Phi_{\varphi\circ  H_W}(x_\gamma) = a - \varphi''(1) b.
\end{equation}
From the 2-homogeneity of $H_W$ we deduce the identity
\[
\nabla^2 H_W (z) z = \nabla H_W(z) \qquad \forall z\in \R^{2n} \setminus \{0\},
\]
which implies that $x_\gamma$ belongs to the kernel of $a$. This reflects the isochronicity property of 2-homogeneous Hamiltonians, namely the fact that the Hamiltonian flows on their different energy levels are conjugated. The kernel of $b$ is the $L^2$-orthogonal complement of the line $\R \nabla H_W(x_\gamma)$ in $\mathbb{H}_{1/2}$. Since
\[
b(x_\gamma,x_\gamma) = \int_\T \big(\nabla H_W(x_\gamma)\cdot x_\gamma\big)^2\, dt = \int_\T(2H_W(x_\gamma))^2\,dt=4 >0
\]
by the Euler identity, we have
\[
\mathbb{H}_{1/2} = \ker b \oplus \R x_\gamma.
\]
Let $s\neq 0$. A vector $u=v+\lambda x_\gamma$, $v\in \ker b$, $\lambda\in \R$, belongs to the kernel of $a+sb$ if and only if
\[
(a+sb)(u,x_\gamma) = 0 \qquad \mbox{and} \qquad (a+sb)(u,w)=0 \quad \forall w\in \ker b.
\]
The first identity is equivalent to
\[
0 = (a+sb)(v+\lambda x_\gamma,x_\gamma) = s \lambda\, b(x_\gamma,x_\gamma),
\]
and hence to $\lambda=0$. The second identity then reads
\[
0 = (a+sb)(v,w) = a(v,w) \quad \forall w\in \ker b.
\]
Since $x_\gamma$ belongs to the kernel of $a$, the latter requirement is equivalent to the fact that $v$ belongs to the kernel of $a$. We conclude that for every real number $s\neq 0$ the kernel of $a+sb$ is the following space
\[
\ker (a+sb) = \ker a \cap \ker b \qquad \forall s\in \R \setminus \{0\},
\]
which is independent of $s$ and has codimension 1 in $\ker a$. The formula for the nullity of $d^2 \Phi_{\varphi\circ H_W}(x_\gamma)$ immediately follows from this and (\ref{2nddiffs}).

The path $s\mapsto a+sb$ describes a 1-parameter family of continuous symmetric bilinear forms that are rank one perturbations of the Fredholm form $a$. The fact that the nullity of $a+sb$ is constant for $s\neq 0$ implies that the relative index of $a+sb$ with respect to $\mathbb{H}_{1/2}^-\oplus E^-$ is constant for $s>0$ and for $s<0$. Since the kernel of $a+sb$ increases by one dimension - and precisely by addition of the line $\R x_\gamma$ - when $s$ becomes 0, the inequality
\[
\frac{d}{ds}\Big|_{s=0} (a+sb)(x_\gamma,x_\gamma) = b(x_\gamma,x_\gamma) > 0
\]
implies the identities
\[
\dim \bigl( V^-(a+sb), \mathbb{H}^-_{1/2} \oplus E^-) = 
\left\{\begin{aligned}  
& \dim \bigl( V^-(a), \mathbb{H}^-_{1/2} \oplus E^-) \qquad & \forall s\geq 0, \\[.5ex]
& \dim \bigl( V^-(a), \mathbb{H}^-_{1/2} \oplus E^-)+ 1\qquad & \forall s<0.
\end{aligned}
\right.
\]
The formula for the relative index of $x_\gamma$ as a critical point of $\Phi_{\varphi\circ H_W}$ immediately follows from this and (\ref{2nddiffs}).
\end{proof}

\section{Uniform bounds for solutions of the Floer equation}
\label{secunifbou}

In this section, we wish to prove uniform bounds for solutions
\[
u: I \times \T \rightarrow \R^{2n}
\]
of the Floer equation
\begin{equation}
\label{floer}
\partial_s u + J(s,t,u) \bigl( \partial_t u - X_{H_t}(u) \bigr) = 0.
\end{equation}
Here $I$ is an open unbounded interval, i.e.\ either $\R$, or $(a,+\infty)$, or $(-\infty,a)$ for some $a\in \R$. The symbol $J$ denotes a smooth almost complex structure on $\R^{2n}$, which is allowed to depend on the variables $(s,t)\in I\times \T$ and is supposed to be $\omega_0$-compatible, meaning that the formula
\[
g_J(u,v) := \omega_0(Ju,v)
\]
defines an $(s,t)$-dependent family of Riemannian metrics on $\R^{2n}$. The corresponding family of norms will be denoted by $|\cdot |_J$, and $\nabla_J$ will denote the gradient operator with respect to these Riemannian metrics. By our sign conventions we have
\[
X_H = - J \nabla_J H.
\]
The Floer equation can be rewritten as
\[
\overline{\partial}_J u = \nabla_J H(u), \qquad \mbox{where}\quad \overline{\partial}_J u := \partial_s u + J(u) \partial_t u.
\]
We shall also assume that $J$ is uniformly bounded as a map with values in the space of endomorphisms of $\R^{2n}$. It follows that the family of metrics $g_J$ is uniformly globally equivalent to the Euclidean metric $g_{J_0}(u,v) = u\cdot v$. 

We will consider the following growth assumptions on Hamiltonian functions $H\in C^{\infty}(\T\times \R^{2n})$.

\begin{description}
\item[Linear growth of the Hamiltonian vector field:] The Hamiltonian vector field $X_H$ is said to have linear growth at infinity if there exists a positive number $c$ such that $|X_{H_t}(z)| \leq c  (1+|z|)$ for every $(t,z)\in \T\times \R^{2n}$.
\item[Non-resonance at infinity:] The Hamiltonian $H$ is said to be non-resonant at infinity if there exist positive numbers $\epsilon>0$ and $r>0$ such that for every smooth curve $x:\T \rightarrow \R^{2n}$ satisfying
\[
\|\dot{x}-X_H(x)\|_{L^2(\T)} \leq \epsilon,
\]
there holds $\|x\|_{L^2(\T)}\leq r$.
\end{description}

The latter requirement is a version of the Palais-Smale condition for the direct action functional $\Phi_H$: It is equivalent to saying that every sequence $(x_h)\subset C^{\infty}(\T,\R^{2n})$ on which the $L^2$-gradient of $\Phi_H$ tends to zero in the $L^2$-norm is bounded in $L^2(\T,\R^{2n})$. The above two assumptions imply in particular that 1-periodic orbits of $X_{H}$ are uniformly bounded. At the end of the next section, we shall discuss some sufficient conditions for $H$ to be non-resonant at infinity in the above sense. Here, we shall prove the following a priori bounds.

\begin{prop}
\label{linfbound} 
Let $I\subset \R$ be an open unbounded interval, $J$ be a uniformly bounded $\omega_0$-compatible almost complex structure on $\R^{2n}$, smoothly depending on $(s,t)\in I\times \T$, and $H\in C^{\infty}(\T\times \R^{2n})$ be a smooth Hamiltonian which is non-resonant at infinity and whose Hamiltonian vector field has linear growth  at infinity.
Let $I'$ be a (possibly unbounded) interval such that $\overline{I'}\subset I$. For every $E>0$ there is a positive number $M=M(E)$ such that every solution $u\in C^{\infty}(I\times \T,\R^{2n})$ of (\ref{floer}) with energy bound
\[
\int_{I\times \T} |\partial_s u|_J^2\, ds dt \leq E
\]
satisfies
\[
\sup_{(s,t)\in I'\times \T} |u(s,t)| \leq M.
\]
\end{prop}

When $I=\R$, we are allowed to take $I'=\R$ and we obtain uniform bounds on the whole cylinder. The above formulation allows us to get uniform bounds also for solutions on half-cylinders, as long as we stay away from the boundary.

\begin{proof}
Without loss of generality, we may assume that $I$ is unbounded from below, the case in which it is unbounded from above being completely analogous.
Since the family of metrics $g_J$ is uniformly globally equivalent to the Euclidean metric, the energy bound on $u$ translates into
\begin{equation}
\label{eebb}
\int_{I\times \T} |\partial_s u|^2\, ds dt \leq E',
\end{equation}
for a suitable number $E'=E'(E)$. Consider the set
\[
S = S(u) = \{ s\in \R \mid \|\partial_s u(s,\cdot)\|_{L^2(\T)} < \epsilon/\|J\|_{\infty}\},
\]
where $\epsilon$ is the positive number appearing in the assumption of non-resonance at infinity and the $L^{\infty}$-norm of $J$ is induced by the Euclidean metric on $\R^{2n}$. 
By the Chebyshev inequality, the complement of $S$ in $I$ has uniformly bounded measure:
\begin{equation}
\label{cheb}
|I \setminus S| \leq \frac{\|J\|_{\infty}^2}{\epsilon^2} \int_{I} \|\partial_s u(s,\cdot)\|_{L^2(\T)}^2\, ds =  \frac{\|J\|_{\infty}^2}{\epsilon^2}  \int_{I \times \T} |\partial_s u|^2 \, ds dt < L,\end{equation}
where we have set 
\[
L=L(E):= \frac{\|J\|_{\infty}^2 E'}{\epsilon^2}+1.
\]
Let $I''\subset I$ be an interval of length $L$. By (\ref{cheb}), $I''$ intersects $S$ in at least one point $s_0$. For such a point $s_0$ we have, using the fact that $u$ solves the Floer equation (\ref{floer}):
\[
\begin{split}
\|\partial_t u(s_0,\cdot) - X_H(u(s_0,\cdot))\|_{L^2(\T)} &= \|J(s_0,\cdot,u(s_0,\cdot)) \partial_s u(s_0,\cdot)\|_{L^2(\T)} \\ &\leq \|J\|_{\infty} \| \partial_s u(s_0,\cdot)\|_{L^2(\T)} 
< \epsilon,
\end{split}
\]
and the non-resonance at infinity implies the $L^2$-bound
\[
\|u(s_0,\cdot)\|_{L^2(\T)} \leq r.
\]
Together with the energy bound on $u$, we easily get a uniform $L^2$-bound for $u$ on $I''\times \T$. Indeed, from the identity
\[
u(s,t) = u(s_0,t) + \int_{s_0}^s \partial_s u(\sigma,t)\, d\sigma 
\]
and the bound (\ref{eebb}), we get the inequality
\[
\begin{split}
|u(s,t)|^2 &\leq 2 |u(s_0,t)|^2 + 2 \left| \int_{s_0}^s \partial_s u(\sigma,t)\, d\sigma \right|^2 \leq 2  |u(s_0,t)|^2 + 2|s-s_0| \left| \int_{s_0}^s |\partial_s u(\sigma,t)|^2\, d\sigma \right|  \\ &\leq 2 |u(s_0,t)|^2 + 2 L E'.
\end{split}
\]
for every $s\in I''$. Integration over $\T$ yields
\[
\int_{\T} |u(s,t)|^2\, dt \leq 2 \int_{\T} |u(s_0,t)|^2 \, dt + 2 L E' \leq  2r^2 + 2 L E'=: C, \qquad \forall s\in I'',
\]
and hence
\[
\int_{I'' \times \T} |u(s,t)|^2\, ds dt \leq L C.
\]
This inequality holds for every interval $I''\subset I$ of length $L$. It follows that 
\begin{equation}
\label{L2}
\int_{I'' \times \T} |u(s,t)|^2\, ds dt \leq C( |I''| + L),
\end{equation}
for any bounded interval $I''\subset I$. Together with the assumption of linear growth at infinity for $X_H=-J\nabla_J H$ and the boundedness of $J$, we get the bound
\begin{equation}
\label{L3}
\int_{I'' \times \T} |\nabla_J H_t(u(s,t))|_J^2 \, ds dt \leq C'( |I''| + L),
\end{equation}
for a suitable positive number $C'$. Now let $I'$ be the possibly unbounded interval such that $\overline{I'}\subset I$ which appears in the statement. Fix a positive number $\delta$ such that 
\[
I'+[-2\delta,2\delta]\subset I.
\]
In order to prove a uniform bound for $|u|$ on $I'\times \T$ it is enough to prove such a uniform bound on $I''\times \T$ for all intervals $I''\subset I'$ of length 1. Let $I''$ be such an interval.

By the Calderon-Zygmund estimates in $L^2$ for the uniformly bounded almost complex structure $J$ (see e.g.~\cite{ms04}[Proposition B.4.9]), we have
\[
\begin{split}
\|\nabla u\|_{L^2((I''+ [-\delta,\delta])\times \T)} &\leq c_2 \bigl( \|\overline{\partial}_J u\|_{L^2((I''+[-2\delta,2\delta])\times \T)} + \|u\|_{L^2((I''+[-2\delta,2\delta])\times \T)} \bigr) \\ &= c_2 \bigl( \|\nabla_J H(u)\|_{L^2((I''+[-2\delta,2\delta])\times \T)} + \|u\|_{L^2((I''+[-2\delta,2\delta])\times \T)} \bigr).
\end{split}
\]
Together with the estimates (\ref{L2}) and (\ref{L3}) we obtain a bound
\[
\|\nabla u\|_{L^2((I''+ [-\delta,\delta])\times \T)} \leq C''',
\]
holding for all intervals $I''\subset I'$ of length 1. Using again (\ref{L2}), we deduce that the restriction of $u$ to $(I''+ [-\delta,\delta])\times \T$ is uniformly bounded in the Sobolev space $H^1$, and hence in all Lebesgue spaces $L^p$ for $p<+\infty$. Using again the linear growth at infinity of $X_H=-J\nabla_J H$ we deduce that the function 
\[
\overline{\partial}_J u = \nabla_J H_t(u)
\]
has a uniform $L^p$ bound on $(I''+ [-\delta,\delta])\times \T$. By the Calderon-Zygmund estimate in $L^p$ (see again \cite{ms04}[Proposition B.4.9])
\[
\|\nabla u\|_{L^p(I''\times \T)} \leq c_p \bigl( \|\overline{\partial}_J u\|_{L^p((I''+[-\delta,\delta])\times \T)} + \|u\|_{L^p((I''+[-\delta,\delta])\times \T)} \bigr),
\]
we conclude that the restriction of $u$ to $I'' \times \T$ has a uniform $W^{1,p}$ bound for every $p<\infty$. For $p>2$ this bound implies the desired uniform $L^{\infty}$ bound on the restriction of $u$ to $I'' \times \T$, where $I''\subset I'$ is an arbitrary interval of length 1.
\end{proof}

\begin{rem}
\label{sdep}
A straightforward modification of the above argument allows one to extend the above result to $s$-dependent Hamiltonians. The precise assumptions are that $H\in C^{\infty}(I\times \T \times \R^{2n})$ depends on $s$ only for $s$ in a bounded interval, that for $s$ outside of this interval $H(s,\cdot)$ is non-resonant at infinity, and that the Hamiltonian vector field of $H$ has linear growth at infinity, uniformly on $s\in I$.
\end{rem}

\section{The Floer complex of $H$}

\label{floer-complex}

Assume that the Hamiltonian $H\in C^{\infty}(\T\times \R^{2n})$ is non-degenerate, non-resonant at infinity and that the corresponding Hamiltonian vector field $X_H$ has linear growth at infinity. Such a Hamiltonian has a well-defined Floer complex, whose construction we quickly recall here.

Let $J=J(t,z)$ be a time-periodic smooth almost complex structure on $\R^{2n}$, that is compatible with $\omega_0$ and uniformly bounded. By Proposition \ref{linfbound}, energy bounds on solutions of the Floer equation (\ref{floer}) on the whole cylinder $\R \times \T$ imply $L^{\infty}$-bounds. Once uniform bounds in $L^{\infty}$ have been established, the standard bubbling-off argument and an elliptic bootstrap imply that solutions with uniformly bounded energy are compact in $C^{\infty}_{\mathrm{loc}}(\R\times \T,\R^{2n})$. Actually, the fact that we are working in $\R^{2n}$ would allow us to prove the above results using only a bootstrap argument involving the Calderon-Zygmund inequalities, as in the last part of the proof of Proposition \ref{linfbound}, avoiding the bubbling-off argument.

Let $x$ and $y$ be two 1-periodic orbits of $X_H$ with $\mu_{CZ}(x)-\mu_{CZ}(y)=1$. Standard index and transversality arguments imply that, if $J$ is chosen generically, the space of solutions of the Floer equation (\ref{floer}) that for $s\rightarrow -\infty$ are asymptotic to $x$ and for $s\rightarrow +\infty$ to $y$ is finite, after modding out translation in the $s$-variables. Let $n^F(x,y)\in \Z_2$ denote the parity of this finite set. The Floer complex is then defined as usual by
\[
\partial^F : F_{k}(H) \rightarrow F_{k-1}(H), \qquad \partial^F x := \sum_{y} n^F(x,y) \, y,
\]
where $F_k(H)$ denotes the $\Z_2$-vector space generated by the 1-periodic orbits of $X_H$ with Conley-Zehnder index $k$, $x$ is any 1-periodic orbit with $\mu_{CZ}(x)=k$, and the sum ranges over all 1-periodic orbits $y$ of Conley-Zehnder index $k-1$. A standard cobordism argument implies that $\partial^F\circ \partial^F=0$, so $\{F_*(H),\partial^F\}$ is a chain complex of finite dimensional $\Z_2$-vector spaces.

The fact that $n^F(x,y)=0$ whenever $\Phi_H(y)>\Phi_H(x)$ implies that the Floer complex of $H$ is graded by the Hamiltonian action: If $F^{<a}_k(H)$ denotes the subspace of $F_k(H)$ that is generated by the 1-periodic orbits with $\Phi_H(x)<a$, then $\partial^F$ maps $F_k^{<a}(H)$ into $F^{<a}_{k-1}(H)$. For real numbers $b>a$, there is a canonical inclusion $F^{<a}_*(H)\hookrightarrow F^{<b}_*(H)$ and it is a chain map. We denote the homology of the subcomplex $\{F^{<a}_*(H),\partial^F\}$ by
\[
HF_*^{<a}(H).
\]
We simply write $HF_*(H)$ when $a=+\infty$. The choice of a different generic almost complex structure $J$ produces chain isomorphic Floer complexes. 

We show in this section that a typical growth condition on the Hamiltonian $H$ - which is considered for instance in Viterbo's definition of symplectic homology of a starshaped domain $W\subset \R^{2n}$, see \cite{vit99} - implies the non-resonance at infinity condition considered here. Below we denote by $\|\cdot\|_{\mathbb{H}_{1/2}^*}$ the dual norm of $\|\cdot\|_{1/2}$.

\begin{lem}
\label{cresce}
Let $W\subset \R^{2n}$ be a smooth starshaped domain. Assume that the Hamiltonian $H: \T \times \R^{2n} \rightarrow \R$ is smooth and satisfies
\begin{equation}
\label{slope}
H(t,z) = \eta H_W(z) + \xi, \qquad \forall t\in \T, \; \forall |z|\geq R, 
\end{equation}
where $\eta$ is a positive number that does not belong to $\mathrm{spec}(\partial W)$, $\xi$ is any real number and $R>0$. Then there exist positive numbers $a$ and $b$ such that
\[
\bigl\|d\Phi_H(x)\bigr\|_{\mathbb{H}_{1/2}^*} \geq a \|x\|_{1/2} - b,
\]
for every $x\in \mathbb{H}_{1/2}$. In particular, there are positive numbers $a'$ and $b'$ such that
\[
\bigl\| \dot{x} - X_{H}(x) \bigr\|_{L^2(\T)} \geq a' \|x\|_{L^2(\T)} - b',
\]
for every $x\in C^{\infty}(\T,\R^{2n})$, and hence $H$ is non-resonant at infinity.
\end{lem}

\begin{proof}
The functional $\Phi_{\eta H_W+\xi}$ is continuously differentiable on $\mathbb{H}_{1/2}$. The fact that $\eta$ is not in the spectrum of $\partial W$ implies that $\Phi_{\eta H_W+\xi}$ has a unique critical point at 0. Indeed, a critical point $x$ of this functional is a 1-periodic orbit of the Hamiltonian vector field 
\[
X_{\eta H_W + \xi} = X_{\eta H_W} = \eta X_{H_W},
\]
and hence $\gamma(t):= x(t/\eta)$ is an $\eta$-periodic orbit of $H_W$, which then lies in a level set of $H_W$. If, by contradiction, $x$ is not identically zero, then $H_W(x)=H_W(\gamma)>0$, and hence by rescaling we obtain an $\eta$-periodic orbit of $R_{\alpha_W}=X_{H_W}|_{\partial W}$ on $\partial W$, contradicting the fact that $\eta$ is not a period of closed Reeb orbits on $(\partial W,\alpha_W)$.  

We claim that the number
\[
a:= \inf_{\substack{x\in \mathbb{H}_{1/2} \\ \|x\|_{1/2} = 1}} \bigl\| d\Phi_{\eta H_W+\xi}(x)\bigr\|_{\mathbb{H}_{1/2}^*}
\]
is positive. If by contradiction $a$ is zero, then there exists a sequence $(x_h)\subset \mathbb{H}_{1/2}$ such that $\|x_h\|_{1/2}=1$ and $d\Phi_{\eta H_W+\xi}(x_h)\rightarrow 0$ in $\mathbb{H}_{1/2}^*$. Up to considering a subsequence, we may assume that $(x_h)$ converges weakly to some $x\in \mathbb{H}_{1/2}$. The sequence
\[
v_h:= P^+(x_h - x) - P^-(x_h-x)
\]
is bounded, and hence the real sequence
\[
d\Phi_{\eta H_W+\xi}(x_h)[v_h] 
\]
is infinitesimal. This sequence has the form
\[
\begin{split}
d\Phi_{\eta H_W +\xi}&(x_h)[v_h] = (x_h,P^+ v_h)_{1/2} - (x_h,P^-v_h)_{1/2} - \eta \int_{\T} \nabla H_W(x_h)\cdot v_h\, dt \\ &= \|x_h - x\|_{1/2}^2+(x,x_h-x)_{1/2} - (x_h, P^0 (x_h-x))_{1/2}  - \eta \int_{\T} \nabla H_W(x_h)\cdot v_h\, dt,
\end{split}
\]
where $P^0: \mathbb{H}_{1/2} \rightarrow  \mathbb{H}_{1/2}$ is the orthogonal projector onto the space of constant loops.  Since $x_h$ converges to $x$ weakly in $\mathbb{H}_{1/2}$, the term $(x,x_h-x)_{1/2}$ is  infinitesimal. Moreover since $P^0$ has finite rank, the sequence $P^0 (x_h-x)$ converges to zero strongly and hence the sequence $(x_h, P^0 (x_h-x))_{1/2}$ is also infinitesimal. Since $(v_h)$ converges to 0 weakly in $\mathbb{H}_{1/2}$, it converges to 0 strongly in $L^2(\T)$. Since the sequence $\nabla H_W(x_h)$ is bounded in $L^2$, we deduce that the last integral defines an infinitesimal sequence. We conclude that the sequence $\|x_h - x\|_{1/2}^2$ is also infinitesimal, so the convergence of $(x_h)$ to $x$ is actually strong in $\mathbb{H}_{1/2}$. But then $x$ is a critical point of $\Phi_{\eta H_W+\xi}$ lying on the unit sphere and this is impossible, because 0 is the only critical point of $\Phi_{\eta H_W+\xi}$. This contradiction proves the claim.

The above claim and the 1-homogeneity of $d\Phi_{\eta H_W+\xi}= d\Phi_{\eta H_W}$ imply that
\begin{equation}
\label{bdhom}
\bigl\| d\Phi_{\eta H_W+\xi}(x)\bigr\|_{\mathbb{H}_{1/2}^*} \geq a \|x\|_{1/2} \qquad \forall x\in \mathbb{H}_{1/2}.
\end{equation}
The continuously differentiable function
\[
K: \T \times \R^{2n} \rightarrow \R, \qquad
K_t(z):= \eta H_W(z) + \xi - H_t(z),
\]
is compactly supported and hence has uniformly bounded first derivatives. From the identity
\[
\Phi_H(x) - \Phi_{\eta H_W+\xi}(x) = \int_{\T} ( \eta H_W(x) + \xi - H_t(x) )\, dt = \int_{\T} K_t(x)\, dt
\]
we deduce that
\[
d\Phi_H(x)[u] - d\Phi_{\eta H_W+\xi}(x)[u] = \int_{\T} dK_t(x)[u]\, dt \qquad \forall x,u\in \mathbb{H}_{1/2},
\]
and hence there is a number $b$ such that
\[
\bigl\|d\Phi_H(x) - d\Phi_{\eta H_W+\xi}(x)\bigr\|_{\mathbb{H}_{1/2}^*} \leq b \qquad \forall x\in \mathbb{H}_{1/2}.
\]
Together with (\ref{bdhom}) we deduce the desired first bound.

From the identity
\[
d\Phi_H(x)[v] = \int_{\T} (J_0 \dot{x} - \nabla H(x)) \cdot v\, dt =  \big(J_0(\dot{x} - X_H(x)), v\big)_{L^2(\T)} \qquad \forall x,v\in  C^{\infty}(\T,\R^{2n}),
\]
we deduce that
\[
\|d\Phi_H(x)\|_{\mathbb{H}_{1/2}^*} = \|J_0(\dot{x} - X_{H}(x))\|_{H^{-1/2}(\T)} \leq c \|\dot{x} - X_{H}(x)\|_{L^2(\T)} \qquad \forall x\in  C^{\infty}(\T,\R^{2n}),
\]
and the second bound follows from the first one and from the inequality $\|x\|_{L^2(\T)} \leq \|x\|_{1/2}$.
 \end{proof}

\begin{rem}
The condition on $H$ that appears in the above lemma clearly implies also the linear growth condition on $X_H$, so a Hamiltonian of this kind has a well-defined Floer complex. Another class of Hamiltonians that are non-resonant at infinity and have a Hamiltonian vector field with linear growth is given by functions of the form
\[
H(t,x) = \frac{1}{2} \langle A(t)x,x\rangle + K(t,x),
\]
where the function $K\in C^{\infty}(\T\times \R^{2n})$ satisfies  $\nabla K(t,x)= o(\|x\|)$ for $\|x\|\rightarrow \infty$ uniformly in $t\in \T$ and $t\mapsto A(t)$ is a smooth  loop of symmetric endomorphisms of $\R^{2n}$ such that the linear Hamiltonian system
\[
\dot{x}(t) = - J_0 A(t) x(t)
\]
does not have any non-zero 1-periodic orbit.
\end{rem}

\section{The dual action functional}
\label{dualsec}

We now focus our attention on convex Hamiltonians. More precisely, we consider the following convexity assumption.

\begin{description}
\item[Quadratic convexity:] The Hamiltonian $H\in C^{\infty}(\T \times \R^{2n})$ is said to be quadratically convex if there are positive numbers $\underline{h}$ and $\overline{h}$ such that
\[
\underline{h}\,  I \leq \nabla^2 H_t(x)  \leq \overline{h}\, I
\]
for all $x\in \R^{2n}$ and $t\in \T$. Here, $I$ denotes the identity endomorphism on $\R^{2n}$.
\end{description}
Note that the upper bound on the Hessian of $H$ implies that the Hamiltonian vector field $X_H$ is globally Lipschitz-continuous in the space variables, 
\begin{equation}
\label{globlip}
|X_{H_t}(y)-X_{H_t}(x)| \leq \overline{h} |y-x| \qquad \forall t\in \T, \; \forall x,y\in \R^{2n},
\end{equation}
and in particular $X_H$ has linear growth at infinity.

Let $C\subset \R^{2n}$ be a bounded convex open set with smooth boundary. Assume moreover that all the sectional curvatures of $\partial C$ are positive. This is equivalent to the fact that, after shifting $C$ so that the origin belongs to its interior, the second differential of the positively 2-homogeneous function $H_C$ at every point in $\R^{2n}\setminus \{0\}$ is positive definite. We shall refer to such a set as a smooth strongly convex domain. From the compactness of $\partial C$ and from the 2-homogeneity of $H_C$ we deduce the bounds
\[
\underline{c}\, I \leq \nabla^2 H_C(x)  \leq \overline{c}\, I
\]
for all $x\in \R^{2n}\setminus \{0\}$, for suitable positive numbers $\underline{c}$ and $\overline{c}$.  Therefore, the square $H_C$ of the Minkowski gauge function of the strongly convex domain $C$ is quadratically convex, except for the lack of differentiability at the origin. 

Throughout this section, the Hamiltonian $H: \T \times \R^{2n} \rightarrow \R$ is assumed to be smooth and quadratically convex. We denote by
\[
H_t^* : \R^{2n} \rightarrow \R, \qquad H_t^*(x) := \max_{y\in \R^{2n}} \bigl( x\cdot y - H(t,y) \bigr),
\]
the Fenchel conjugate of $H_t$. The properties of Fenchel duality, see e.g.~\cite[\S 26]{Roc70}, imply that the function $H^*: \T \times \R^{2n}\rightarrow \R$ is smooth and satisfies
\[
\overline{h}^{-1} I \leq \nabla^2 H_t^*(x) \leq \underline{h}^{-1} I \qquad \forall (t,x)\in\T\times \R^{2n}.
\]
By the definition of $H^*_t$, we have for every $t\in \T$
\begin{equation}\label{eq:Fenchel_ineq}
H_t^*(x)+H_t(y)\geq x\cdot y \qquad \forall x,y\in\R^{2n}	
\end{equation}
where the equality holds if and only if $x=\nabla H_t(y)$, or equivalently $y=\nabla H_t^*(x)$. The equality case is called the Legendre reciprocity formula.

Clarke's dual action functional is defined by the formula
\[
\Psi_{H^*}(x) := - \frac{1}{2} \int_{\T} J_0\dot{x}(t)\cdot x(t) \, dt + \int_{\T} H_t^*\bigl(J_0 \dot{x}(t)\bigr)\, dt.
\]
The functional $\Psi_{H^*}$ is continuously differentiable on the Hilbert space
\[
\mathbb{H}_1 := H^1(\T,\R^{2n})/\R^{2n},
\]
where the action of $\R^{2n}$ on the Sobolev space $H^1(\T,\R^{2n})$ is given by translations. Rather than working with equivalence classes of curves modulo translations, it is convenient to work with genuine curves by identifying  $\mathbb{H}_1$ with the space of closed curves with zero mean:
\[
\mathbb{H}_1 = \Big\{ x\in H^1(\T,\R^{2n}) \mid \int_{\T} x(t)\, dt = 0 \Big\}.
\]
On this space, it is convenient to use the inner product 
\[
(x,y)_{\mathbb H_1}:=(\dot x,\dot y)_{L^2}.
\]
By the Poincar\'e-Wirtinger inequality, the induced norm $\|\cdot\|_{\mathbb{H}_1}$ is equivalent to the $H^1$-norm on $\mathbb{H}_1$. We denote by 
\[
\pi : H^1(\T,\R^{2n}) \longrightarrow \mathbb{H}_1, \qquad \pi(x) = x - \int_{\T} x(t)\, dt
\]
the quotient projection.
The differential of $\Psi_{H^*}$ is given by the formula
\begin{equation}\label{eq:Psi'}
\begin{split}
d\Psi_{H^*}(x)[v] &= - \int_{\T} J_0 \dot{v}(t) \cdot x(t)\, dt + \int_{\T} dH_t^*\bigl(J_0 \dot{x}(t)\bigr)[J_0\dot{v}(t)]\, dt \\ &= - \int_{\T} J_0 \dot{v}(t) \cdot x(t)\, dt + \int_{\T} J_0 \dot{v}(t)\cdot \nabla H_t^*\bigl(J_0 \dot{x}(t)\bigr) \, dt.
\end{split}
\end{equation}
Since $\nabla H_t^*$ is Lipschitz continuous, so is the differential $d\Psi_{H^*}$.

There is a one-to-one correspondence between the critical points of $\Phi_H$ and $\Psi_{H^*}$. More precisely, we have the following well-known fact, of which we include a proof for sake of completeness.

\begin{lem}
\label{crit}
If $x$ is a critical point of $\Phi_H$, then $\pi(x)$ is a critical point of $\Psi_{H^*}$. Conversely, every critical point $x$ of $\Psi_{H^*}$ is smooth and there exists a unique vector $v_0\in \R^{2n}$ such that $x+v_0$ is a critical point of $\Phi_H$. In this case, we have
\[
\Phi_H(x+v_0) = \Psi_{H^*}(x).
\]
\end{lem}

\begin{proof}
Assume that $x$ is a critical point of $\Phi_H$. Then $x$ is smooth and is a 1-periodic orbit of $X_H$, that is,
\[
J_0 \dot{x}(t) = \nabla H_t(x(t)) \qquad \forall t\in \T.
\]
Fenchel duality implies that for every $t\in \T$ the map $\nabla H_t^*$ is the inverse of the map $\nabla H_t$. Therefore, the above identity implies
\[
\nabla H_t^*\bigl(  J_0 \dot{x}(t) \bigr) = x(t) \qquad \forall t\in \T.
\]
Then $y:= \pi(x)$ satisfies
\[
\nabla H_t^*\bigl(  J_0 \dot{y}(t) \bigr) = y(t) + \hat{x}_0 \qquad \forall t\in \T,
\]
where $\hat{x}_0\in \R^{2n}$ denotes the average of $x$.
For every $v\in \mathbb{H}_1$ we have
\[
d\Psi_{H^*}(y)[v] = - \int_{\T} J_0 \dot{v}(t) \cdot y(t)\, dt + \int_{\T} J_0 \dot{v}(t)\cdot \nabla H_t^*\bigl(J_0 \dot{y}(t)\bigr) \, dt =  \int_{\T} J_0 \dot{v}(t)\cdot \hat{x}_0\, dt = 0,
\]
so $\pi(x)$ is a critical point of $\Psi_{H^*}$.

Now assume that $x\in \mathbb{H}_1$ is a critical point of $\Psi_{H^*}$. Then
\[
\int_{\T} J_0 \dot{v}(t) \cdot \Bigl( x(t) - \nabla H_t^*\bigl(J_0 \dot{x}(t)\bigr) \Bigr)\, dt = 0
\]
for every $v\in \mathbb{H}_1$, and hence there exists $v_0\in \R^{2n}$ such that
\[
\nabla H_t^*\bigl(J_0 \dot{x}(t)\bigr) - x(t)= v_0 \qquad \mbox{for a.e. } t\in \T,
\]
that is,
\[
\nabla H_t^*\bigl(J_0 \dot{x}(t)\bigr) = x(t) + v_0  \qquad \mbox{for a.e. } t\in \T.
\]
By applying the map $\nabla H_t$ to the above identity we obtain
\[
J_0 \dot{x}(t) = \nabla H_t (x(t)+v_0) \qquad \mbox{for a.e. } t\in \T.
\]
The above identity and a standard bootstrap argument imply that the curve $x$ is smooth and $x+v_0$ is a 1-periodic orbit of $X_H$, and hence a critical point of $\Phi_H$. 

The fact that $x(t)+v_0$ is the image of $J_0\dot{x}(t)$ by $\nabla H^*_t$ implies that
\[
H_t(x(t)+v_0) + H^*_t(J_0\dot{x}(t)) = J_0 \dot{x}(t) \cdot (x(t)+v_0) \qquad \forall t\in \T
\]
by \eqref{eq:Fenchel_ineq}, and integrating over $\T$ we find
\[
\int_{\T} H_t(x(t)+v_0)\, dt + \int_{\T} H^*_t(J_0\dot{x}(t))\, dt = \int_{\T} J_0 \dot{x}(t)\cdot x(t)\, dt.
\]
We conclude that
\[
\Psi_{H^*}(x) - \Phi_H(x+v_0) = - \int_{\T} J_0 \dot{x}(t)\cdot x(t)\, dt + \int_{\T} H^*_t(J_0\dot{x}(t))\, dt + \int_{\T} H_t(x(t)+v_0)\, dt = 0.
\]
\end{proof}

The following result plays a fundamental role in the construction of the isomorphism between the Morse complex induced by $\Psi_{H^*}$ and the Floer complex of $\Phi_H$.

\begin{prop}
\label{confronto}
Let $x\in H^1(\T,\R^{2n})$ and $y\in  \R^{2n}\oplus\mathbb{H}^-_{1/2}$. Then we have
\begin{equation}
\label{disdadim}
\Phi_H(x+y) \leq \Psi_{H^*}(\pi(x)) - \frac{1}{2} \|P^-y\|_{1/2}^2,
\end{equation}
with the equality holding if and only if $J_0\dot{x} = \nabla H_t(x+y)$ almost everywhere. In particular, the equality
\[
\Phi_H(x+y) = \Psi_{H^*}(\pi(x))
\]
holds if and only if $y\in \R^{2n}$ and $x+y$ is a critical point of $\Phi_H$.
\end{prop}

\begin{proof}
Fenchel duality \eqref{eq:Fenchel_ineq} implies that
\begin{equation}
\label{fenin}
H_t(x+y) \geq J_0\dot{x} \cdot (x+y) - H_t^*(J_0 \dot{x}) \qquad \mbox{a.e.}
\end{equation}
with equality if and only if
\begin{equation}
\label{eqcase}
J_0\dot{x} = \nabla H_t(x+y) \qquad \mbox{a.e.}
\end{equation}
By integration we get
\[
\begin{split}
\Phi_H(x+y) &= \frac{1}{2} \int_{\T} J_0 (\dot{x} + \dot{y})\cdot (x+y)\, dt - \int_{\T} H_t(x+y)\, dt \\ &\leq    \frac{1}{2} \int_{\T} J_0 (\dot{x} + \dot{y})\cdot (x+y)\, dt - \int_{\T} J_0 \dot{x} \cdot (x+y)\, dt + \int_{\T} H_t^*(J_0\dot{x})\, dt \\ &= -\frac{1}{2} \int_{\T} J_0 \dot{x}\cdot x\, dt + \frac{1}{2} \int_{\T} J_0 \dot y\cdot y +  \int_{\T} H_t^*(J_0\dot{x})\, dt \\ &= \Psi_{H^*}(x)  - \frac{1}{2} \|P^- y\|_{1/2}^2,
\end{split}
\]
where in the last equality we have used the fact that $y$ belongs to $\mathbb{H}^-_{1/2}\oplus \R^{2n}$. This shows that (\ref{disdadim}) holds, with equality if and only if the Fenchel inequality (\ref{fenin}) is an equality, that is if and only if (\ref{eqcase}) holds. In particular,
\[
\Phi_H(x+y) \leq \Psi_{H^*}(\pi(x))
\]
with equality if and only if $P^- y=0$ and (\ref{eqcase}) holds. This is equivalent to the fact that $y$ is a constant loop and the loop $x+y$ is a 1-periodic orbit of $X_H$, or equivalently a critical point of $\Phi_H$.
\end{proof}

The functional $\Psi_{H^*}$ is in general not twice differentiable, unless the function $x\mapsto H_t(x)$ is quadratic for every $t\in \T$, but it is twice Gateaux differentiable, meaning that for every $x,v\in \mathbb{H}_1$ the limit
\[
\nabla^2 \Psi_{H^*}(x) v := \lim_{h\rightarrow 0} \frac{1}{h} \bigl( \nabla \Psi_{H^*}(x+hv) - \nabla \Psi_{H^*}(x) \bigr)
\]
exists and defines a continuous linear operator $\nabla^2 \Psi_{H^*}(x)$ on $\mathbb{H}_1$. The corresponding second differential of $\Psi_{H^*}$ at $x$ has the form
\[
d^2 \Psi_{H^*}(x)[u,v] = -\int_{\T} J_0 \dot{u}\cdot v\, dt + \int_{\T} \nabla^2 H_t^*(J_0\dot{x}) J_0 \dot{u} \cdot J_0 \dot v\, dt.
\]
The second integral defines a coercive bilinear form on $\mathbb{H}_1$, that is,
\[
\int_\T\nabla^2H_t^*(J_0\dot x)J_0\dot u\cdot J_0\dot u dt\geq {\bar h}^{-1}\|u\|_{\mathbb{H}_1}^2.
\]
The first integral defines a bilinear form which is continuous in $H^{1/2}$ and, thanks to the compactness of the embedding $H^{1} \hookrightarrow H^{1/2}$, this bilinear form is represented by a compact operator with respect to the inner product of $\mathbb{H}_1$. It follows that the critical point $x$ of $\Psi_{H^*}$ has finite Morse index and finite nullity
\[
\begin{split}
\mathrm{ind}(x;\Psi_{H^*}) &:= \dim V^- \bigl( \nabla^2 \Psi_{H^*}(x) \bigr) \\ = \max &\{ \dim W \mid W \mbox{ linear subspace of } \mathbb{H}_1, \; d^2 \Psi_{H^*}(x) \mbox{ negative definite on } W \}, \\
\mathrm{null}(x;\Psi_{H^*}) &:= \dim \ker \nabla^2 \Psi_{H^*}(x).
\end{split}
\]
The next result could be deduced from Proposition \ref{cz=relind} and from the relationship between the Conley-Zehnder index and the Morse index of the dual action functional, see \cite{bro86,bro90,lon02}. Here we give a direct proof.

\begin{prop}
\label{relind=ind}
Let $x$ be a critical point of $\Phi_H$ and let $\pi(x)$ be the corresponding critical point of $\Psi_{H^*}$. Then
\[
\mathrm{null}(x;\Phi_{H}) = \mathrm{null}(\pi(x);\Psi_{H^*}) \quad \mbox{and} \quad
\mathrm{ind}_{\mathbb{H}_{1/2}^- \oplus E^-} (x;\Phi_H) = \mathrm{ind} (\pi(x);\Psi_{H^*})+ n.
\]
\end{prop}

\begin{proof}
We denote by 
\[
\widehat{\Psi}_{H^*}: H^1(\T,\R^{2n}) \longrightarrow \R, \qquad \widehat{\Psi}_{H^*} = \Psi_{H^*} \circ \pi
\]
the natural lift of the functional $\Psi_{H^*}$. Notice that $x\in H^1(\T,\R^{2n})$ is a critical point of $\widehat{\Psi}_{H^*}$ if and only if $\pi(x)$ is a critical point of $\Psi_{H^*}$. In this case we have
\begin{equation}
\label{lift}
\mathrm{null}(x;\widehat{\Psi}_{H^*}) = \mathrm{null}(\pi(x);\Psi_{H^*}) + 2n, \qquad
\mathrm{ind} (x; \widehat{\Psi}_{H^*}) = \mathrm{ind} (\pi(x);\Psi_{H^*}).
\end{equation}
Now let $x\in \mathbb{H}_{1/2}$ be a critical point of $\Phi_H$. Being a 1-periodic orbit of $X_H$, $x: \T \rightarrow \R^{2n}$ is smooth and in particular in $H^1(\T,\R^{2n})$.
Let $S$ be the smooth loop of positive definite symmetric endomorphisms of $\R^{2n}$ defined by
\[
S(t):= \nabla^2 H_t(x(t)).
\]
By Fenchel duality we have
\[
\nabla H_t^* \circ \nabla H_t = \mathrm{id},
\]
and differentiation gives us
\[
\nabla^2 H_t^* (\nabla H_t(z)) \nabla^2 H_t(z) = I \qquad \forall z\in \R^{2n}, \; \forall t\in \T.
\]
If $z=x(t)$ then $\nabla H_t(z)=J_0 \dot{x}(t)$, so the above identity yields
\[
\nabla^2 H_t^* (J_0 \dot{x}(t)) \nabla^2 H_t(x(t)) = I, \qquad \forall t\in \T,
\]
and hence
\[
\nabla^2 H_t^* (J_0 \dot{x}(t)) = S(t)^{-1}, \qquad \forall t\in \T.
\]
If we denote by $\Omega$ the symmetric bilinear form
\[
\Omega : \mathbb{H}_{1/2} \times \mathbb{H}_{1/2} \rightarrow \R, \qquad \Omega[u,v] = \int_{\T} J_0 \dot{u}\cdot v \, dt,
\]
the second differentials of $\Phi_H$ and $\widehat{\Psi}_{H^*}$ at $x$ take the form
\begin{equation}
\label{diffe}
\begin{split}
d^2 \Phi_H(x)[u,v] &= \Omega[u,v] - \int_{\T} S u \cdot v\, dt, \qquad \forall u,v\in \mathbb{H}_{1/2}\\
d^2 \widehat{\Psi}_{H^*}(x)[u,v] &= -\Omega[u,v] + \int_{\T} S^{-1} (J_0\dot{u}) \cdot (J_0\dot{v}) \, dt, \qquad \forall u,v\in \mathbb{H}_1.
\end{split}
\end{equation}
Since $S(t)$ is symmetric and positive definite, the formula 
\[
(u,v)_S := \int_{\T} S(t) u(t)\cdot v(t)\, dt
\]
defines an equivalent inner product on $L^2(\T,\R^{2n})$. Let $T$ be the symmetric operator which represents the bilinear form $\Omega$ with respect to the inner product $(\cdot,\cdot)_S$:
\[
\Omega [u,v] = (T u,v)_S, \qquad \mbox{where } T := S^{-1} J_0 \frac{d}{dt}.
\]
Here $T$ is an unbounded operator on $L^2(\T,\R^{2n})$ with domain $H^1(\T,\R^{2n})$ and the above identity holds for every $u\in H^1(\T,\R^{2n})$ and $v\in L^2(\T,\R^{2n})$. The operator $T$ is self-adjoint and has a compact resolvent. Its spectrum is discrete and consists of real eigenvalues with finite multiplicities which are unbounded from above and from below. The space $L^2(\T,\R^{2n})$ admits a Hilbert basis $\{\varphi_j\}_{j\in \Z}$ of eigenvectors of $T$  which is orthonormal with respect to the inner product $(\cdot,\cdot)_S$:
\[
T \varphi_j = \mu_j \varphi_j \qquad \forall j\in \Z.
\]
Here, the real eigenvalues  $\mu_j$ satisfy
\[
\lim_{j\rightarrow -\infty} \mu_j = - \infty, \qquad \lim_{j\rightarrow +\infty} \mu_j = + \infty.
\]
Being solutions of the differential equation
\[
J_0 \dot{\varphi}_j = \mu_j S \varphi_j,
\]
the eigenvectors $\varphi_j$ are smooth loops. There are exactly $2n$ eigenvalues $\mu_j$ with value $0$ and the corresponding eigenvectors $\varphi_j$ are constant loops forming a basis of $\R^{2n}$. 

Let $V$ be the closure in $\mathbb{H}_{1/2}$ of the linear subspace
\[
\mathrm{span} \{\varphi_j \mid \mu_j < 0\}.
\]
Then $V$ is a maximal subspace of $\mathbb{H}_{1/2}$ on which $\Omega$ is negative definite. Therefore, 
\begin{equation}
\label{splitting}
\mathbb{H}_{1/2} = V \oplus \R^{2n} \oplus \mathbb{H}_{1/2}^+.
\end{equation}
By the first formula in (\ref{diffe}), the second differential of $\Phi_H$ at $x$ has the following representation with respect to the inner product $(\cdot,\cdot)_S$:
\[
d^2 \Phi_H(x)[u,v] = ((T-I) u,v)_S.
\]
Therefore,
\begin{equation}
\label{nullPhi}
\ker d^2 \Phi_H(x) = \ker (T-I) = \mathrm{span} \{\varphi_j \mid \mu_j = 1\}.
\end{equation}
Moreover, the $\mathbb{H}_{1/2}$-closure of the linear subspace
\[
\mathrm{span} \{\varphi_j \mid \mu_j < 1\}
\]
is a maximal subspace of $\mathbb{H}_{1/2}$ on which $d^2 \Phi_H(x)$ is negative definite. This space has the form
\[
V \oplus \R^{2n} \oplus W,
\]
where $W$ is the following finite dimensional subspace 
\[
W := \mathrm{span} \{\varphi_j \mid 0< \mu_j < 1\}.
\]
Together with (\ref{splitting}), we compute the relative Morse index of $x$ defined in \eqref{eq:relative_ind} as follows:
\begin{equation}
\label{ilrelind}
\begin{split}
\mathrm{ind}_{\mathbb{H}_{1/2}^-\oplus E^-}(x;\Phi_H) &= \dim(V \oplus \R^{2n} \oplus W,\mathbb{H}_{1/2}^- \oplus E^-) \\
&=  \dim(V \oplus \R^{2n} \oplus W,\mathbb{H}_{1/2}^- \oplus \R^{2n} \oplus W) \\
&\quad + \dim ( \mathbb{H}_{1/2}^- \oplus \R^{2n} \oplus W, \mathbb{H}_{1/2}^- \oplus E^-)\\
&= \dim(V ,\mathbb{H}_{1/2}^-) + \dim W + \dim E^+=  \dim W + n.
\end{split}
\end{equation}
Since
\[
\begin{split}
\int_{\T} S^{-1} (J_0\dot{u}) \cdot (J_0\dot{v}) \, dt &= \int_{\T} (J_0\dot{u}) \cdot S^{-1} (J_0\dot{v}) \, dt = \int_{\T} S S^{-1} (J_0\dot{u}) \cdot S^{-1} (J_0\dot{v}) \, dt \\ &= (Tu,Tv)_S = (T^2 u,v)_S,
\end{split}
\]
the second formula in (\ref{diffe}) gives us the following representation for $d^2 \widehat{\Psi}_{H^*}(x)$ with respect to the inner product $(\cdot,\cdot)_S$:
\[
d^2 \widehat{\Psi}_{H^*}(x)[u,v] = ((T^2-T)u,v)_S = (T(T-I)u,v)_S.
\]
The Hilbert basis $\{\varphi_j\}_{j\in \Z}$ is a basis of eigenvectors for $T(T-I)$ and the eigenvalue corresponding to $\varphi_j$ is $\mu_j(\mu_j -1)$. This eigenvalue is 0 if and only if $\mu_j=0$ or $\mu_j=1$. Therefore, the kernel of $d^2 \widehat{\Psi}_{H^*}(x)$ is given by
\[
\ker d^2 \widehat{\Psi}_{H^*}(x) = \mathrm{span} \{ \varphi_j \mid \mu_j=0\} \oplus \mathrm{span} \{ \varphi_j \mid \mu_j=1\} = \R^{2n} \oplus \ker d^2 \Phi_H(x),
\]
where we have used (\ref{nullPhi}).  Together with the first identity in (\ref{lift}), we deduce that
\[
\mathrm{null}(x;\Phi_{H}) = \mathrm{null}(x;\widehat{\Psi}_{H^*}) - 2n = \mathrm{null}(\pi(x);\Psi_{H^*}).
\]
The eigenvalue $\mu_j(\mu_j-1)$ is negative if and only if $0<\mu_j <1$. Therefore, the finite dimensional space $W$ defined above is a maximal subspace of $H^1(\T,\R^{2n})$ on which $d^2 \widehat{\Psi}_{H^*}(x)$ is negative definite and hence
\[
\ind (x;\widehat{\Psi}_{H^*}) = \dim W = \mathrm{ind}_{\mathbb{H}_{1/2}^-\oplus E^-}(x;\Phi_H) - n,
\]
where we have used (\ref{ilrelind}). By the second identity in (\ref{lift}), we conclude that
\[
\ind (\pi(x);\Psi_{H^*}) = \mathrm{ind}_{\mathbb{H}_{1/2}^-\oplus E^-}(x;\Phi_H) - n.
\]
\end{proof}

Let $C$ be a smooth strongly convex domain, as defined at the beginning of this section. We conclude this section by showing how Clarke's duality can be used to determine the relative Morse index - and hence by Proposition \ref{cz=relind} the Conley-Zehnder index - of periodic Hamiltonian orbits that correspond to closed characteristics of $\partial C$ having minimal action.

\begin{prop}\label{index2}
Let $C$ be a smooth strongly convex domain.
If $\gamma:\R/T\Z\to\partial C$ is a periodic orbit of $R_{\alpha_C} = X_{H_C}|_{\partial C}$ with minimal period $T$ among all periodic Reeb orbits on $\partial C$, then setting $x_\gamma(t):=\gamma(Tt)$ we have
\[
\mathrm{ind}_{\mathbb{H}^-_{1/2} \oplus E^-}(x_\gamma,\Phi_{TH_C})=n.
\]
\end{prop}

\begin{proof}
Set 
\[
H:=\varphi\circ H_C, \quad \mbox{with} \quad \varphi(r) = \frac{2T}{p} r^{\frac{p}{2}},
\]
for some real number $p\in (1,2)$. Applying Proposition \ref{index1}, we know that $x_\gamma$ is a critical point of $\Phi_H$, which we now see as a functional on $W^{1,p}(\T,\R^{2n})$. Let $q>2$ be such that $1/p+1/q=1$. The dual functional $\Psi_{H^*}$ is well-defined on the space $W^{1,q}_0(\T,\R^{2n})$ of $W^{1,q}$-loops in $\R^{2n}$ with zero mean, and the proof of Lemma \ref{crit} shows that $\pi(x_\gamma)$ is a critical point of $\Psi_{H^*}$. We use the well-known fact, which we will recall below, that $\Psi_{H^*}$ is bounded from below and attains a global minimizer that corresponds exactly to a periodic orbit of $X_{H_C}$ on $\partial C$ with minimal period. From this, we deduce that $\pi(x_\gamma)$ is a minimizer of $\Psi_{H^*}$ and therefore
\[
\ind(\pi(x_\gamma);\Psi_{H^*})=0.
\]
Then Proposition \ref{relind=ind} and Proposition \ref{index1} yield the identities
\[
\mathrm{ind}_{\mathbb{H}_{1/2}^- \oplus E^-} (x_\gamma;\Phi_{TH_C})=\mathrm{ind}_{\mathbb{H}_{1/2}^- \oplus E^-} (x_\gamma;\Phi_{H})=\ind(\pi(x_
\gamma);\Psi_{H^*})+n=n.
\]

For sake of completeness, we include a proof of the well-known fact mentioned above. Since $H_C$ is positively 2-homogeneous, $H$ is positively $p$-homogeneous and $H^*$ is positively $q$-homogeneous with $q>2$. Using this and the Poincar\'e-Wirtinger inequality, we obtain the following lower bound for an arbitrary element $x$ of $W^{1,q}_0(\T,\R^{2n})$: 
\[
\begin{split}
\Psi_{H^*}(x) &= - \frac{1}{2} \int_{\T} J_0\dot{x}\cdot x \, dt + \int_{\T} H^*(J_0 \dot{x})\, dt \geq -\frac{1}{2} \|\dot x\|_{L^2} \|x\|_{L^2} + c \|\dot{x}\|_{L^q}^q\\
&\geq - \frac{1}{4\pi} \|\dot x\|_{L^2}^2 + c \|\dot{x}\|_{L^q}^q \geq  \frac{c}{2} \|\dot{x}\|_{L^q}^q - d
\end{split}
\]
for some suitable constants $c,\,d>0$. Since the $L^q$ norm of the derivative defines a Banach norm on $W^{1,q}_0(\T,\R^{2n})$, thanks to the $L^q$-version of the Poincar\'e-Wirtinger inequality, the above lower bound shows that the sublevels of the functional $\Psi_{H^*}$ are bounded in  $W^{1,q}_0(\T,\R^{2n})$. Moreover, $\Psi_{H^*}$ is weakly lower semi-continuous on $W^{1,q}_0(\T,\R^{2n})$, because the quadratic form
\[
\frac{1}{2} \int_{\T} J_0\dot{x}\cdot x \, dt 
\]
is weakly continuous, being strongly continuous on the space $\mathbb{H}_{1/2}$ into which $W^{1,q}(\T,\R^{2n})$ embeds compactly, and the term
\[
\int_{\T} H^*(J_0 \dot{x})\, dt
\]
is strongly continuous and convex. Therefore, the sublevels $\{\Psi_{H^*}\leq c\}$ are bounded and weakly closed in $W^{1,q}_0(\T,\R^{2n})$, and hence weakly compact in this reflexive Banach space, thanks to the Banach-Alaoglu theorem. We conclude that $\Psi_{H^*}$ attains its minimum on $W^{1,q}_0(\T,\R^{2n})$.

Moreover, this minimum  is negative, since for any element $x$ of $W^{1,q}_0(\T,\R^{2n})$ with positive action $\frac{1}{2}\int_\T J_0\dot x\cdot x dt\,>0$ and any $\lambda>0$ sufficiently small we have
\[
\Psi_{H^*}(\lambda x)=-\frac{\lambda^2}{2}\int_\T J_0\dot x\cdot x\,dt + \lambda^q\int_\T H^*(-J_0 \dot x)\,dt <0,
\]
because $q>2$. Since the loop $z$ mapping to the origin is the unique constant critical point of $\Psi_{H^*}$ and $\Psi_{H^*}(z)=0$, every minimizer of $\Psi_{H^*}$ is non-constant.

Next we compute critical values of $\Psi_{H^*}$. If $x$ is a critical point of $\Psi_{H^*}$ and $y$ is the critical point of $\Phi_H$ with $\pi(y)=x$, then 
\[
\begin{split}
\Psi_{H^*}(x)=\Phi_H(y)& =  \frac{1}{2} \int_{\T} J_0 \dot{y}\cdot y\, dt - \int_{\T} H(y)\, dt, \\
&= \frac{1}{2} \int_{\T} \nabla H(y)\cdot y\, dt -  H(y(0)) =\left(\frac{p}{2}-1\right)H(y(0))
\end{split}
\]
where we have used the Euler identity and the fact that $H$ is constant along $y$. In particular if $x$ is non-constant, $\Psi_{H^*}(x)<0$. A simple computation shows that the curve
\[
 t\mapsto H(y(0))^{-\frac{1}{p}}y\left(\frac{2}{p}H(y(0))^{\frac{2-p}{p}}t\right),\qquad t\in\R
\] 
is a periodic orbit of $X_{H_C}$ sitting on $\partial C$ with period 
\[
\frac{p}{2}H(y(0))^\frac{p-2}{p}=\frac{p}{2}\left(\frac{2}{p-2}\Psi_{H^*}(x)\right)^\frac{p-2}{p}.
\]
The above formula shows the following: There exists a one-to-one correspondence between the closed orbits of $X_{H_C}$ on $\partial C$ and the critical points of $\Psi_{H^*}$ with negative critical value, and the function that to every negative critical value of $\Psi_{H^*}$ associates the period of the corresponding closed orbit - or orbits - of $X_{H_C}$ on $\partial C$ is strictly monotonically increasing. In particular, the global minimizers of $\Psi_{H^*}$ correspond to periodic Reeb orbits on $\partial C$ with minimal period, as claimed above.
\end{proof}

\section{The Morse complex of the dual action functional}
\label{sec:Morse}

The first aim of this section is to prove the Palais-Smale condition for the dual action functional that is associated with a quadratically convex Hamiltonian that is non-resonant at infinity. We begin with the following lemma:

\begin{lem}
\label{psb}
Let $K:\T \times \R^{2n} \rightarrow \R$ be a smooth function such that
\begin{equation}
\label{c1}
| dK_t(x) | \leq c (1+|x|) \qquad \forall (t,x) \in \T\times \R^{2n},
\end{equation}
and
\begin{equation}
\label{c2}
\nabla^2 K_t(x) \geq \delta I \qquad \forall (t,x) \in  \T\times \R^{2n}, 
\end{equation}
for some positive number $\delta$. Then the functional
\[
\Psi_K : \mathbb{H}_1 \rightarrow \R, \qquad \Psi_K(x) = -\frac{1}{2} \int_{\T} J_0\dot{x}(t)\cdot x(t)\, dt + \int_{\T} K_t(J_0\dot{x}(t))\, dt,
\]
has the following property: If the sequence $(x_h)\subset \mathbb{H}_1$ converges weakly to $x\in \mathbb{H}_1$ and $d\Psi_K(x_h)$ converges to zero strongly in the dual of $\mathbb{H}_1$, then $(x_h)$ converges to $x$ strongly in $\mathbb{H}_1$.
\end{lem}

\begin{proof}
Assumption (\ref{c1}) guarantees that $\Psi_K$ is continuously differentiable on $\mathbb{H}_1$. From the assumption on $(d\Psi_K(x_h))$ and the boundedness of $(x_h-x)$ in $\mathbb{H}_1$ we deduce that the real sequence
\[
d\Psi_K(x_h)[x_h-x] = -\int_{\T} J_0 (\dot{x}_h - \dot{x})\cdot x_h\, dt + \int_{\T} dK_t(J_0 \dot{x}_h)[J_0(\dot{x}_h - \dot{x})]\, dt
\]
is infinitesimal. The fact that $(\dot{x}_h - \dot{x})$ converges to zero weakly in $L^2$ and $(x_h)$ converges to $x$ weakly in $H^1$ and hence strongly in $L^2$ implies that the first integral in the above expression is infinitesimal. Therefore, the second integral must be infinitesimal too:
\begin{equation}
\label{inf}
\int_{\T} dK_t(J_0 \dot{x}_h)[J_0(\dot{x}_h - \dot{x})]\, dt = o(1).
\end{equation}
From assumption (\ref{c2}) we deduce the inequality
\[
\begin{split}
dK_t(J_0\dot{x}_h)[J_0(\dot{x}_h &- \dot{x})] - dK_t(J_0\dot{x})[J_0(\dot{x}_h - \dot{x})] \\ &= \int_0^1 d^2 K_t(J_0\dot{x} + s J_0(\dot{x}_h-\dot{x}))[ J_0 (\dot{x}_h - \dot{x}),J_0 (\dot{x}_h - \dot{x})]\, ds \\ &\geq \delta |J_0 (\dot{x}_h - \dot{x})|^2 = \delta |\dot{x}_h-\dot{x}|^2 \qquad \mbox{a.e.}
\end{split}
\]
By integrating this inequality over $\T$ we get
\[
\delta \int_{\T} |\dot{x}_h-\dot{x}|^2 \, ds \leq \int_{\T} dK_t(J_0\dot{x}_h)[J_0(\dot{x}_h - \dot{x})]\, dt - \int_{\T} dK_t(J_0\dot{x})[J_0(\dot{x}_h - \dot{x})]\, dt.
\]
The first integral on the right-hand side is infinitesimal because of (\ref{inf}). The second integral is also infinitesimal, because $(\dot{x}_h-\dot{x})$ converges to zero weakly in $L^2$ and $dK_t(J_0\dot{x})$ is an $L^2$ function, thanks to (\ref{c1}). We conclude that the $L^2$-norm of $\dot{x}_h-\dot{x}$ is infinitesimal, that is, $(x_h)$ converges to $x$ in $\mathbb{H}_1$.
\end{proof}

\begin{prop}
\label{psc}
Assume that the Hamiltonian $H\in C^{\infty}(\T \times \R^{2n})$ is quadratically convex and non-resonant at infinity. Then the dual action functional $\Psi_{H^*} : \mathbb{H}_1 \rightarrow \R$ satisfies the Palais-Smale condition. More precisely, any sequence $(x_h)\subset\mathbb{H}_1$ such that $d\Psi_{H^*}(x_h)$ converges to zero strongly in the dual of $\mathbb{H}_1$ has a convergent subsequence.
\end{prop}

\begin{proof}
By (\ref{eq:Psi'}), the differential of $\Psi_{H^*}$ at $x\in \mathbb{H}_1$ has the form
\[
d\Psi_{H^*}(x)[v] = \bigl( \dot{v}, J_0(x- \nabla H_t^*(J_0 \dot{x})) \bigr)_{L^2(\T)}.
\]
Therefore, endowing $\mathbb{H}_1$ with the inner product given by the $L^2$-product of the derivatives, the gradient of $\Psi_{H^*}$ has the form
\[
\nabla \Psi_{H^*}(x) = \Pi \bigl( J_0(x- \nabla H_t^*(J_0 \dot{x})) \bigr),
\]
where 
\begin{equation}\label{eq:integration}
\Pi: L^2(\T,\R^{2n}) \rightarrow \mathbb{H}_1,\qquad (\Pi v)(t)=\int_0^t v(s)\,ds-\int_\T\left(\int_0^tv(s)\,ds\right)dt.
\end{equation}
is the linear operator mapping each $v$ into the primitive of $v$ with zero mean.

Let $(x_h)\subset \mathbb{H}_1$ be a sequence such that $(d\Psi_{H^*}(x_h))$ converges to zero strongly in $\mathbb{H}_1^*$, or equivalently $(\nabla \Psi_{H^*}(x_h))$ converges to zero strongly in $\mathbb{H}_1$. Then
\[
\Pi(x_h - \nabla H_t^*(J_0 \dot{x}_h)) = y_h
\]
where $(y_h)\subset \mathbb{H}_1$ converges to zero strongly. Differentiation in $t$ yields
\[
x_h - \nabla H_t^*(J_0 \dot{x}_h)= \dot{y}_h.
\]
By applying the nonlinear map $\nabla H_t$ to the identity
\[
\nabla H_t^*(J_0 \dot{x}_h) = x_h - \dot{y}_h,
\]
we obtain
\[
J_0 \dot{x}_h = \nabla H_t(x_h - \dot{y}_h),
\]
or equivalently
\begin{equation}
\label{dari}
\dot{x}_h = X_{H_t}(x_h - \dot{y}_h).
\end{equation}
Therefore, using the fact that $X_{H_t}$ is globally Lipschitz-continuous (see (\ref{globlip})), we find
\[
|\dot{x}_h - X_{H_t}(x_h)| = |X_{H_t}(x_h - \dot{y}_h) - X_{H_t}(x_h)| \leq \overline{h} |\dot{y}_h|,
\]
and integrating over $\T$ we obtain the bound
\[
\|\dot{x}_h - X_{H}(x_h)\|_{L^2(\T)} \leq  \overline{h} \|\dot{y}_h\|_{L^2(\T)}.
\]
Therefore, the sequence $(\dot{x}_h - X_{H}(x_h))$ is infinitesimal in $L^2(\T)$, and hence the non-resonance at infinity assumption implies that $(x_h)$ is uniformly bounded in $L^2$. But then the identity (\ref{dari}) and the linear growth of $X_H$ imply that $(\dot{x}_h)$ is uniformly bounded in $L^2(\T)$, and hence $(x_h)$ is uniformly bounded in $\mathbb{H}_1$.
 
Up to passing to a subsequence, we may assume that $(x_h)$ converges to some $x$ weakly in  $\mathbb{H}_1$. By Lemma \ref{psb} we conclude that this convergence is strong. This proves that $\Psi_{H^*}$ satisfies the Palais-Smale condition.
\end{proof} 

If the quadratically convex Hamiltonian $H\in C^{\infty}(\T\times \R^{2n})$ is non-degenerate, then the functional $\Psi_{H^*}$ is Morse, meaning that the (Gateaux) second differential of $\Psi_{H^*}$ at each critical point is non-degenerate. However, the functional $\Psi_{H^*}$ is in general not of class $C^2$ (it is not even twice differentiable), so some care is needed in order to associate a Morse complex with it.

One way of doing this would be to show that $\Psi_{H^*}$ admits a smooth pseudo-gradient vector field on $\mathbb{H}_1$ with good properties. This has been done in another setting for a functional whose analytical properties are similar to those of $\Psi_{H^*}$, see \cite{as09b}. Here we prefer to use a different strategy and to use the fact that $\Psi_{H^*}$ is smooth when restricted to a suitable finite dimensional smooth submanifold of $\mathbb{H}_1$, which contains all the critical points of $\Psi_{H^*}$ and is defined by a saddle-point reduction. This approach has also been used by Viterbo in \cite{vit89b}.

Given a natural number $N\in\N$, consider the splitting
\[
\mathbb{H}_1 = \mathbb{H}_1^{N,+} \oplus \widehat{\mathbb{H}}_1^{N,+},
\]
where 
\[
\begin{split}
\mathbb{H}_1^{N,+} &:= \left\{ x\in \mathbb{H}_1 \mid x(t) = \sum_{k=1}^N e^{-2\pi k J_0 t} \hat{x}_k, \; \hat{x}_k\in \R^{2n} \right\}, \\
\widehat{\mathbb{H}}_1^{N,+} &:= \left\{ x\in \mathbb{H}_1 \mid x(t) = \sum_{k\leq -1} e^{-2\pi k J_0 t}\hat{x}_k +\sum_{k\geq N+1} e^{-2\pi k J_0 t}\hat{x}_k, \; \hat{x}_k\in \R^{2n} \right \}.
\end{split}
\]
This splitting is orthogonal with respect to the $\mathbb{H}_1$ and to the $L^2$ inner products. We identify $\mathbb{H}_1$ with the product space $\mathbb{H}_1^{N,+} \times \widehat{\mathbb{H}}_1^{N,+}$. The following proposition summarizes the main properties of the saddle point reduction.

\begin{prop}\label{prop:reduction}
Assume the Hamiltonian $H\in C^{\infty}(\T \times \R^{2n})$ to be quadratically convex and non-resonant at infinity.
If $N\in \N$ is large enough, then the following facts hold:
\begin{enumerate}[(a)]
\item For every $x\in \mathbb{H}_1^{N,+}$ the restriction of $\Psi_{H^*}$ to $\{x\} \times \widehat{\mathbb{H}}_1^{N,+}$ has a unique critical point $(x,Y(x))$, which is a non-degenerate global minimizer of this restriction. 
\item The map $Y: \mathbb{H}_1^{N,+} \rightarrow \widehat{\mathbb{H}}_1^{N,+}$ takes values in $C^\infty(\T,\R^{2n})$ and is smooth with respect to the $C^k$-norm  on the target, for any $k\in\N$. In particular, its graph
\[
M:= \{(x,y) \in \mathbb{H}_1^{N,+} \times \widehat{\mathbb{H}}_1^{N,+} \mid y=Y(x)\}
\]
is a smooth $2nN$-dimensional submanifold of $\mathbb{H}_1$ consisting of smooth loops.
\item The restriction of $\Psi_{H^*}$ to $M$, which we denote by $\psi_{H^*} : M \rightarrow \R$, is smooth.
\item A point $z\in \mathbb{H}_1$ is a critical point of $\Psi_{H^*}$ if and only if it belongs to $M$ and is a critical point of $\psi_{H^*}$. In this case, the Morse index and the nullity with respect to the two functionals coincide:
\[
\mathrm{ind}(z;\Psi_{H^*}) = \mathrm{ind}(z;\psi_{H^*}), \qquad \mathrm{null}(z;\Psi_{H^*}) = \mathrm{null}(z;\psi_{H^*}).
\]
\item If $M$ is endowed with the Riemannian metric induced by the inclusion into $\mathbb{H}_1$, then the functional $\psi_{H^*}$ satisfies the Palais-Smale condition.
\end{enumerate}
\end{prop}

The proof of this proposition is postponed to the end of this section. Statements (a)-(d) require only the assumption of quadratic convexity, whereas the proof of (e) uses also the assumption of non-resonance at infinity.
If we further assume that the Hamiltonian $H$ is non-degenerate, from (d) and (e) we obtain that $\psi_{H^*}$ is a smooth Morse function with finitely many critical points and satisfying the Palais-Smale condition on the finite-dimensional manifold $M$. As such, it has a Morse complex, which we denote by
\[
\{M_*(\psi_{H^*}),\partial^M\}
\]
and is uniquely defined up to chain isomorphisms. The space $M_*(\psi_{H^*})$ is the $\Z_2$-vector space generated by the critical points of $\psi_{H^*}$, graded by the Morse index. The boundary operator
\[
\partial^M : M_*(\psi_{H^*}) \rightarrow M_{*-1}(\psi_{H^*})
\]
is defined by the formula
\[
\partial^M x= \sum_{y} n^M(x,y) y \qquad \forall x\in \mathrm{crit}\, \psi_{H^*},
\]
where $y$ ranges over all critical points with Morse index equal to the index of $x$ minus 1 and $n^M(x,y)\in \Z_2$ is the parity of the finite set of negative gradient flow lines of $\psi_{H^*}$ going from $x$ to $y$. Here, the negative gradient vector field of $\psi_{H^*}$ is induced by a generic Riemannian metric on $M$, uniformly equivalent to the standard one and such that the negative gradient flow is Morse-Smale, meaning that stable and unstable manifolds of pairs of critical points meet transversally. Changing the generic metric changes the Morse complex by a chain isomorphism. The homology of the Morse complex of $\psi_{H^*}$ is isomorphic to the singular homology of the pair $(M,\{\psi_{H^*}<a\})$, where $a$ is any number which is smaller than the smallest critical level of $\psi_{H^*}$:
\begin{equation}\label{morse=singular}
HM_k(\psi_{H^*}) \cong H_k(M,\{\psi_{H^*}<a\}).
\end{equation}

We conclude this section by proving Proposition \ref{prop:reduction}.

\begin{proof}[Proof of Proposition \ref{prop:reduction}]
By the inequality $\nabla^2H_t^*(x)\geq \overline{h}^{-1}I$, we have for every $x,u\in\mathbb{H}_1$,
\[
\begin{split}
d^2\Psi_{H^*}(x)[u,u]&=-\int_\T J_0\dot u\cdot u\,dt+\int_\T\nabla^2H_t^*(J_0\dot x)J_0\dot u\cdot J_0\dot u\,dt\\
&\geq-\int_\T J_0\dot u\cdot u\,dt+\frac{1}{\overline{h}}\|\dot u\|^2_{L^2}\\
&=-2\pi\sum_{k\in\Z}k|\hat{u}_k|^2+\frac{4\pi^2}{\overline{h}}\sum_{k\in\Z}k^2|\hat{u}_k|^2.
\end{split}
\]
We choose $N\in\N$ so that $2\pi(N+1)>\overline{h}$. For all $x\in\mathbb{H}_1$ and all $u\in\widehat{\mathbb{H}}_1^{N,+}$, we have
\begin{equation}\label{eq:lower_bound}
\begin{split}
d^2\Psi_{H^*}(x)[u,u]&\geq\frac{4\pi^2}{\overline{h}}\sum_{k\leq -1}k^2|\hat{u}_k|^2+2\pi\sum_{k\geq N+1}\left(\frac{2\pi k}{\overline{h}}-1\right)k|\hat{u}_k|^2\\
&\geq4\pi^2\delta\left(\sum_{k\leq -1}k^2|\hat{u}_k|^2+\sum_{k\geq N+1}k^2|\hat{u}_k|^2\right)\\
&=\delta\|u\|^2_{\mathbb{H}_1},
\end{split}
\end{equation}
where $\delta>0$ is a sufficiently small constant. This shows that for every $x\in\mathbb{H}_1^{N,+}$, the second differential of the function 
\[
\widehat{\mathbb{H}}_1^{N,+}\to\R,\qquad y\mapsto\Psi_{H^*}(x+y)
\]
is bounded from below by a coercive quadratic form. In particular, this function is strictly convex and coercive and hence has a unique non-degenerate critical point $Y(x)$ which is a minimizer. This proves (a).\\[-2ex]

From the expression \eqref{eq:Psi'} for $d\Psi_{H^*}$ we deduce that the gradient of $\Psi_{H^*}$ with respect to the inner product $(\cdot,\cdot)_{\mathbb{H}_1}$ is 
\begin{equation}\label{eq:gradient_Psi}
\nabla \Psi_{H^*}(x)=\Pi\big(J_0x-J_0\nabla H_t^*(J_0\dot x)\big)
\end{equation}
where $\Pi:L^2(\T,\R^{2n})\to\mathbb{H}_1$ is the inverse of the derivative given by \eqref{eq:integration}.
The vector $y\in\widehat{\mathbb{H}}_1^{N,+}$ satisfies $y=Y(x)$ for some $x\in\mathbb{H}_1^{N,+}$ if and only if $\nabla\Psi_{H^*}(x+y)\in\mathbb{H}_1^{N,+}$ which by \eqref{eq:gradient_Psi} is equivalent to 
\[
J_0(x+y)-J_0\nabla H_t^*(J_0(\dot x+\dot y))=\dot u
\]
for some $u\in\mathbb{H}_1^{N,+}$. The above equality can be reformulated as 
\begin{equation}\label{eq:formula_y}
\dot y=-J_0\nabla H_t(x+y+J_0\dot u)-\dot x.
\end{equation}
The fact that $x$ and $u$ are smooth implies that $y=Y(x)$ is also smooth.

We now deal with the regularity of the map $Y$ and start by showing that $Y$ is Lipschitz continuous. We  use subscripts 1 and 2 to denote partial derivatives with respect to the splitting $\mathbb{H}_1=\mathbb{H}_1^{N,+}\times\widehat{\mathbb{H}}_1^{N,+}$. Then by (a) we have for every $x\in\mathbb{H}_1^{N,+}$,
\[
\nabla_2\Psi_{H^*}(x,Y(x))=0.
\]
From the fact that $\Psi_{H^*}$ is twice Gateaux-differentiable and from the lower bound \eqref{eq:lower_bound}, we deduce for all $x\in\mathbb{H}_1^{N,+}$ and all $u\in\widehat{\mathbb{H}}_1^{N,+}$
\begin{equation}\label{eq:nabla2}
\begin{split}
\left\|\nabla_2\Psi_{H^*}(x,Y(x)+u)\right\|_{\mathbb{H}_1}&=\left\|\nabla_2\Psi_{H^*}(x,Y(x)+u)-\nabla_2\Psi_{H^*}(x,Y(x))\right\|_{\mathbb{H}_1}\\
&=\left\|\left(\int_0^1\nabla_{22}\Psi_{H^*}(x,Y(x)+tu)dt\right)u\right\|_{\mathbb{H}_1}\\
&\geq\delta\|u\|_{\mathbb{H}_1}.
\end{split}
\end{equation}
The above inequality with $x$ replaced by $x+h$, $h\in\mathbb{H}_1^{N,+}$ and with $u=Y(x)-Y(x+h)$ gives us
\[
\begin{split}
\delta\|Y(x+h)-Y(x)\|_{\mathbb{H}_1}&\leq \left\|\nabla_2\Psi_{H^*}(x+h,Y(x))\right\|_{\mathbb{H}_1}\\
&=\left\|\nabla_2\Psi_{H^*}(x+h,Y(x))-\nabla_2\Psi_{H^*}(x,Y(x)) \right\|_{\mathbb{H}_1}\\
&\leq c\|h\|_{\mathbb{H}_1}
\end{split}
\]
for some constant $c>0$ where we used the fact that $\nabla\Psi_{H^*}$ is Lipschitz continuous. This proves that $Y$ is Lipschitz continuous.

Let $x,h\in\mathbb{H}_1^{N,+}$. By the Gateaux differentiability of $\nabla\Psi_{H^*}$, a first order expansion yields
\[
\begin{split}
\nabla_2\Psi_{H^*}\big(x+th,Y(x)&-t\nabla_{22}\Psi_{H^*}(x,Y(x))^{-1}\nabla_{12}\Psi_{H^*}(x,Y(x))h\big)\\
&=\nabla_2\Psi_{H^*}(x,Y(x))+t\nabla_{12}\Psi_{H^*}(x,Y(x))h\\
&\quad-t\nabla_{22}\Psi_{H^*}(x,Y(x))\nabla_{22}\Psi_{H^*}(x,Y(x))^{-1}\nabla_{12}\Psi_{H^*}(x,Y(x))h+o(t)\\
&=o(t),
\end{split}
\]
where $\nabla_{22}\Psi(x,Y(x))$ is invertible since it is self adjoint and bounded from below as observed in \eqref{eq:lower_bound}.
On the other hand, the bound \eqref{eq:nabla2} with $x$ replaced by $x+th$ and with $u=-Y(x+th)+Y(x)-t\nabla_{22}\Psi_{H^*}(x,Y(x))^{-1}\nabla_{12}\Psi_{H^*}(x,Y(x))h$ gives 
\[
\begin{split}
\delta\|Y(x+th)& - Y(x) + t\nabla_{22}\Psi_{H^*}(x,Y(x))^{-1}\nabla_{12}\Psi_{H^*}(x,Y(x))h\|_{\mathbb{H}_1}\\
&\leq \big\|\nabla_2\Psi_{H^*}\big(x+th,Y(x)-t\nabla_{22}\Psi_{H^*}(x,Y(x))^{-1}\nabla_{12}\Psi_{H^*}(x,Y(x)\big)h\big\|_{\mathbb{H}_1}\\
&=o(t).
\end{split}
\]
This shows that $Y:\mathbb{H}_1^{N,+}\to\widehat{\mathbb{H}}_1^{N,+}$ is Gateaux-differentiable with Gateaux-gradient
\begin{equation}\label{eq:nabla_Y}
\nabla Y(x)=-\nabla_{22}\Psi_{H^*}(x,Y(x))^{-1}\nabla_{12}\Psi_{H^*}(x,Y(x)).
\end{equation}
We have already seen that the map $Y$ takes values in $C^\infty(\T,\R^{2n})$. We claim that $Y$ is continuous with respect to the $C^\infty$-topology on the target space. Indeed, we assume that $(x_n)\subset\mathbb{H}_1^{N,+}$ converges to $x$ in the $\mathbb{H}_1$-norm. Since $\mathbb{H}_1^{N,+}$ is contained in $C^\infty(\T,\R^{2n})$ and is finite dimensional, $(x_n)$ converges to $x$ in the $C^k$-norm for any $k\in\N$. The vector 
\[
u_n=\nabla\Psi_{H^*}(x_n,Y(x_n))
\]
converges to $u=\nabla\Psi_{H^*}(x,Y(x))$ in the $\mathbb{H}_1$-norm due to the continuity of $Y$ and $\nabla\Psi_{H^*}$. Being a sequence in $\mathbb{H}_1^{N,+}$, the sequence $(u_n)$ converges to $u$ in any $C^k$-norm. As seen in \eqref{eq:formula_y}, the vector $y_n=Y(x_n)$ is characterized by
\[
\dot y_n=-J_0\nabla H_t(x_n+y_n+J_0\dot u_n)-\dot x_n.
\]
This ODE shows that $(y_n)$ converges to the solution $y=y(x)$ of
\[
\dot y=-J_0\nabla H_t(x+y+J_0\dot u)-\dot x
\]
in any $C^k$-norm. This proves the claim.

Although $\Psi_{H^*}$ is not of class $C^2$, the map $x\mapsto\nabla^2\Psi_{H^*}(x)$ is easily seen to be continuous from the $C^1$-topology on $\mathbb{H}_1$ to the operator norm topology on $L(\mathbb{H}_1,\mathbb{H}_1)$. Then the identity \eqref{eq:nabla_Y} and the regularity property of $Y$ proven above yield that the map 
\[
\nabla Y:\mathbb{H}_1^{N,+}\to L(\mathbb{H}_1^{N,+},\widehat{\mathbb{H}}_1^{N,+})
\]
is continuous. This together with the Gateaux-differentiability of $Y$ implies that the map $Y:\mathbb{H}_1^{N,+}\to\widehat{\mathbb{H}}_1^{N,+}$ is of class $C^1$ by the total differential theorem.

Since the restriction of $\Psi_{H^*}$ to $C^k(\T,\R^{2n})$ is smooth for all $k\in\N$, the above argument can be bootstrapped and implies that the map $Y$ is smooth with respect to the $C^k$-norm for all $k\in\N$ on the target. This completes the proof of (b).\\[-2ex]

Statement (c) follows from (b) and the smoothness property of $\Psi_{H^*}$ mentioned above.\\[-2ex]

To prove (d) we first observe that all critical points of $\Psi_{H^*}$ are contained in $M$. A point $(x,Y(x))\in M$ is a critical point of the restriction $\psi_{H^*}$ of $\Psi_{H^*}$ to $M$ if and only if
\[
d\Psi_{H^*}(x,Y(x))|_{T_{(x,Y(x))}M}=0
\]
which is equivalent to $d\Psi_{H^*}(x,Y(x))=0$ since $d\Psi_{H^*}(x,Y(x))|_{\widehat{\mathbb{H}}_1^{N,+}}=0$ and $\mathbb{H}_1=T_{(x,Y(x))}M\oplus \widehat{\mathbb{H}}_1^{N,+}$. This proves the first statement of (d).

The estimate \eqref{eq:lower_bound} that $d^2\Psi_{H^*}(x,Y(x))$ is positive on $\widehat{\mathbb{H}}_1^{N,+}$ guarantees that the index and the nullity does not change when restricting $\Psi_{H^*}$ to $M$. This completes the proof of (d).\\[-2ex]

The statement (e) follows immediately from Proposition \ref{psc} since any sequence $(z_h)\subset M$  with the property that $d\psi_{H^*}(z_h)$ converges to zero with respect to the Riemannian metric induced from $\mathbb{H}_1$ satisfies also that $d\Psi_{H^*}(z_h)$ strongly converges to zero. 
\end{proof}

\section{The functional setting for the hybrid problem}
\label{funsetsec}

Throughout this section, we assume that $H\in C^\infty(\T\times\R^{2n})$ is non-degenerate, quadratically convex, and non-resonant at infinity. 

Let $x$ and $y$ be 1-periodic orbits of $X_H$. We shall view $\pi(x)\in \mathbb{H}_1$ as a critical point of $\Psi_{H^*}$, and hence of $\psi_{H^*}$ on the finite dimensional manifold $M$ that is introduced in Section \ref{sec:Morse}, and $y\in \mathbb{H}_{1/2}$ as a critical point of $\Phi_H$. Let $J$ be a family of uniformly bounded $\omega_0$-compatible almost complex structures on $\R^{2n}$ parametrized by $[0,+\infty)\times\T$ such that $J=J_0$ on $[0,1]\times\T$ and $J(s,t)$ is independent of $s$ for $s$ large.
We denote by 
\[
\mathcal{M}(x,y)=\mathcal{M}(x,y;H,J)
\] 
the space of smooth maps
\[
u: [0,+\infty) \times \T \rightarrow \R^{2n}
\]
which solve the Floer equation
\[
\partial_s u + J(s,t,u) (\partial_t u - X_{H_t}(u)) = 0\qquad \mbox{on }\;\;  [0,+\infty)\times \T
\]
with the asymptotic condition
\[
\lim_{s\rightarrow +\infty} u(s,\cdot) = y \qquad \mbox{in }\;\; C^{\infty}(\T,\R^{2n}),
\]
and the boundary condition
\[
u(0,\cdot) \in \pi^{-1}\bigl(W^u(\pi(x);-\nabla\psi_{H^*}) \bigr) + \mathbb{H}_{1/2}^-.
\]
Here $W^u(\pi(x);-\nabla\psi_{H^*})$ is the unstable manifold of the negative gradient vector field of $\psi_{H^*}$ at $\pi(x)$ in the finite dimensional submanifold $M$ of $\mathbb H_1$, which is used to construct the Morse complex of $\psi_{H^*}$ in Section \ref{sec:Morse}. In other words, the trace of $u$ at the boundary of the half-cylinder is the sum of a loop in $W^u(\pi(x);-\nabla\psi_{H^*})$, seen as a submanifold of the space of loops with zero mean, and a loop in $\R^{2n} \oplus \mathbb{H}_{1/2}^-$.

The proposition below will be used in the following sections.
\begin{prop}
\label{hybrid_ineq}
If $u\in\mathcal{M}(x,y)$, we have
\[
\Phi_H(x)- \Phi_H(y)\geq \|\partial_su\|_{L^2([0,+\infty)\times\T)}
\] 
Moreover $\Phi_H(x)=\Phi_H(y)$ if and only if $x=y$, and in this case $\mathcal{M}(x,x)$ consists of a unique solution, namely the constant half-cylinder mapping to $x$.
\end{prop}
\begin{proof}
If $u\in\mathcal{M}(x,y)$, then $u(0,\cdot) = v + w$ for some $v\in W^u(\pi(x);-\nabla\psi_{H^*})$ and $w\in \R^{2n} \oplus \mathbb{H}_{1/2}^-$. This yields the estimate
\begin{equation}
\label{bdden}
\begin{split}
\int_{[0,+\infty)\times \T} |\partial_s u|^2\, ds dt &= \Phi_H(u(0,\cdot)) - \Phi_H(y) \leq \Psi_{H^*}(v) - \Phi_H(y) \\ &\leq \Psi_{H^*}(\pi(x)) - \Phi_H(y) = \Phi_{H}(x) - \Phi_H(y)
\end{split}
\end{equation}
where we have used Lemma \ref{crit} and Proposition \ref{confronto}. If we have $\Phi_H(x)=\Phi_H(y)$, the two inequalities in \eqref{bdden} become equalities and $u(s,\cdot)=y$ for all $s\in[0,+\infty)$. The last inequality is equality if and only if $\pi(x)=v$. That the first one is equality is equivalent to $w\in\R^{2n}$ and $u(0,\cdot)=y$ is a critical point of $\Phi_H$ by Proposition \ref{confronto}. Therefore $\pi(x)+w=y$ and by Lemma \ref{crit} again, this shows that $x=y$.
\end{proof}

In this section, we exhibit the functional setting which allows us to see $\mathcal{M}(x,y)$ as the set of zeroes of a nonlinear Fredholm map. 
We consider the following space of $\R^{2n}$-valued maps on the positive half-cylinder converging to $y$ for $s\rightarrow +\infty$ and having the prescribed boundary condition at $s=0$: 
\[
\begin{split}
\mathcal{H}_{x,y} := \big\{ u: (0,+\infty) \times \T \rightarrow \R^{2n} \mid & u - y \in H^1((0,+\infty) \times \T,\R^{2n}), \\ & u(0,\cdot) \in \pi^{-1}\bigl( W^u(\pi(x);-\nabla\psi_{H^*}) \bigr) + \mathbb{H}_{1/2}^- \big\}.
\end{split}
\]
Notice that the tangent space of $M$ at any point has trivial intersection with $\mathbb{H}^-_{1/2}$ and therefore the set 
\[
\pi^{-1}\bigl( W^u(\pi(x);-\nabla\psi_{H^*}) \bigr) + \mathbb{H}_{1/2}^-
\]
is a smooth submanifold of $\mathbb{H}_{1/2}$ having infinite dimension and codimension. The boundary condition at $s=0$ in the definition of $\mathcal{H}_{x,y}$ is well-posed because the trace of an $H^1$ map on the half-cylinder belongs to $\mathbb{H}_{1/2}$. Therefore, $\mathcal{H}_{x,y}$ is a smooth submanifold of the affine Hilbert space
\[
y + H^1((0,+\infty) \times \T,\R^{2n}).
\]
On $\mathcal{H}_{x,y}$ we shall consider the topology and the differentiable structure which is induced by this embedding. This Hilbert manifold is the domain of the map
\[
\overline{\partial}_{J,H}: \mathcal{H}_{x,y} \rightarrow L^2((0,+\infty) \times \T,\R^{2n}), \qquad \overline{\partial}_{J,H} u = \overline{\partial}_J u - \nabla_J H_t(u),
\]
where the Cauchy-Riemann operator $\overline{\partial}_J=\partial_s+J\partial_t$ is to be understood in the distributional sense.
It is easy to check that this map is well-defined, meaning that $\overline{\partial}_{J,H} u$ belongs indeed to $L^2((0,+\infty) \times \T)$. Indeed, if $u=y+u_0$, $u_0\in H^1((0,+\infty) \times \T,\R^{2n})$, is an element of $\mathcal{H}_{x,y}$ we have
\[
 \overline{\partial}_{J,H} u = \overline{\partial}_J u_0 + J(s,t,u)(y' -X_{H_t}(y+u_0)).
\]
Since $\|J\|_\infty<\infty$, the map $\overline{\partial}_J u_0$ belongs to $L^2((0,+\infty) \times \T,\R^{2n})$. The fact that $y$ is a 1-periodic orbit of $X_H$ implies that  $y' - X_{H_t}(y+u_0)$ belongs to $L^2((0,+\infty) \times \T,\R^{2n})$ as well. To see this, we compute
\[
y'-X_{H_t}(y+u_0)=y'-X_{H_t}(y)-\int_0^1\frac{d}{d\theta}X_{H_t}(y+\theta u_0)d\theta=\int_0^1J_0\nabla^2 H_t(y+\theta u_0)u_0\,d\theta,
\]
and obtain the pointwise estimate
\[
|y'(t)-X_{H_t}(y(t)+u_0(s,t))|\leq \|\nabla^2 H\|_{\infty}|u_0(s,t)|
\]
which in particular implies
\[
\|y' - X_{H_t}(y+u_0)\|_{L^2((0,+\infty) \times \T)}^2 \leq \|\nabla^2 H\|_{\infty}^2 \|u_0\|^2_{L^2((0,+\infty) \times \T)}< +\infty.
\]
This shows that $\overline{\partial}_{J,H} u$ belongs to $L^2((0,+\infty) \times \T,\R^{2n})$. The regularity and growth assumptions on $H$ easily imply that the map $\overline{\partial}_{J,H}$ is smooth.

We claim that the set of zeroes of $\overline{\partial}_{J,H}$ coincides with $\mathcal{M}(x,y)$. In order to prove the inclusion $\mathcal{M}(x,y) \subset \overline{\partial}_{J,H}^{-1}(0)$, we have just to prove that every $u\in \mathcal{M}(x,y)$ is also in $\mathcal{H}_{x,y}$. This is true because by the non-degeneracy of $y$ the elements $u$ of $\mathcal{M}(x,y)$ converge to $y$ for $s\rightarrow +\infty$ exponentially fast together with all their derivatives, and in particular $u-y\in H^1((0,+\infty)\times \T)$.
The opposite inclusion follows from the next regularity result instead:

\begin{prop}
\label{bdryreg}
Let $u\in \mathcal{H}_{x,y}$ be such that $\overline{\partial}_{J,H} u = 0$. Then $u$ is smooth on $[0,+\infty)\times \T$ and $u(s,\cdot)\rightarrow y$ for $s\rightarrow +\infty$ in $C^{\infty}(\T,\R^{2n})$.
\end{prop}

The regularity of $u$ on the open half-cylinder $(0,+\infty) \times \T$ does not follow from the standard regularity results in Floer theory (see e.g. \cite[Appendix B.4]{ms04}), because these require the map $u$ to be in $W^{1,p}_{\mathrm{loc}}$ for some $p>2$, or at least in $H^1_{\mathrm{loc}}\cap C^0$ (see \cite[Section 2.3]{is99} or \cite{is00}). However, a different argument implies that interior regularity holds also for solutions of the Floer equation that are just $H^1_{\mathrm{loc}}$. This argument is explained in Appendix \ref{appA}.

The convergence to $y$ in $C^{\infty}(\T,\R^{2n})$ is due to the non-degeneracy of $y$, see e.g.\ \cite[Section 2.7]{sal99}. It remains to prove that $u$ is smooth up to the boundary. The proof of this fact is based on a bootstrap argument which makes use of the following lemmas. In what follows, we omit the subscript in the standard Cauchy-Riemann operator $\overline{\partial}_{J_0}=\overline{\partial}$.

\begin{lem}
\label{formula}
Let $-\infty<a<b<+\infty$ and $u\in H^1((a,b)\times \T,\R^{2n})$. Denote by $\alpha$ and $\beta$ the boundary traces of $u$,
\[
\alpha(t):= u(a,t), \qquad \beta(t):= u(b,t),
\]
which are almost everywhere well-defined functions belonging to $H^{1/2}(\T,\R^{2n})$. Then
\[
\int_{(a,b) \times \T} |\nabla u|^2\, ds dt =  \int_{(a,b) \times \T} |\overline{\partial} u|^2\, ds dt - 2 \int_{\T} \beta^*\lambda_0 + 2 \int_{\T} \alpha^* \lambda_0
\]
and 
\[
\int_{(a,b) \times \T} |\nabla u|^2\, ds dt =  \int_{(a,b) \times \T} |\overline{\partial}^* u|^2\, ds dt + 2 \int_{\T} \beta^*\lambda_0 - 2 \int_{\T} \alpha^* \lambda_0,
\]
where $\overline{\partial}^*=-\partial_s+J_0\partial_t$.
\end{lem}

The simple proof of the lemma is based on Stokes' theorem, see \cite[Lemma 1.1]{as15}.

\begin{lem}
\label{reglemma}
Let $u\in H^1((0,1)\times \T,\R^{2n})\cap C^{\infty}((0,1]\times \T,\R^{2n})$ be such that $\overline{\partial} u$ belongs to $H^1((0,1)\times \T,\R^{2n})$ and 
\[
u(0,\cdot) = v + w,
\]
with $v\in C^{\infty}(\T,\R^{2n})$ and $w\in \mathbb{H}_{1/2}^-$. Then $u$ belongs to $H^2((0,1)\times \T,\R^{2n})$. Moreover,
\begin{equation}
\label{extremis}
\|\nabla^2 u\|_{L^2((0,1)\times \T)} \leq C\bigl(\|\overline{\partial} u\|_{H^1((0,1)\times \T)} +  \|\partial_t u(1,\cdot)\|_{1/2} + \|v'\|_{1/2}\bigr),
\end{equation}
for some $C>0$.
\end{lem}

\begin{proof}
Set
\[
\Delta_h u (s,t) := \frac{u(s,t+h)-u(s,t)}{h},
\]
where $h\in \R$.  The fact that $u$ is in $H^1$ implies that 
\[
\lim_{h\rightarrow 0} \Delta_h u = \partial_t u \ \qquad \mbox{in }\; L^2((0,1)\times \T,\R^{2n}).
\]
Indeed, this follows from the inequality
\[
\begin{split}
\|\Delta_h u - \partial_t u\|_{L^2((0,1)\times \T)}^2 &= \int_{(0,1)\times \T} \left| \int_0^1 \bigl( \partial_t u(s,t+\theta h) - \partial_t u(s,t) \bigr) \, d\theta \right|^2\, ds dt \\ &\leq \int_{(0,1)\times \T} \left( \int_0^1 \bigl| \partial_t u(s,t+\theta h) - \partial_t u(s,t)\bigl|^2 \, d\theta \right) \, ds dt \\ &= \int_0^1 \| T_{\theta h} \partial_t u - \partial_t u\|_{L^2((0,1)\times \T)}^2\, d\theta,
\end{split}
\]
where $T_h$ is the translation operator $T_h u(s,t) = u(s,t+h)$, and from the fact that $T_h \partial_t u \rightarrow \partial_t u$ in $L^2((0,1)\times \T)$ for $h\rightarrow 0$ because $\partial_t u$ belongs to $L^2((0,1)\times \T)$.

Analogously, the function
\[
f:= \overline{\partial} u \in H^1((0,1)\times \T,\R^{2n})\cap C^{\infty}((0,1]\times \T,\R^{2n})
\]
satisfies
\[
\lim_{h\rightarrow 0} \Delta_h f = \partial_t f \ \qquad \mbox{in } L^2((0,1)\times \T).
\]
The fact that $u$ is in $C^{\infty}((0,1]\times \T)$ implies that for every $\epsilon>0$
\[
\lim_{h\rightarrow 0} \Delta_h u = \partial_t u \ \qquad \mbox{in } C^{\infty}([\epsilon,1]\times \T).
\]
In particular,
\[
\nabla \Delta_h u \rightarrow \nabla \partial_t u \qquad \mbox{pointwise in } (0,1]\times \T,
\]
and by the Fatou Lemma
\begin{equation}
\label{controllo}
\|\nabla \partial_t u\|_{L^2((0,1)\times \T)} \leq \liminf_{h\rightarrow 0} \|\nabla \Delta_h u\|_{L^2((0,1)\times \T)}.
\end{equation}
We claim that the right-hand side of this inequality is finite. Indeed, by the identity of Lemma \ref{formula} we have
\[
\begin{split}
\int_{(0,1)\times \T} &|\nabla \Delta_h u|^2\, ds dt = \int_{(0,1)\times \T} |\overline{\partial} \Delta_h u|^2\, ds dt - 2 \int_{\T} (\Delta_h u(1,\cdot))^* \lambda_0 + 2 \int_{\T} (\Delta_h u(0,\cdot))^* \lambda_0 \\ &= \int_{(0,1)\times \T} |\Delta_h f|^2\, ds dt- 2 \int_{\T} (\Delta_h u(1,\cdot))^* \lambda_0 + 2 \int_{\T} (\Delta_h v)^* \lambda_0 + 2 \int_{\T} (\Delta_h w)^* \lambda_0.
\end{split}
\]
The first integral in the last expression converges to the square of the $L^2$ norm of $\partial_t f$. The second and third one converge to the integral of $\partial_t u(1,\cdot)^* \lambda_0$ and $(v')^* \lambda_0$ over $\T$, because the functions $u(1,\cdot)$ and $v$ are smooth. The last integral is non-positive, because $\Delta_h w$ belongs to $\mathbb{H}_{1/2}^-$. We conclude that
\begin{equation}
\label{extremis1}
\limsup_{h\rightarrow 0} \|\nabla \Delta_h u\|^2_{L^2((0,1)\times \T)} \leq \int_{(0,1)\times \T} |\partial_t f|^2\, ds dt- 2 \int_{\T} (\partial_t u(1,\cdot))^* \lambda_0 + 2 \int_{\T} (v')^* \lambda_0 < +\infty,
\end{equation}
and by (\ref{controllo}) the $L^2$ norm of $\nabla \partial_t u$ is finite. Equivalently, the functions $\partial_s \partial_t u$ and $\partial_t^2 u$ have finite $L^2$ norm on $(0,1)\times \T$. From the identity
\begin{equation}
\label{extremis2}
\partial_s^2 u = \partial_s \overline{\partial} u - J_0 \partial_s \partial_t u = \partial_s f -  J_0 \partial_s \partial_t u 
\end{equation}
and the fact that $f$ is in $H^1((0,1)\times \T)$, we deduce that also the $L^2$ norm of $\partial_s^2 u$ is finite. We conclude that $u$ is in $H^2((0,1)\times \T)$.

From (\ref{controllo}) and (\ref{extremis1}) we deduce the bound
\[
\begin{split}
\|\partial_t^2 u\|_{L^2((0,1)\times \T)}^2 + \|\partial_s \partial_t  u\|_{L^2((0,1)\times \T)}^2 &= \|\nabla \partial_t u\|_{L^2((0,1)\times \T)}^2 \\ &\leq \|\partial_t \overline{\partial} u\|_{L^2((0,1)\times \T)}^2 + 2 \|\partial_t u(1,\cdot)\|_{1/2}^2 + 2 \|v'\|_{1/2}^2.
\end{split}
\]
The bound (\ref{extremis}) follows from the above inequality together with (\ref{extremis2}).
\end{proof}

In order to complete the bootstrap argument, we need the following easy consequence of the chain rule.

\begin{lem}
\label{bootstrap}
Let $H\in C^{\infty}(\T \times \R^{2n})$ and $u\in C^{\infty}((0,1) \times \T,\R^{2n})$. Let $h\geq 0$ and $k\geq 0$ be integers with $h+k\geq 1$. Then
\[
\partial_s^h \partial_t^k (\nabla H_t\circ u) = \nabla^2H_t(u)  \partial_s^h \partial_t^k u + p,
\]
where $p$ is a $\R^{2n}$-valued polynomial mapping of the partial derivatives $\partial_s^i \partial_t^j u$ with $0\leq i \leq h$, $0\leq j \leq k$, $1\leq i+j\leq h+k-1$, whose coefficients are of the form $A(t,u(s,t))$, where $A$ is smooth.
\end{lem} 

\begin{proof}
We argue by induction on $h+k$. If $h+k=1$, then either $h=1$ and $k=0$ or $h=0$ and $k=1$. In the first case we find
\[
\partial_s (\nabla H_t\circ u) = \nabla^2 H_t(u) \partial_s u,
\]
so the desired conclusion holds with $p=0$. In the second case we have
\[
\partial_t (\nabla H_t\circ u) = \nabla^2 H_t(u) \partial_t u + \nabla (\partial_t H_t) (u),
\]
so the desired conclusion holds with $p$ being the polynomial map of degree 0 
\[
p = A(t,z) := \nabla (\partial_t H_t) (z) \qquad \forall (t,z)\in \T \times \R^{2n},
\]
which is indeed smooth.

Now we assume that the thesis is true for all integers $h\geq 0$ and $k\geq 0$ with $1 \leq h+ k \leq \ell$,  for a given $\ell\geq 1$. Let $h'\geq 0$ and $k'\geq 0$ be integers with $h'+k'=\ell+1$. Our aim is to show that $\partial_s^{h'} \partial_t^{k'} (\nabla H_t\circ u) $ has the desired form.

We first assume that $k'\leq \ell$, so that $h'\geq 1$. Then the inductive assumption implies that
\[
\partial_s^{h'-1} \partial_t^{k'} (\nabla H_t\circ u) = \nabla^2 H_t(u)  \partial_s^{h'-1} \partial_t^{k'} u + p,
\]
where $p$ is a polynomial map of the partial derivatives $\partial_s^i \partial_t^j u$ with $0\leq i \leq h'-1$, $0\leq j \leq k'$, $1\leq i+j\leq h'+k'-2$, whose coefficients are of the form $A(t,u(t))$, where $A$ is smooth. By differentiating the above identity with respect to $s$ we obtain
\[
\partial_s^{h'} \partial_t^{k'} (\nabla H_t\circ u) = \nabla^2 H_t(u)  \partial_s^{h'} \partial_t^{k'} u + \nabla^3 H_t(u) [\partial_s u, \partial_s^{h'-1} \partial_t^{k'} u] + \partial_s p.
\]
The middle term on the right-hand side is a bilinear map in $\partial_s u$ and $\partial_s^{h'-1} \partial_t^{k'} u$, which are partial derivatives of $u$ of order not exceeding $h'+k'-1$, with coefficient of the form $A(t,u(s,t))$, where
\[
A(t,z) := \nabla^3 H_t(z) \qquad \forall (t,z)\in \T \times \R^{2n},
\]
is smooth. When we differentiate $p$ with respect to $s$ we obtain terms of two kinds. The terms of the first kind are obtained by differentiating a given coefficient $A(t,u(s,t))$ with respect to $s$. This produces a term of the form $\nabla A(t,u) \partial_s u$. This term is multilinear in the set of partial derivatives of $u$ of admissible order. Such a term has the required form. The terms of the second kind are obtained by differentiating with respect to $s$ a given partial derivative of $u$. This operation produces a monomial in which a term of the form $\partial_s^i \partial_t^j u$, with $0\leq i \leq h'-1$, $0\leq j \leq k'$, $1\leq i+j\leq h'+k'-2$, is replaced by $\partial_s^{i+1} \partial_t^j u$. Since $i+1\leq h'$, the new monomial satisfies the required conditions. We conclude that $\partial_s p$ is a polynomial map of the partial derivatives $\partial_s^i \partial_t^j u$ with $0\leq i \leq h'$, $0\leq j \leq k'$, $1\leq i+j\leq h'+k'-1$, whose coefficients are of the required form.

There remains to consider the case $k'=\ell+1$, which implies that $h'=0$. By the inductive assumption we have
\[
\partial_t^{\ell}  (\nabla H_t\circ u) = \nabla^2 H_t(u) \partial_t^{\ell}  u + p,
\]
where $p$ is a polynomial map in $\partial_t u,\partial_t^2 u, \dots,\partial_t^{\ell-1} u$, whose coefficients have the required form. Differentiation with respect to $t$ gives 
\[
\partial_t^{\ell+1}  (\nabla H_t\circ u) = \nabla^2 H_t(u) \partial_t^{\ell+1}  u + \nabla^3 H_t(u)[\partial_t u, \partial_t^{\ell} u] + \nabla^2 (\partial_t H_t) (u) \partial_t^{\ell} u + \partial_t p.
\]
The maps
\[
(t,z) \mapsto \nabla^3 H_t(z) \qquad (t,z) \mapsto \nabla^2 (\partial_t H_t) (z)
\]
are smooth. Moreover, an argument analogous to the previous one shows that $\partial_t p$ is  a polynomial map of the partial derivatives $\partial_t^j u$ with $1\leq j \leq \ell$, whose coefficients are of the required form. 

This proves that in both cases $\partial_s^{h'} \partial_t^{k'} (\nabla H_t\circ u)$ has the desired form and concludes the proof of the induction step.
\end{proof}

\begin{proof}[Proof of Proposition \ref{bdryreg}]
We can now conclude the proof of Proposition \ref{bdryreg} by showing that $u$ is smooth up to the boundary. By the Sobolev embedding theorems, it is enough to prove that the restriction of $u$ to $(0,1)\times \T$ belongs to $H^k((0,1)\times \T,\R^{2n})$ for every natural number $k$. It certainly belongs to $H^1((0,1)\times \T,\R^{2n})$ by the definition of $\mathcal{H}_{x,y}$. Note that on $(0,1)\times\T$, $J=J_0$ and the equation $\overline{\partial}_{J,H}u=0$ simplifies to 
\begin{equation}
\label{floern}
\overline{\partial} u = \nabla H_t(u).
\end{equation}
Since the Hessian of $H_t$ is globally bounded, the maps
\[
\partial_s(\nabla H_t(u)) = \nabla^2 H_t(u) \partial_s u, \qquad  \partial_t(\nabla H_t(u)) = \nabla^2 H_t(u) \partial_t u + \nabla (\partial_t H_t)(u),
\]
belong to $L^2((0,1)\times \T,\R^{2n})$. Therefore, the right-hand side of (\ref{floern}) belongs to $H^1((0,1)\times \T,\R^{2n})$. Thanks to the boundary conditions satisfied by $u$, Lemma \ref{reglemma} implies that  the restriction of $u$ to $(0,1)\times \T$ belongs to $H^2((0,1)\times \T,\R^{2n})$. In particular, $u$ extends continuously to the closed half-cylinder $[0,+\infty) \times \R$ and is globally bounded.

Arguing by induction, we assume that the restriction of $u$ to $(0,1)\times \T$ belongs to $H^k((0,1)\times \T,\R^{2n})$ for some integer $k\geq 2$ and need to show that it belongs to $H^{k+1}((0,1)\times \T,\R^{2n})$. By differentiating \eqref{floern} $k-1$ times with respect to $t$ we obtain, thanks to Lemma \ref{bootstrap}:
\begin{equation}
\label{difffloer}
\overline{\partial} \partial_t^{k-1} u = \nabla^2 H_t(u)  \partial_t^{k-1} u  + p,
\end{equation}
where $p$ is a $\R^{2n}$-valued polynomial mapping of the partial derivatives $\partial_t u,\dots,\partial_t^{k-2} u$ whose coefficients are of the form $A(t,u(t))$, where $A$ is smooth. By the inductive assumption, the function $ \partial_t^{k-1} u$ belongs to $H^1((0,1)\times \T)\cap C^{\infty}((0,1]\times \T)$. Therefore, its trace at $s=0$ is in $\mathbb{H}_{1/2}$ and, since
\[
u(0,\cdot) \in \pi^{-1}(W^u(x;-\nabla\psi_{H^*})) + \mathbb{H}_{1/2}^-
\]
where the first set consists of smooth loops by Proposition \ref{prop:reduction}, it has the form
\[
\partial_t^{k-1} u(0,\cdot)= v + w,
\]
where $v$ is a smooth loop and $w$ is an element of $\mathbb{H}_{1/2}$ which is the $(k-1)$-th derivative of an element of $\mathbb{H}_{1/2}^-$. As such, $w$ also belongs to $\mathbb{H}_{1/2}^-$.

By differentiating the right-hand side of (\ref{difffloer}) we get
\begin{equation}
\label{dueder}
\begin{split}
\partial_s (\nabla^2 H_t(u)  \partial_t^{k-1} u  + p) &= \nabla^2 H_t(u)  \partial_t^{k-1} \partial_s u + q_1, \\
\partial_t(\nabla^2 H_t(u)  \partial_t^{k-1} u  + p) &= \nabla^2 H_t(u)  \partial_t^{k} u + q_2,
\end{split}
\end{equation}
where 
\[
\begin{split}
q_1 &:= \nabla^3 H_t(u)[\partial_s u] \partial_t^{k-1} u + \partial_s p, \\
q_2 &:= \nabla^2 (\partial_t H)(u) \partial_t^{k-1} u + \nabla^3 H_t(u)[\partial_t u] \partial_t^{k-1} u + \partial_t p,
\end{split}
\]
are $\R^{2n}$-valued polynomial mappings of the partial derivatives 
\[
\partial_t u,\dots,\partial_t^{k-1} u,\partial_s u,\partial_s\partial_t u,\dots,\partial_s\partial_t^{k-2}u
\] 
whose coefficients are of the form $A(t,u(t))$, where $A$ is smooth. Note that the coefficients $A(t,u(s,t))$ are uniformly bounded since $u$ is bounded as observed above. The function $\nabla^2 H_t(u)  \partial_t^{k-1} \nabla u$ is in $L^2((0,1)\times \T)$ because of the inductive assumption and the boundedness of $\nabla^2 H$. The polynomial mappings $q_1$ and $q_2$ have the pointwise estimate
\[
|q_j|^2 \leq C \bigl( 1 + |\partial_t u|^N + \dots + |\partial_t^{k-1} u|^N +  |\partial_s u|^N + |\partial_s\partial_t u|^N + \dots + |\partial_s\partial_t^{k-2}u|^N \bigr), \qquad j=1,2,
\]
on $(0,1)\times\T$ for a suitable positive number $C$ and a suitable natural number $N$. Thanks to the Sobolev embedding of $H^k((0,1)\times \T)$ into $W^{k-1,N}((0,1)\times \T)$, all the partial derivatives which appear in the right-hand side of the above estimate are in $L^N((0,1)\times \T)$ and hence $q_1$ and $q_2$ are in $L^2((0,1)\times \T)$. From \eqref{dueder}, we conclude that the right-hand side of (\ref{difffloer}) is in $H^1((0,1)\times \T)$. Then we can apply Lemma \ref{reglemma} to the function $\partial_t^{k-1} u$ and we obtain that this function belongs to $H^2((0,1)\times \T)$, which means that the functions $\partial_s^2 \partial_t^{k-1} u$, $\partial_s \partial_t^k u$ and $\partial_t^{k+1}u$ are in $L^2((0,1)\times \T)$.

The fact that $u$ solves the equation (\ref{floern}) easily implies that all other derivatives of order $k+1$ of $u$ are in $L^2((0,1)\times \T)$. Indeed, by applying the differential operator $\partial_s^2 \partial_t^{k-2}$ to (\ref{floern}) we obtain, thanks to Lemma \ref{bootstrap},
\[
\partial_s^3 \partial_t^{k-2} u = - J_0 \partial_s^2 \partial_t^{k-1} u + \nabla^2 H_t(u) \partial_s^2 \partial_t^{k-2} u + r,
\]
where $r$ is a $\R^{2n}$-valued polynomial mapping of the partial derivatives 
\[
\partial_t u,\dots,\partial_t^{k-2} u,\partial_s u,\partial_s \partial_t u, \dots, \partial_s \partial_t^{k-2} u,\partial_s^2u, \partial_s^2 \partial_t u, \dots, \partial_s^2 \partial_t^{k-3} u
\]
whose coefficients are of the form $A(t,u(t))$, where $A$ is smooth. Arguing as above, we deduce that $\partial_s^3 \partial_t^{k-2} u$ belongs to $L^2((0,1)\times \T)$. Iteratively, we conclude that all derivatives of order $k+1$ of $u$ belong to $L^2((0,1)\times \T)$ and hence $u$ is in $H^{k+1}((0,1)\times \T)$, as we wished to prove.
\end{proof}

\section{The Fredholm index of the hybrid problem}
\label{fredsec}

We continue to assume that $H\in C^{\infty}(\T\times \R^{2n})$ is non-degenerate, quadratically convex and non-resonant at infinity. Let $J$ be as before a family of uniformly bounded $\omega_0$-compatible almost complex structures smoothly parametrized by $[0,+\infty)\times\T$, such that $J(s,t)=J_0$ for all $s\in [0,1]$ and $J(s,t)$ is independent of $s$ for all $s$ large. The aim of this section is to prove the following result.

\begin{prop}
\label{fredholm}
Let $x$ and $y$ be 1-periodic orbits of $X_H$. For every $u\in \mathcal{M}(x,y)$, the linear operator 
\[
D\overline{\partial}_{J,H}(u): T_u \mathcal{H}_{x,y} \longrightarrow L^2((0,+\infty)\times\T,\R^{2n})
\]
is Fredholm of index $\mu_{CZ}(x) - \mu_{CZ}(y)$.
\end{prop}

Denote by
\[
\mathrm{tr}_0 : H^1((0,+\infty)\times \T,\R^{2n}) \rightarrow \mathbb{H}_{1/2}
\]
the trace operator, i.e.\ the unique continuous extension of the operator mapping each smooth map $u:[0,+\infty) \times \T \rightarrow \R^{2n}$ with compact support to its restriction $u(0,\cdot)$ to the boundary of the half cylinder. As we already did in Section \ref{funsetsec}, we will sometimes use the notation $u(0,\cdot)$ to denote $\mathrm{tr}_0 u$.

Given a closed linear subspace $V$ of $\mathbb{H}_{1/2} = H^{1/2}(\T,\R^{2n})$, we define
\[
H^1_V((0,+\infty)\times \T,\R^{2n}) := \{ u\in H^1((0,+\infty)\times \T,\R^{2n}) \mid \mathrm{tr}_0 u \in V \}.
\]
The proof of the above result relies on the following proposition, which is a slightly more general version of \cite[Theorem 4.4]{hec13}.

\begin{prop}
\label{linop}
Let $A\in C^0([0,+\infty]\times \T, L(\R^{2n}))$ be a continuous map into the space of  linear endomorphisms of $\R^{2n}$. Let $J\in C^0([0,+\infty]\times \T, L(\R^{2n}))$ be a continuous map such that 
$J(s,t)$ is an $\omega_0$-compatible almost complex structure for every $(s,t)\in[0,+\infty]\times\T$ and $J(s,t)=J_0$ for every $s\in[0,1]\times\T$. We assume that $-J_0J(+\infty,t)A(+\infty,t)$ is symmetric for every $t\in \T$. Denote by $Z_A:[0,1] \rightarrow \mathrm{Sp}(2n)$ the symplectic path which is defined by
\[
Z_A(0) = I, \qquad Z_A'(t) = - J(+\infty,t) A(+\infty,t) Z_A (t) \quad \forall t\in [0,1],
\]
and assume that $Z_A$ is non-degenerate, meaning that 1 is not an eigenvalue of $Z_A(1)$. So $Z_A\in\mathcal{SP}(2n)$ and its Conley-Zehnder index is denoted by $\mu_{CZ}(Z_A)$. Let $V$ be a closed linear subspace of  $\mathbb{H}_{1/2}$ which is a compact perturbation of $\mathbb{H}_{1/2}^-$. Then the linear operator
\[
T:H^1_V((0,+\infty)\times \T,\R^{2n}) \rightarrow L^2((0,+\infty)\times \T,\R^{2n}), \quad u \mapsto \partial_s u + J \partial_t u - A u,
\]
is Fredholm of index
\[
\ind T=\mathrm{dim}(V,\mathbb{H}^-_{1/2} \oplus E^-) - \mu_{CZ}(Z_A).
\]
\end{prop}

\begin{proof}[Proof of Proposition \ref{fredholm}]
Let $u\in \mathcal{M}(x,y)$ and write
\[
u(0,\cdot) = v + w,
\]
where $v\in W^u(\pi(x);-\nabla\psi_{H^*})$, seen as a loop with zero mean, and $w\in \R^{2n} \oplus \mathbb{H}_{1/2}^-$. The tangent space of $\mathcal{H}_{x,y}$ at $u$ is
\[
T_u \mathcal{H}_{x,y} = H^1_V((0,+\infty)\times \T,\R^{2n}),
\]
where 
\[
V := T_v W^u(\pi(x);-\nabla\psi_{H^*}) \oplus \R^{2n} \oplus \mathbb{H}_{1/2}^-
\]
is  a closed linear subspace of $\mathbb{H}_{1/2}$. This subspace is clearly a compact perturbation of $\mathbb{H}_{1/2}^-$ and
\[
\begin{split}
\dim ( V, \mathbb{H}_{1/2}^- \oplus E^-) &= \dim T_v W^u(\pi(x);-\nabla\psi_{H^*}) + \dim ( \R^{2n} \oplus \mathbb{H}_{1/2}^-, \mathbb{H}_{1/2}^- \oplus E^-) \\ &= \mathrm{ind}(\pi(x);\Psi_{H^*}) + n.
\end{split}
\]
Together with Propositions \ref{relind=ind} and \ref{cz=relind}, we find
\[
\dim ( V, \mathbb{H}_{1/2}^- \oplus E^-) = \mathrm{ind}_{\mathbb{H}^-_{1/2}\oplus E^-}(x;\Phi_H) -n + n = \mu_{CZ}(x).
\]
The differential of $\overline{\partial}_{J,H}$ at $u$ is an operator of the form considered 
in Proposition \ref{linop}, that is
\[
D \overline{\partial}_{J,H}(u) : H^1_V((0,+\infty)\times \T,\R^{2n}) \to L^2((0,+\infty)\times \T,\R^{2n}), \quad v\mapsto  \partial_s v + J(s,t,u) \partial_t v - A(s,t) v,
\]
where 
\[
A(s,t) v = J(s,t,u)\nabla_v X_{H_t}(u(s,t)) - \nabla_v J(s,t,u)\big(\partial_tu(s,t)-X_{H_t}(u(s,t))\big)
\]
and 
\[
Z_A (t) = d\phi_{X_H}^t(y(0)).
\] 
Thanks to the above computation of the relative dimension of $V$ with respect to $\mathbb{H}_{1/2}^- \oplus E^-$, Proposition \ref{linop} implies that this operator is Fredholm of index
\[
\mu_{CZ}(x) - \mu_{CZ}(y).
\]
\end{proof}

\begin{proof}[Proof of Proposition \ref{linop}] 
Let $\Lambda(s,t)$ be a family of endomorphisms of $\R^{2n}$ smoothly depending on $(s,t)\in[0,+\infty]\times \T$ such that 
\[
\Lambda(s,t)^*\omega_0=\omega_0,\qquad \Lambda(s,t)^*J(s,t)=J_0,\qquad \Lambda (s,t)=I\quad \mbox{for all }\;(s,t)\in[0,1]\times\T.
\]
Then $T$ conjugates to an operator $T_\Lambda$ by $\Lambda$ of the form 
\[
T_\Lambda:H^1_V((0,+\infty)\times \T,\R^{2n}) \rightarrow L^2((0,+\infty)\times \T,\R^{2n}), \quad u \mapsto \partial_s u + J_0 \partial_t u - A_\Lambda u
\]
where 
\[
A_\Lambda:=\Lambda^{-1}(A\Lambda-\partial_s\Lambda-J\partial_t\Lambda):[0,+\infty]\times\T\to L(\R^{2n}). 
\]
To show that $T_\Lambda$ is Fredholm, we choose smooth cut-off functions $\beta_0$, $\beta_1$, $\beta_2:[0,+\infty)\to[0,1]$ such that 
\[
\supp \beta_0\subset [0,1],\quad \supp\beta_1\in [1/2,\tau-1],\quad \supp\beta_2\subset [\tau-2,+\infty),\quad \beta_1+\beta_2+\beta_3=1
\]
where $\tau>0$ is determined below. For $u\in H^1_V((0,+\infty)\times \T,\R^{2n})$, let $u_0=\beta_0 u$, $u_1=\beta_1 u$, and $u_2=\beta_2 u$.

Since $u_0$ has support in $[0,1]\times\T$, using the first identity in Lemma \ref{formula} we estimate 
\[
\begin{split}
\|\nabla u_0\|^2_{L^2((0,+\infty)\times\T)}&=\|\overline{\partial}u_0\|^2_{L^2((0,+\infty)\times\T)}+2\int_\T u(0,\cdot)^*\lambda_0\\
&\leq\|\overline{\partial}u_0\|^2_{L^2((0,+\infty)\times\T)} + 2\int_\T\big(P^+u(0,\cdot)\big)^*\lambda_0\\  
&=\|\overline{\partial}u_0\|^2_{L^2((0,+\infty)\times\T)} + \|P^+ \mathrm{tr}_0 u_0\|^2_{1/2}.
\end{split}
\]
This implies
\begin{equation}\label{estimate1}
\|u_0\|_{H^1((0,+\infty)\times\T)}\leq c_0\big(\|T_\Lambda u_0\|_{L^2((0,+\infty)\times\T)}+\|u_0\|_{L^2((0,1)\times\T)}+\|P^+ \mathrm{tr}_0 u_0\|_{1/2}\big)
\end{equation}
where $c_0=1+\|A_\Lambda\|_{L^\infty}$. 

The estimates for $u_1$ and $u_2$ below are standard. Applying the Calderon-Zygmund estimate to $u_1$, we know that there exists $c_1>0$ such that 
\[
\|u_1\|_{H^1((0,+\infty)\times\T)}\leq c_1\big(\|T_\Lambda u_1\|_{L^2((0,+\infty)\times\T)}+\|u_1\|_{L^2((0,\tau)\times\T)}\big).
\]
Since $Z_A$ is non-degenerate, for sufficiently large $\tau$, there exists $c_2>0$ such that 
\[
\|u_2\|_{H^1((0,+\infty)\times\T)}\leq c_2\|T_\Lambda u_2\|_{L^2((0,+\infty)\times\T)}.
\]
Putting the above three estimates together, we obtain $c_3>0$ such that 
\begin{equation}
\label{CZ_bdry}
\|u\|_{H^1((0,+\infty)\times\T,\R^{2n})}\leq c_3\big(\|T_\Lambda u\|_{L^2((0,+\infty)\times\T,\R^{2n})}+\|u\|_{L^2((0,\tau)\times\T,\R^{2n})}+\|P^+ \mathrm{tr}_0 u\|_{1/2}\big).
\end{equation}
Since $V$ is a compact perturbation of $\mathbb{H}^-_{1/2}$, 
\[
P^+|_V:V\to\mathbb{H}_{1/2}
\]
and hence $P^+|_V\circ \mathrm{tr}_0$ is a compact operator. A standard argument using \eqref{CZ_bdry} shows that $T_\Lambda$ is semi-Fredholm. To conclude that $T_\Lambda$ is Fredholm, we  analyze the formal adjoint operator $T_\Lambda^*$ which enjoys the property $\ker T_\Lambda^*=\coker T_\Lambda$. In view of the computation 
\[
\int_{(0,+\infty)\times\T} v\cdot T_\Lambda u\,dsdt=\int_{(0,+\infty)\times\T} (-\partial_s v+J_0\partial_t v-A_\Lambda^*)\cdot u\,dsdt -\int_{\{0\}\times\T} v\cdot u\,dt
\]
for a compactly supported smooth map $v:[0,+\infty)\times\T\to\R^{2n}$, where we have used integration by parts, it is given by
\[
T_\Lambda^*:H^1_{V^\perp}((0,+\infty)\times\T,\R^{2n})\to L^2((0,+\infty)\times\T,\R^{2n}),\qquad  -\partial_s + J_0\partial_t - A_\Lambda^*,
\]
where $A_\Lambda^*$ denotes the transpose of $A_\Lambda$. Here $V^\perp$ is the orthogonal complement of $V$ with respect to the $L^2$-metric which  is a compact perturbation of $\mathbb{H}^+_{1/2}$. Arguing as above, one can readily see that $T_\Lambda^*$ also satisfies an inequality like \eqref{CZ_bdry} with $P^+$ replaced by $P^-$ using the second identity in Lemma \ref{formula}. This proves that $T_\Lambda^*$ is also semi-Fredholm and consequently $T_\Lambda$ is Fredholm. 
\medskip

In order to compute the Fredholm index of $T$, we homotope $T$ to another Fredholm operator in two steps and use the following well-known facts. First, the Fredholm index is locally constant in the space of Fredholm operators and therefore the index does not change along a continuous homotopy of Fredholm operators. Second, two paths of symplectic matrices in $\mathcal{SP}(2n)$ lie in the same connected component if and only if their Conley-Zehnder indices coincide.

We choose a continuous path of $\omega_0$-compatible almost complex structures $\{J_r\}_{r\in[0,1]}$ such that $J_0(s,t)=J_0$ and $J_1(s,t)=J(s,t)$ for all $(s,t)\in[0,+\infty]\times\T$ and $J_r(s,t)=J_0$ for all $(r,s,t)\in[0,1]\times[0,1]\times\T$, which exists due to the contractibility of the space of $\omega_0$-compatible almost complex structures. We set
\[
A_r(s,t)=-J_r(s,t)J(s,t)A(s,t)\in L(\R^{2n})
\]
and consider a continuous family of  operators 
\[
T_r:H^1_V((0,+\infty)\times \T,\R^{2n}) \rightarrow L^2((0,+\infty)\times \T,\R^{2n}), \qquad v\mapsto  \partial_s v + J_r(s,t) \partial_t v - A_r(s,t) v,
\]
such that $T_1=T$. Since its asymptotic operator at $s=+\infty$ being of the form 
\[
H^1_V(\T,\R^{2n}) \rightarrow L^2(\T,\R^{2n}),\qquad v\mapsto J_r(+\infty,t)(\partial_tv+J(+\infty,t)A(+\infty,t)v)
\]
is invertible, the arguments as above show that $T_r$ is Fredholm for all $r\in[0,1]$.  Therefore we have $\ind T=\ind T_1=\ind T_0$. Moreover, the symplectic path $Z_{A_r}(t)$ defined by
\[
Z_{A_r}(0) = I, \qquad Z_{A_r}'(t) = - J_r(+\infty,t) A_r(+\infty,t) Z_{A_r} (t) \quad \forall t\in [0,1],
\]
is independent of $r$ which yields that the Conley-Zehnder index $\mu_{CZ}(Z_{A_r})$ does not change in $r$. It remains to show 
\begin{equation}\label{index_T_0}
\ind T_0=\dim(V,\mathbb H^-_{1/2}\oplus E^-)-\mu_{CZ}(Z_A).
\end{equation}
\medskip

\noindent{\bf Case 1}:~The Conley-Zehnder index $\mu_{CZ}(Z_A)$ is odd.\\[-2ex]

We pick any number $\theta\in\R\setminus 2\pi\Z$ satisfying 
\begin{equation}
\label{CZ=theta}
\mu_{CZ}(Z_A)=2\left\lfloor\frac{\theta}{2\pi}\right\rfloor+1
\end{equation}
and define a symmetric matrix 
\[
A_{odd}=
  \begin{pmatrix}
    \theta & 0 \\
    0 & \theta
  \end{pmatrix}
\oplus 
\begin{pmatrix}
    -1 & 0 \\
    0 & 1
  \end{pmatrix}^{\oplus n-1}.
\]
The associated non-degenerate path of symplectic matrices
\[
Z_{A_{odd}}=  
e^{-tJ_0\theta}
\oplus 
\begin{pmatrix}
    \cosh t & \sinh t \\
    \sinh t & \cosh t
  \end{pmatrix}^{\oplus n-1}
\]
satisfies $\mu_{CZ}(Z_{A_{odd}})=\mu_{CZ}(Z_A)$ since $\begin{pmatrix}
    -1 & 0 \\
    0 & 1
  \end{pmatrix}$ has zero signature and thus do not contribute to the Conley-Zehnder index. Then we choose another continuous family 
\[
A_r:[0,+\infty]\times\T\to L(\R^{2n}) ,\qquad r\in[-1,0]
\]
such that 
\begin{enumerate}
\item[-] $A_0(s,t)=A(s,t)$, $A_{-1}(s,t)=A_{odd}(t)$ for all $(s,t)\in[0,+\infty]\times\T$,
\item[-] $A_r(+\infty,t)$ is symmetric  for all $t\in\T$ and $r\in[-1,0]$,
\item[-] $Z_{A_r}\in\mathcal {SP}(2n)$ for all $r\in[-1,0]$.
\end{enumerate}
Due to the last property, we have a continuous family of Fredholm  operators 
\[
T_r=\partial_s+J_0\partial_t-A_r:H^1_V((0,+\infty)\times \T,\R^{2n}) \rightarrow L^2((0,+\infty)\times \T,\R^{2n}).
\]
In particular, $\mathrm{ind}\,T_0=\mathrm{ind}\,T_{-1}$.  Since the Fredholm index, the relative dimension, the Conley-Zehnder index are additive under direct sum, it is enough to establish \eqref{index_T_0} for
\[
T_\theta=\partial_s+J_0\partial_t-  \theta I:H^1_V((0,+\infty)\times \T,\R^{2}) \rightarrow L^2((0,+\infty)\times \T,\R^{2}),
\]
and for 
\[
T_Q=\partial_s+J_0\partial_t-Q
  :H^1_V((0,+\infty)\times \T,\R^{2}) \rightarrow L^2((0,+\infty)\times \T,\R^{2})
\]
where $I$ denotes the identity map on $\R^2$, $Q=\begin{pmatrix}
    -1 & 0 \\
    0 & 1
  \end{pmatrix}$, and $V$ is now a closed subspace of $\mathbb H_{1/2}=H^{1/2}(\T,\R^2)$. Since $V$ is a compact perturbation of $\mathbb{H}^-_{1/2}$, the spaces $V^\perp\cap\mathbb{H}^-_{1/2}$ and $V\cap(\R^{2}\oplus\mathbb{H}^+_{1/2})$ are finite dimensional. In other words, there exist $\ell\in\N$, a subspace $W_k$ of $\R^2e^{-2\pi kJ_0t}$ for $-\ell\leq k\leq\ell$, and a closed subspace $W^-$ of $\mathbb{H}^-_{1/2}$ such that 
\[
\mathbb{H}^-_{1/2}=W^-\oplus\bigoplus_{-\ell\leq k\leq -1}\R^2e^{-2\pi kJ_0t}
\]
and
\[
V= W^-\oplus \bigoplus_{-\ell\leq k\leq \ell} W_k.
\]
Then we have 
\[
V^\perp=W^+\oplus\bigoplus_{-\ell\leq k\leq \ell} W_k^\perp
\]
where $W^+$ is a closed subspace of $\mathbb{H}^+_{1/2}$ such that 
\[
\mathbb{H}^+_{1/2}=W^+\oplus\bigoplus_{1\leq k\leq\ell}\R^2e^{-2\pi kJ_0t}
\]
and $W^\perp_k$ is the orthogonal complement of $W_k$ in $\R^2e^{-2\pi kJ_0t}$.

We first compute the index of $T_\theta$. Let $u\in H^1_V((0,+\infty)\times \T,\R^{2n})$. We represent $u$ as its Fourier series
\[
u(s,t)=\sum_{k\in\Z}e^{-2\pi kJ_0t}\hat u_k(s),\qquad \hat u_k(s)\in\R^2.
\]
Then the condition $u\in\ker T_\theta$ implies 
\[
0=T_\theta u=\partial_su+J_0\partial_tu-\theta u=\sum_{k\in\Z}e^{-2\pi kJ_0t}\big(\hat u'_k(s)+(2\pi k-\theta)\hat u_k(s)\big),
\]
from which we deduce 
\[
\hat u_k(s)=e^{-(2\pi k -\theta)s}\hat u_k(0).
\]
The condition that $u$ has finite $H^1$-norm translates to  
\[
\hat u_k(0)=0,\qquad \forall k\leq\frac{\theta}{2\pi}.
\]
Next we study the boundary condition $u(0,\cdot)\in V$. Let us consider the case $\theta<-2\pi\ell$. In this case, $u\in\ker T_\theta$ if and only if
\[
u(0,\cdot)\in \bigoplus_{k=\lceil\frac{\theta}{2\pi}\rceil}^{-\ell-1}\R^2e^{-2\pi kJ_0t}\oplus\bigoplus_{k=-\ell}^\ell W_k
\]
and hence,
\[
\dim\ker T_\theta=-2\left(\left\lceil\frac{\theta}{2\pi}\right\rceil+\ell\right)+\sum_{k=-\ell}^\ell \dim W_k.
\]
If $\theta>-2\pi\ell$, $u\in\ker T_\theta$ exactly when 
\[
u(0,\cdot)\in\bigoplus_{k=\lceil\frac{\theta}{2\pi}\rceil}^\ell W_k
\]
and thus,
\[
\dim\ker T_\theta=\sum_{k=\lceil\frac{\theta}{2\pi}\rceil}^\ell\dim W_k.
\]
In particular if $\theta>2\pi\ell$, then $\dim\ker T_\theta=0$. 

To compute the dimension of the cokernel of $T_\theta$, we use its formal adjoint operator 
\[
T_\theta^*:H^1_{V^\perp}((0,+\infty)\times\T,\R^{2n})\to L^2((0,+\infty)\times\T,\R^{2n}),\quad u\mapsto -\partial_s u+J_0\partial_tu-\theta u.
\]
Arguing as above we can show that if $u\in\ker T_\theta^*$, 
\[
u(s,t)=\sum_{k\in\Z}e^{-2\pi kJ_0t}e^{(2\pi k-\theta)s}\hat u_k(0).
\]
Since $u\in H^1_{V^\perp}((0,+\infty)\times\T,\R^{2n})$, we have $\hat u_k(0)=0$ for all $k\geq\frac{\theta}{2\pi}$ and $u(0,\cdot)\in V^\perp$. If $\theta<2\pi\ell$, $u\in\ker T_\theta^*$ is equivalent to
\[
u(0,\cdot)\in\bigoplus_{k=-\ell}^{\lfloor\frac{\theta}{2\pi}\rfloor}W_k^\perp
\]
and this computes
\[
\dim\ker T_\theta^*=\sum_{k=-\ell}^{\lfloor\frac{\theta}{2\pi}\rfloor}\dim W_k^\perp.
\]
In particular if $\theta<-2\pi\ell$, $\dim\ker T_\theta^*=0$. In the case of $\theta>2\pi\ell$, $u\in\ker T_\theta^*$ if and only if
\[
u(0,t)\in \bigoplus_{k=\ell+1}^{\lfloor\frac{\theta}{2\pi}\rfloor} \R^2e^{-2\pi kJ_0t}\oplus\bigoplus_{k=-\ell}^\ell W_k^\perp
\]
and therefore,
\[
\dim\ker T_\theta^*=2\left(\left\lfloor\frac{\theta}{2\pi}\right\rfloor-\ell\right)+\sum_{k=-\ell}^{\ell}\dim W_k^\perp.
\]

Using the identities that $\dim W^\perp_k=2-\dim W_k$ and $\lceil\frac{\theta}{2\pi}\rceil=\lfloor\frac{\theta}{2\pi}\rfloor+1$, we see that in all the cases
\begin{equation}
\label{indexT_theta}
\ind T_\theta=\dim \ker T_\theta-\dim\ker T_\theta^*=-2\left(\left\lceil\frac{\theta}{2\pi}\right\rceil+\ell\right)+\sum_{k=-\ell}^\ell \dim W_k.
\end{equation}
On the other hand, we have 
\begin{equation}\label{rel_dim}
\begin{split}
\mathrm{dim}(V,\mathbb{H}^-_{1/2} \oplus E^-)&=\dim\big(V\cap(\mathbb{H}^+_{1/2}\oplus E^+)\big)-\dim\big(V^\perp\cap(\mathbb{H}^-_{1/2}\oplus E^-)\big)\\
&=\dim(W_0\cap E^+)+\sum_{k=1}^\ell\dim W_k - \dim (W_0^\perp\cap E^-) - \sum_{k=-\ell}^{-1} \dim W_k^\perp\\
&=\dim W_0-1+\sum_{k=1}^\ell\dim W_k -\sum_{k=-\ell}^{-1} (2-\dim W_{k})\\
&=\sum_{k=-\ell}^\ell\dim W_k-2\ell-1.
\end{split}
\end{equation}
Combining \eqref{CZ=theta}, \eqref{indexT_theta}, and \eqref{rel_dim}, we conclude
\[
\ind T_\theta=\mathrm{dim}(V,\mathbb{H}^-_{1/2} \oplus E^-)-\mu_{CZ}(Z_\theta).
\]
\medskip

To compute the index of $T_Q$, where $Q=\begin{pmatrix}
    -1 & 0 \\
    0 & 1
  \end{pmatrix}$,
we write as before $u\in H^1_V((0,+\infty)\times \T,\R^{2n})$ as 
\[
u(s,t)=\sum_{k\in\Z}e^{-2\pi kJ_0t}\hat u_k(s),\qquad \hat u_k(s) = \big(\hat a_k(s),\hat b_k(s)\big) \in\R^2=E^+\oplus E^-.
\]
Then the condition $u\in\ker T_Q$ translates to 
\[
\sum_{k\in\Z}e^{-2\pi kJ_0t}\big( \hat a_k'(s)+(2\pi k+1)\hat a_k(s),  \hat b_k'(s)+(2\pi k-1) \hat b_k(s)\big)=0
\]
Therefore we have
\[
\hat a_k(s)=e^{-(2\pi k+1)s}\hat a_k(0),\qquad \hat b_k(s)=e^{-(2\pi k-1)s}\hat b_k(0)
\]
Since $u$ has finite $H^1$-norm, 
\[
\hat a_k(0)=0,\qquad \forall k\leq -1
\]
and 
\[
\hat b_k(0)=0,\qquad \forall k\leq 0.
\]
This together with the boundary condition $u(0,\cdot)\in V$ yields that $u\in\ker T_Q$ if and only if
\[
\hat u(0)\in\bigoplus_{1\leq k\leq \ell} W_k\oplus (W_0\cap E^+)
\]
and therefore
\[
\dim\ker T_Q=\sum_{k=1}^\ell\dim W_k+\dim (W_0\cap E^+).
\]
To compute the dimension of the cokernel of $T_Q$, we use its formal adjoint
\[
T_Q^*:H^1_{V^\perp}((0,+\infty)\times\T,\R^{2n})\to L^2((0,+\infty)\times\T,\R^{2n}),\quad u\mapsto -\partial_s u+J_0\partial_tu-Q u.
\]
Arguing as above, we see that $u\in T_Q^*$ if and only if $\hat u_k$ is such that
\[
\hat a_k(s)=e^{(2\pi k+1)s}\hat a_k(0),\qquad  \hat b_k(s)=e^{(2\pi k-1)s}\hat b_k(0)
\]
with
\[
\hat u(0)\in \bigoplus_{-\ell\leq k\leq -1}W_k^\perp\oplus (W_0^\perp\cap  E^-).
\]
Thus,
\[
\dim\coker T_Q=\dim\ker T_Q^*=\sum_{k=-\ell}^{-1}\dim W_k^\perp+\dim(W_0^\perp\cap E^-).
\]
Finally we have
\[
\begin{split}
\ind T_Q&=\sum_{k=-\ell}^\ell\dim W_k-\dim W_0-2\ell+\dim (W_0\cap E^-) - \dim (W_0^\perp\cap E^-)\\
&=\sum_{k=-\ell}^\ell\dim W_k-2\ell-1\\
&=\mathrm{dim}(V,\mathbb{H}^-_{1/2} \oplus E^-)-\mu_{CZ}(Z_{Q})
\end{split}
\]
where the last equality is again by \eqref{rel_dim} and $\mu_{CZ}(Z_{Q})=0$.
\bigskip

\noindent{\bf Case 2}:~The Conley-Zehnder index $\mu_{CZ}(Z_A)$ is even.\\[-2ex]

Suppose that $n\geq2$. We pick any $\theta_1$, $\theta_2\in\R\setminus 2\pi\Z$ and define
\[
A_{even}=
  \begin{pmatrix}
    \theta_1 & 0 \\
    0 & \theta_1
  \end{pmatrix}
  \oplus 
\begin{pmatrix}
    \theta_2 & 0 \\
    0 & \theta_2
      \end{pmatrix}
\oplus 
\begin{pmatrix}
    -1 & 0 \\
    0 & 1
  \end{pmatrix}^{\oplus n-2}.
\]
such that 
\[
\mu_{CZ}(Z_{A_{even}})=2\left\lfloor\frac{\theta_1}{2\pi}\right\rfloor+2\left\lfloor\frac{\theta_2}{2\pi}\right\rfloor+2=\mu_{CZ}(A).
\] 
Arguing as in Case 1 and using the computations of $\ind T_\theta$ and $\ind T_Q$ in Case 1, we obtain again \eqref{index_T_0} in this case.

If $n=1$, we consider $A\oplus A:[0,+\infty]\times\T\to L(\R^4)$ and the associated Fredholm operator. We also double $V$ to have a closed subspace $V\oplus V$ in $\mathbb{H}_{1/2}\oplus\mathbb{H}_{1/2}=H^{1/2}(\T,\R^4)$. Due to the additivity properties of the indices and the relative dimension, the case $n=1$ follows from the case $n=2$, that we have just shown. 
\end{proof}

When $x\neq y$, a generic choice of the $\omega_0$-compatible almost complex structure $J$ and of the Riemannian metric on $M$ makes the operator $D\overline{\partial}_{J,H}(u)$ surjective for every $u\in \mathcal{M}(x,y)$. When $x=y$, $\mathcal{M}(x,y)= \mathcal{M}(x,x)$ consists of just the constant half-cylinder mapping to $x$, see Proposition \ref{hybrid_ineq}. At such a stationary solution $u$, changing the almost complex structure and the metric does not affect the linearized operator $D\overline{\partial}_{J,H}(u)$. Therefore, we need the following automatic transversality result.

\begin{prop}
\label{autotrans}
Let $x$ be 1-periodic orbit of $X_H$. Assume that $J$ satisfies in addition that $J(s,t,x(t))=J_0$ for all $(s,t)\in[0,+\infty)\times\T$. Let $u\in\mathcal{M}(x,x)$ be the constant half-cylinder mapping to $x$. Then the linear operator 
\[
D\overline{\partial}_{J,H}(u): T_u \mathcal{H}_{x,x} \longrightarrow L^2((0,+\infty)\times\T,\R^{2n})
\]
is invertible.
\end{prop}

\begin{proof}
Since we have seen in Proposition \ref{fredholm} that $D\overline{\partial}_{J_0,H}(u)$ is Fredholm of index 0, it is enough to show that  the kernel of $D\overline{\partial}_{J_0,H}(u)$ is trivial. Let $v\in\ker D\overline{\partial}_{J_0,H}(u)$. Since $u(s,\cdot)=x$ is a critical point of $\Phi_H$, this translates into
\[
\partial_s v+\nabla_{L^2}^2\Phi_H(x)v=0
\]
where $\nabla_{L^2}^2\Phi_H(x)=J_0\partial_t-\nabla^2H(x)$ is the $L^2$-Hessian of $\Phi_H$ at $x$ which satisfies 
\[
\big(\nabla_{L^2}^2\Phi_H(x)\,\cdot,\cdot\big)_{L^2(\T)}=d^2\Phi_H(x).
\]
We define a function $\varphi:[0,+\infty)\to [0,+\infty)$ by $\varphi(s):=\|v(s,\cdot)\|_{L^2(\T)}^2$. Its first and second derivatives are
\[
\varphi'(s) = -2\big( v(s,\cdot),\nabla_{L^2}^2\Phi_H(x)v(s,\cdot)\big)_{L^2(\T)}, \qquad  \varphi''(s) = 4\big\|\nabla_{L^2}^2\Phi_H(x)v(s,\cdot)\big\|^2_{L^2(\T)},
\]
where we used the fact that $\nabla_{L^2}^2\Phi_H(x)$ is symmetric with respect to the $L^2$-inner product.
Since $v$ has finite $H^1$-norm, $\varphi(s)$ converges to $0$ as $s$ goes to $+\infty$. Unless $v=0$, this happens only if 
\begin{equation}
\label{varphi'}
-2\big( v(0,\cdot), \nabla_{L^2}^2\Phi_H(x)v(0,\cdot)\big)_{L^2(\T)}=\varphi'(0)<0
\end{equation}
since $\varphi''\geq0$. Using that $x$ is a critical point of $\Phi_H$ again, we deduce 
\[
\Phi_H(x+\epsilon v(0,\cdot))=\Phi_H(x)+\frac{\epsilon^2}{2}d^2\Phi_H(x)[v(0,\cdot),v(0,\cdot)]+O(\epsilon^3).
\]
Therefore \eqref{varphi'} yields that for small $\epsilon$,
\begin{equation}
\label{ineq_autotrans}
\Phi_H(x+\epsilon v(0,\cdot))>\Phi_H(x).
\end{equation}
On the other hand, the condition
\[
v(0,\cdot)\in V= T_xW^u(\pi(x);-\nabla\psi_{H^*})\oplus\R^{2n}\oplus\mathbb{H}^-_{1/2}
\]
implies the opposite. To see this, we write 
\[
v(0,\cdot)=v_0+v_1+v_2,\quad v_0\in T_xW^u(\pi(x);-\nabla\psi_{H^*}),\;\; v_1\in \R^{2n}, \;\;  v_2\in\mathbb{H}^-_{1/2}.
\]
If $v_0\neq0$, then
\[
d^2\Psi_{H^*}(\pi(x))[v_0,v_0]< 0.
\]
Arguing as above, this implies that for small $\epsilon$,
\[
\Psi_{H^*}\big(\pi(x+\epsilon(v_0+v_1))\big)= \Psi_{H^*}\big(\pi(x+\epsilon v_0)\big)\leq \Psi_{H^*}(\pi(x)).
\]
Using this together with Proposition \ref{confronto}, we deduce
\[
\Phi_H(x+\epsilon v(0,\cdot))\leq \Psi_{H^*}\big(\pi(x+\epsilon(v_0+v_1))\big)-\frac{1}{2}\|\epsilon v_2\|_{1/2}^2\leq  \Psi_{H^*}(\pi(x))=\Phi_H(x).
\]
This contradicts \eqref{ineq_autotrans}. This proves $v=0$ and completes the proof.
\end{proof}

\section{Compactness properties of the hybrid problem}

In this section, we investigate the compactness properties of the set $\mathcal{M}(x,y)$. We keep the same assumptions on the Hamiltonian and the almost complex structure: $H\in C^{\infty}(\T\times \R^{2n})$ is non-degenerate, quadratically convex and non-resonant at infinity, while $J=J(s,t,x)$ is $\omega_0$-compatible, globally bounded, equal to $J_0$ if $s\in [0,1]$ and independent of $s$ for $s$ large.

We start by observing that the elements of $\mathcal{M}(x,y)$ have uniform energy bounds due to Proposition \ref{hybrid_ineq}.
Then arguments similar to those of Section \ref{floer-complex} and building on Proposition \ref{linfbound} lead to the following result.

\begin{prop}
\label{awabd}
For every $\sigma>0$ the set 
\[
\{u|_{[\sigma,+\infty)\times \R} \mid u\in \mathcal{M}(x,y)\}
\]
is bounded in $L^{\infty}([\sigma,+\infty)\times\T,\R^{2n})$. Moreover, for every $S>\sigma$ the set
\[
\{u|_{[\sigma,S]\times \R} \mid u\in \mathcal{M}(x,y)\}
\]
is pre-compact in $C^{\infty}([\sigma,S]\times\T,\R^{2n})$. 
\end{prop}

In order to find bounds near the boundary for the elements of $\mathcal{M}(x,y)$ we start with the following:

\begin{lem}
\label{interface}
Let $u\in \mathcal{M}(x,y)$ and write
\[
u(0,\cdot) = v + w_0 + w_1,
\]
where $v\in W^u(\pi(x);-\nabla\psi_{H^*})$, $w_0\in \R^{2n}$ and $w_1\in \mathbb{H}_{1/2}^-$. Then:
\begin{enumerate}[(i)]
\item $v$ belongs to a compact subset of $C^{\infty}(\T,\R^{2n})$ which depends only on $x$ and $y$;
\item $\|w_1\|_{1/2}^2 \leq 2 (\Psi_{H^*}(x) - \Phi_H(y))$;
\item $w_0$ belongs to a compact subset of $\R^{2n}$ which depends only on $x$ and $y$.
\end{enumerate}
\end{lem}

\begin{proof}
By Proposition \ref{confronto}, we have
\[
\Phi_H(y) \leq \Phi_H(u(0,\cdot) = \Phi_H(v + w_0 + w_1) \leq \Psi_{H^*}(v) - \frac{1}{2} \|w_1\|_{1/2}^2,
\]
which immediately implies (ii). The above inequality also implies that $v$ belongs to the set
\[
W^u(\pi(x);-\nabla\psi_{H^*}) \cap \{\Psi_{H^*}\geq \Phi_{H}(y)\},
\]
which is pre-compact in $C^{\infty}(\T,\R^{2n})$. So (i) holds.

The quadratic convexity assumption on $H$ guarantees that
\[
H_t(z) \geq a |z|^2 - b \qquad \forall (t,z)\in \T\times \R^{2n},
\]
for suitable positive numbers $a,b$. Then we find
\[
\begin{split}
\int_{\T} H_t(v+w_0+w_1)\, dt &\geq a \int_{\T} |v+w_0+w_1|^2\, dt - b \geq a \int_{\T} (|w_0| - |v| - |w_1|)^2\, dt - b \\ & \geq a |w_0|^2 - 2 a |w_0| \int_{\T} (|v|+|w_1|)\, dt - b.
\end{split}
\]
The  integral $\int_\T(|v|+|w_1|)dt$ is uniformly bounded thanks to (i) and (ii), so we get the uniform bound
\[
\int_{\T} H_t(v+w_0+w_1)\, dt \geq a|w_0|^2 - 2ac|w_0|-b=a(|w_0|-c)^2-ac^2-b,
\]
for a suitable positive number $c$ which depends only on $x$ and $y$. From the above bound we deduce
\[
\begin{split}
 \Phi_H(y) &\leq \Phi_H(v + w_0 + w_1) = \frac{1}{2} \int_{\T} J_0 (\dot{v} + \dot{w_1}) \cdot (v+w_1)\, dt - \int_{\T} H_t(v+w_0+w_1)\, dt \\ &\leq \|v+w_1\|_{1/2}^2 - a( |w_0|-c)^2 +ac^2+b.
 \end{split}
 \]
 This inequality, together with the fact that $\|v+w_1\|_{1/2}$ is uniformly bounded thanks to (i) and (ii), implies that $w_0$ is uniformly bounded in $\R^{2n}$. This concludes the proof of (iii).
\end{proof}

We can now prove the compactness properties of the restrictions of elements of $\mathcal{M}(x,y)$ near the boundary. The proof consists in making the argument of the regularity result Proposition \ref{bdryreg} quantitative.

\begin{prop}
\label{bdrybds}
For every $S>0$ the set
\[
\{u|_{[0,S]\times \T} \mid u\in \mathcal{M}(x,y)\}
\]
is pre-compact in $C^{\infty}([0,S]\times \T,\R^{2n})$. Moreover, $\mathcal{M}(x,y)$ is bounded in $L^{\infty}([0,+\infty)\times \T,\R^{2n})$.
\end{prop}

\begin{proof}
Once the first statement has been proven, the boundedness of $\mathcal{M}(x,y)$  in the space $L^{\infty}([0,+\infty)\times \T,\R^{2n})$ follows immediately from the first statement of Proposition \ref{awabd}. Moreover by the second statement of Proposition \ref{awabd}, the restriction of every $u\in \mathcal{M}(x,y)$ to $[1,S]\times \T$ is uniformly bounded in $C^{\infty}([1,S]\times \T,\R^{2n})$ for every $S\geq 1$, so 
it suffices to prove the first statement for $S=1$.

The restriction of each $u\in \mathcal{M}(x,y)$ to $[0,1]\times \T$ satisfies the equation
\begin{equation}
\label{arifloer}
\overline{\partial} u = \nabla H_t(u) \qquad \mbox{on } [0,1]\times \T
\end{equation}
and the boundary condition
\begin{equation}
\label{aribound}
u(0,\cdot) =  v + w_0 + w_1, \qquad \mbox{where } v\in W^u(\pi(x); -\nabla\psi_{H^*}), \; w_0\in \R^{2n}, \; w_1 \in \mathbb{H}_{1/2}^-.
\end{equation}
The fact that $u$ has uniformly bounded energy, together with the uniform bound on $\|u(0,\cdot)\|_{L^2(\T)}$ following from Lemma \ref{interface}, implies a uniform bound for the $L^2$ norm of $u$ on $(0,1)\times \T$. By the linear growth of $\nabla H$, we deduce that also $\nabla H(u)$ has a uniform $L^2$ bound on $(0,1)\times \T$. Since both $u(0,\cdot)$ and $u(1,\cdot)$ are uniformly bounded in $\mathbb{H}_{1/2}$, because of Lemma \ref{interface} and the second statement in Proposition \ref{awabd}, the first formula of Lemma \ref{formula} and (\ref{arifloer}) give us a uniform $L^2$ bound for $\nabla u$ on $(0,1)\times \T$. Therefore, the elements of $\mathcal{M}(x,y)$ are uniformly bounded in $H^1((0,1)\times \T,\R^{2n})$.

Now we wish to prove that the elements of $\mathcal{M}(x,y)$ are uniformly bounded in $H^k((0,1)\times \T,\R^{2n})$ for every natural number $k$. By the Sobolev embedding theorem, this will give us the boundedness of 
\[
\{u|_{[0,1]\times \T} \mid u\in \mathcal{M}(x,y)\}
\]
in $C^k([0,1]\times \T,\R^{2n})$ for every every natural number $k$, and by the Ascoli-Arzel\`a theorem its pre-compactness in $C^{\infty}([0,1]\times \T,\R^{2n})$.

The uniform bound in $H^1((0,1)\times \T,\R^{2n})$ has just been proven. Then the Floer equation (\ref{arifloer})
and the growth conditions on $H$ imply that $\overline{\partial} u$ has a uniform bound in $H^1((0,1)\times \T,\R^{2n})$. Thanks to the boundary condition (\ref{aribound}),
the bound (\ref{extremis}) from Lemma \ref{reglemma}, together with Lemma \ref{interface} and the fact that all the derivatives of $u(1,\cdot)$ are uniformly bounded, implies a uniform bound for $u$ in $H^2((0,1)\times \T,\R^{2n})$. In particular, the restriction of $u$ to $(0,1)\times \T$ is uniformly bounded in $L^{\infty}((0,1)\times \T,\R^{2n})$.

Arguing by induction, we now assume that  the elements of $\mathcal{M}(x,y)$ are uniformly bounded in $H^k((0,1)\times \T,\R^{2n})$ for some integer $k\geq 2$, and we wish to prove a uniform bound in $H^{k+1}((0,1)\times \T,\R^{2n})$. By differentiating the Floer equation $k$ times with respect to $t$, we find, thanks to Lemma \ref{bootstrap},
\[
\overline{\partial} \partial_t^k u = \partial_t^k \overline{\partial} u = \nabla^2 H_t(u) \partial_t^k u + p,
\]
where $p$ is a $\R^{2n}$-valued polynomial map in $\partial_t u,\partial_t^2 u, \dots,\partial_t^{k-1} u$ whose coefficients are smooth functions of $(t,u(s,t))$. The fact that $u$ has a uniform bound in $L^{\infty}((0,1)\times \T,\R^{2n})$ implies that these coefficients are also uniformly bounded in this space. In particular, the polynomial map $p$ has the pointwise estimate
\[
|p|^2 \leq C_0 ( 1 + |\partial_t u|^N + \dots + |\partial_t^{k-1} u|^N),
\]
for suitable constants $C_0$ and $N$. Then the inductive hypothesis and the continuity of the  Sobolev embedding of $H^k$ into $W^{k-1,N}$ imply a uniform bound of the form
\[
\| \overline{\partial} \partial_t^k u \|_{L^2((0,1)\times \T)} \leq C_1 \qquad \forall u\in \mathcal{M}(x,y),
\]
for some constant $C_1$.

By (\ref{aribound}) we have
\[
\partial_t^k u(0,\cdot) = v^{(k)} + w^{(k)}_1.
\]
Here, $v^{(k)}$ is uniformly bounded in $H^{1/2}(\T,\R^{2n})$ thanks to statement (i) in Lemma \ref{interface}. On the other hand, $w_1^{(k)}$ is an element of $\mathbb{H}_{1/2}^-$, because the time derivative of a loop in $\mathbb{H}_{1/2}^- \cap C^{\infty}(\T,\R^{2n})$ belongs to $\mathbb{H}_{1/2}^-$. Therefore 
\[
\left| \int_{\T} (v^{(k)})^* \lambda_0 \right| \leq \|v^{(k)}\|_{1/2}^2 \leq C_2 \qquad \mbox{and} \qquad \int_{\T} (w_1^{(k)})^* \lambda_0 = - \|w_1^{(k)}\|_{1/2}^2 \leq 0,
\]
for some constant $C_2$. By the second statement in Proposition \ref{awabd}, we also have
\[
\left| \int_{\T} (\partial_t^k u(1,\cdot))^* \lambda_0 \right| \leq \|\partial_t^k u(1,\cdot)\|_{1/2}^2 \leq C_3,
\]
for some constant $C_3$. Then the first formula of Lemma \ref{formula} applied to $\partial_t^k u$ gives
\[
\begin{split}
\int_{(0,1)\times \T} |\nabla \partial_t^k u|^2\, ds dt &= \int_{(0,1)\times \T} |\overline{\partial} \partial_t^k u|^2\, ds dt - 2 \int_{\T} (\partial_t^k u(1,\cdot))^* \lambda_0 \\ &+ 2 \int_{\T} (v^{(k)})^*\lambda_0 + 2\int_{\T}  (w_1^{(k)})^* \lambda_0 \leq  C_1^2 + 2 C_3 + 2 C_2.
\end{split}
\]
This shows that the partial derivatives $\partial_s \partial_t^k u$ and $\partial_t^{k+1} u$ have uniform $L^2$-bounds on $(0,1)\times \T$. The uniform $L^2$ bounds on all the other partial derivatives of order $k+1$ now follow easily from the Floer equation and Lemma \ref{bootstrap}. Indeed, by applying the differential operator $\partial_s \partial_t^{k-1}$ to the Floer equation we find
\begin{equation}
\label{pezzi}
\partial_s^2 \partial_t^{k-1} u = - J_0  \partial_s \partial_t^k u + \nabla^2 H_t(u) \partial_s \partial_t^{k-1} u + q,
\end{equation}
where $q$ is a $\R^{2n}$-valued polynomial map in the variables
\begin{equation}
\label{parziali}
\partial_t u,\; \partial_t^2 u, \dots, \; \partial_t^{k-1} u,\;  \partial_s u, \; \partial_s \partial_t u, \dots, \; \partial_s \partial_t^{k-2} u,
\end{equation}
whose coefficients are uniformly bounded. The first term on the right-hand side of (\ref{pezzi}) is bounded in $L^2((0,1)\times \T)$, as shown above. The middle term is also bounded in $L^2((0,1)\times \T)$ by the inductive hypothesis. Being a polynomial in the partial derivatives (\ref{parziali}) with uniformly bounded coefficients, $q$ has the pointwise estimate
\[
|q|^2 \leq C \bigl( 1 + |\partial_t u|^N + \dots + |\partial_t^{k-1} u|^N + |\partial_s u|^N + \dots + |\partial_s \partial_t^{k-2} u|^N \bigr),
\]
for a suitable positive number $C$ and a suitable natural number $N$. Integration over $(0,1)\times \T$ and the observation that the partial derivatives appearing above have order at most $k-1$ imply that
\[
\|q\|_{L^2((0,1)\times \T)}^2 \leq C \bigl( 1 + \|u\|_{W^{k-1,N}((0,1)\times \T)}^N \bigr). 
\]
By the continuity of the Sobolev embedding
\[
H^k((0,1)\times \T)) \hookrightarrow W^{k-1,N}((0,1)\times \T)
\]
and by the inductive hypothesis we deduce that $q$ has a uniform $L^2$ bound on $(0,1)\times \T$. Then all terms on the right-hand side of (\ref{pezzi}) have uniform $L^2$ bounds on $(0,1)\times \T$ and hence the same is true for the term on the left-hand side, that is, $\partial_s^2 \partial_t^{k-1} u$. By iterating this argument inductively in $h$, we obtain that $\partial_s^h \partial_t^{k-h+1}$ has a uniform $L^2$ bound on $(0,1)\times \T$ for every $h\in \{0,\dots,k+1\}$. We conclude that the elements of $\mathcal{M}(x,y)$ are uniformly 
bounded in $H^{k+1}((0,1)\times \T)$. This proves the induction step and concludes the proof. 
\end{proof}

\section{The chain complex isomorphism}
\label{sec:isom}

Let $H\in C^{\infty}(\T\times\R^{2n})$ be non-degenerate, quadratically convex and non-resonant at infinity. As we have seen in Section \ref{sec:Morse}, the dual action functional $\Psi_{H^*}$ restricts to a smooth Morse function $\psi_{H^*}$ on a finite dimensional manifold $M\subset \mathbb{H}_1$, whose Morse complex $\{M_*(\psi_{H^*}),\partial^M\}$ is well-defined. On the other hand, we have the Floer complex $\{F_*(H),\partial^F\}$ of the Hamiltonian $H$. These two complexes depend on auxiliary data - a generic metric on $M$ and a generic $\omega_0$-compatible almost complex structure on $\R^{2n}$ - but choices of different auxiliary data change them by isomorphisms preserving gradings and actions. 

The generators of  these two chain complexes are in one-to-one correspondence: The generator $x$ of $F_*(H)$ - a 1-periodic orbit of $X_H$ - induces the generator $\pi(x)$ of $M_*(\psi_{H^*})$ and the relationships between grading and actions are
\[
\mathrm{ind}(\pi(x);\psi_{H^*}) = \mu_{CZ}(x) - n, \qquad \psi_{H^*}(\pi(x)) = \Phi_H(x)
\]
by Propositions \ref{cz=relind}, \ref{relind=ind}, \ref{prop:reduction}, and Lemma \ref{crit}. 
In this section, we wish to construct a chain complex isomorphism
\[
\Theta:\{M_{*-n}(\psi_{H^*}),\partial^M\} \longrightarrow\{F_*(H),\partial^F\}
\]
that preserves the action filtrations, and hence induces chain complex isomorphisms
\[
\Theta^{<a}: M_{*-n}^{<a}(\psi_{H^*}) \rightarrow F_*^{<a}(H)
\]
for every real number $a$.

Let $x$ and $y$ be 1-periodic orbits of $X_H$. For the moduli space $\mathcal{M}(x,y)$ of solutions of the hybrid problem, we use a family of uniformly bounded $\omega_0$-compatible almost complex structures smoothly parametrized by $[0,+\infty)\times\T$, such that $J(s,t,z)=J_0$ for all $(s,t,z)\in [0,1]\times \T\times \R^{2n}$, $J(s,t,x(t))=J_0$ for every 1-periodic orbit $x$ of $X_H$ and every $(s,t)\in [0,+\infty)\times \T$, and $J(s,t,z)$ is independent of $s$ for all $s$ large. Standard transversality results, together with Propositions  \ref{fredholm}  and \ref{autotrans}, imply that for a generic choice of such a $J$ and of the Riemannian metric on $M$ the differential of the map $\overline{\partial}_{J,H}$ at every $u\in \mathcal{M}(x,y)$ is onto, for every pair of 1-periodic orbits $x,y$ of $X_H$. We fix such generic auxiliary data.

Let $x$ and $y$ have the same Conley-Zehnder index, $\mu_{CZ}(x)=\mu_{CZ}(y)$. Proposition \ref{fredholm} now implies that $\mathcal{M}(x,y)$ is a zero-dimensional manifold. Together with Propositions \ref{awabd} and \ref{bdrybds}, we deduce that $\mathcal{M}(x,y)$ is a finite set and we denote its parity by $n^\mathrm{hyb}(x,y)$. 

We define the sequence of homomorphisms
\[
\Theta_k:M_{k-n}(\psi_{H^*})\longrightarrow F_k(H),\qquad k\in\mathbb{Z}
\]
by the linear extension of
\[
\Theta_k\pi(x)=\sum_y n^\mathrm{hyb}(x,y)y
\]
where the sum runs over all 1-periodic orbits $y$ of $X_H$ with $\mu_{CZ}(y)=\mu_{CZ}(x)=k$. A standard argument using compactness and gluing results shows that $\{\Theta_k\}_{k\in\N}$ is a chain homomorphism.

Proposition \ref{hybrid_ineq} implies that $\mathcal{M}(x,y)$ is empty whenever $x\neq y$ and $\psi_{H^*}(\pi(x)) \leq \Phi_H(y)$. Therefore, the chain homomorphism $\Theta$ maps the subcomplex $M^{<a}(\psi_{H^*})$ into the subcomplex $F^{<a}(H)$, for every $a\in \R$.

To prove that $\Theta_k$ is an isomorphism, we label all 1-periodic orbits of $X_H$ with Conley-Zehnder index $k$ by $y_1,\dots,y_m$ to satisfy
\[
\Phi_H(y_1)\leq \Phi_H(y_2)\leq \cdots \leq \Phi_H(y_m).
\]
The homomorphism $\Theta_k$ with respect to the bases $\{\pi(y_1),\dots,\pi(y_m)\}$ of $M_k(\psi_{H^*})$ and $\{y_1,\dots,y_m\}$ of $F_k(H)$ is an $m$-by-$m$ matrix which is upper-triangular since $\mathcal{M}(y_j,y_i)$ is empty if $j<i$, by Proposition \ref{hybrid_ineq}. Moreover, $\mathcal{M}(y_i,y_i)$ consists of a single element, namely the constant half-cylinder mapping to $y_i$, for all $i\in\{1,\dots,m\}$, again by Proposition \ref{hybrid_ineq}, and hence the matrix has the entries 1 on the diagonal. In particular,  $\Theta_k$ is an isomorphism for all $k\in\mathbb{Z}$. This concludes the proof of the main theorem stated in the introduction.

\section{Symplectic homology and $SH$-capacity of convex domains}
\label{sec:SHcap}

We briefly recall the definition of symplectic homology and of the associated $SH$-capacity in the special case of a smooth starshaped domain $W\subset \R^{2n}$. 

Let $\mathcal{H}(W)$ be the set of smooth non-degenerate Hamiltonian functions $H:\T\times\R^{2n}\to\R$ that are non-resonant at infinity, have a Hamiltonian vector field $X_H$ with linear growth, and satisfy 
\[
H|_{\T\times W}<0.
\]
As we have seen in Section \ref{floer-complex}, the Floer homology of such a Hamiltonian $H$ is well-defined.
On $\mathcal H(W)$ we consider the standard partial order $\leq$ given by pointwise inequality. For $H,\, K\in\mathcal{H}(W)$ with $H\leq K$, we have a continuation homomorphism
\[
HF^{<a}(H)\rightarrow HF^{<a}(K)
\] 
for any $a\in (-\infty,+\infty]$. Such maps form a direct system over $\mathcal H(W)$. Taking the direct limit, we define the symplectic homology of $W$ filtered by action $a\in (-\infty,+\infty]$:
\[
SH^{<a}(W):=\varinjlim_{H\in\mathcal H(W)} HF^{<a}(H).
\]
If $\epsilon>0$ is smaller than the smallest action of a closed characteristic on $\partial W$ then
\[
SH_n^{<\epsilon}(W) \cong H_n(W,\partial W; \Z_2) = \Z_2.
\]
 On the other hand, the full symplectic homology of $W$ is isomorphic to the full symplectic homology of the unit ball which is zero, see e.g.~\cite{oan04}. This vanishing result also follows from the fact that $W$ is displaceable from itself by a compactly supported Hamiltonian diffeomorphism on $\R^{2n}$, see e.g.~\cite{Kan14}. Therefore we can define the $SH$-capacity of $W$ as the positive number
\[
c_{SH}(W) := \inf \{a > \epsilon \mid SH_n^{<\epsilon}(W) \rightarrow SH_n^{<a}(W) \mbox{ is zero} \} \in (0,+\infty).
\]
It is well-known that the number $c_{SH}(W)$ is an element of the action spectrum of $\partial W$. The boundary of a  general smooth starshaped domain $W$ might possess a closed characteristic whose action is less than $c_{SH}(W)$. Let $W$ be a  Bordeaux-bottle-shaped domain as in \cite[Chapter 3.5]{hz94}. It contains a ball $B^{2n}_R$ of radius $R$ and is contained in a cylinder $B^2_R\times \R^{2n-2}$ of the same radius, and hence $c_{SH}(W)=\pi R^2$. But there is a closed characteristic on the bottle neck of $W$, which corresponds to the boundary of a piece of the cylinder of radius $r<R$, which has action $\pi r^2$. 
It is also worth pointing out that the $SH$-capacity of any starshaped domain $W$ is represented by a closed Reeb orbit on $(\partial W,\alpha_W)$ of transversal Conley-Zehnder index $n+1$, but minimal periodic orbits do not necessarily have transversal Conley-Zehnder index $n+1$.

\begin{rem}
Symplectic homology was introduced in \cite{fh94} and the above symplectic capacity was defined in \cite{fhw94}. See also \cite{vit99} for a somehow different and quite influential approach to symplectic homology. Note that the growth conditions for the Hamiltonians $H\in \mathcal{H}(W)$ in \cite{fh94} and \cite{vit99} are different: In the former reference, it is required that
\[
H(t,z) = \eta |z|^2 + \xi 
\]
for every $z$ outside a large ball and every $t\in \T$, where $\eta \in (0,+\infty) \setminus \pi \Z$ and $\xi\in \R$, while in the latter the condition (\ref{slope}) from Lemma \ref{cresce} adapted to $W$ is used. Both conditions define sets of Hamiltonians that are cofinal in the larger set of non-resonant Hamiltonians that are considered here, so these different definitions of symplectic homology actually coincide. The definition of continuation homomorphisms for $s$-dependent Hamiltonians that are non-resonant at infinity for $|s|$ large requires the compactness results that are discussed in Remark \ref{sdep}.
\end{rem}

The aim of this section is to prove the corollary stated in the introduction: If $C\subset \R^{2n}$ is a bounded open convex set with smooth boundary, then its $SH$-capacity $c_{SH}(C)$ coincides with the minimum of the action spectrum of $\partial C$. 

In order to prove this, we may assume that $C$ is a smooth strongly convex domain containing the origin, in the sense of Section \ref{dualsec}, and that all the closed Reeb orbits on $R_{\alpha_C}$ on $\partial C$ are non-degenerate. The latter requirement means that for every closed orbit $\gamma:\R/T\Z\to\partial C$ of ${R_{\alpha_C}}$, the linearized return map $d\phi_{R_{\alpha_C}}^T(\gamma(0))$ of the flow of the Reeb vector field $R_{\alpha_C}$ restricted to the distribution $\ker\alpha_C$ does not have 1 as an eigenvalue. Indeed, up to a translation, any bounded open convex set with smooth boundary can be approximated by domains with these further properties, and both $c_{SH}$ and the minimum of the action spectrum of $\alpha_C$ are continuous on the space of bounded convex open sets with smooth boundary. Here, the convergence can be either the Hausdorff convergence, or the convergence induced by the smooth topology.

Under the above assumption, the action spectrum of $\partial C$ is discrete and we can prove the following result.

\begin{lem}\label{hamiltonian}
Let $A_{\min}(\partial C)$ and $A_{2nd}(\partial C)$ be the smallest and the second smallest numbers in the action spectrum of $\partial C$, respectively. For every $\epsilon>0$ and  $\eta\in(A_{\min}(\partial C),A_\mathrm{2nd}(\partial C))$, there exists $H\in\mathcal H(C)$ which is quadratically convex, satisfies
\[
H(t,x) = \eta H_C(x) + \xi \qquad \forall t\in \T, \; \forall |x|\geq R,
\]
for suitable real numbers $\xi\in\R$ and $R>0$, and has the following properties: The set of 1-periodic orbits of the Hamiltonian vector field $X_H$ consists of $2m+1\geq 3$ elements $z, y_1^-,y_1^+,\dots,y_m^-,y_m^+$ satisfying the following conditions:
\begin{enumerate}[(a)]
\item The Conley-Zehnder indices are 
\[
\mu_{CZ}(z)=n,\quad \mu_{CZ}(y_i^-)=n+1,\quad \mu_{CZ}(y_i^+)=n+2,\qquad\forall i\in\{1,\dots,m\}.
\]
\item The action values are
\[
-\epsilon<\Phi_H(z)<\epsilon,\quad A_{\min}(\partial C)-\epsilon<\Phi_H(y_i^-)<\Phi_H(y_i^+)<A_{\min}(\partial C)+\epsilon,\quad\forall i\in\{1,\dots,m\}.
\]
\item In the Floer chain complex $\{F(H),\partial^F\}$, 
\[
\partial^F y_i^+=0,\qquad\forall i\in\{1,\dots,m\},
\]
and there exists $j\in\{1,\dots,n\}$ such that
\begin{equation}\label{eq:kill_z}
\partial^F y_j^-=z.
\end{equation}

\end{enumerate}
\end{lem}
\begin{proof}
We choose a smooth function $\varphi:\R\to\R$ such that
\begin{enumerate}[(i)]
\item[-] $\varphi(s)=\varphi(0)$ for all $s\leq0$;
\item[-] $\varphi''(s)>0$ for all $s>0$;
\item[-] $\varphi(s)=\eta s+\xi$ for some $\xi \in\R$ and for all $s>2$;
\item[-] $\varphi(s)\in(-\epsilon/2,0)$ for all $s\leq 1$;
\item[-] $\varphi'(1)=A_{\min}(\partial C)$.
\end{enumerate}
There are two types of 1-periodic orbits of $X_{\varphi\circ H_C}$:
\begin{enumerate}
\item[-] the constant curve $z:\T\to\R^{2n}$ mapping to the origin with action
\[
\Phi_{\varphi\circ H_C}(z)=-\varphi\circ H_C(z)=-\varphi(0)\in(0,\epsilon/2);
\]
\item[-] curves $y_i:\T\to\partial C$, $i\in\{1,\dots,m\}$ where $y_i(t/A_{\min}(\partial C))$ is a closed orbit  of $R_{\alpha_C}$ with period $A_{\min}(\partial C)$, with action
\[
\Phi_{\varphi\circ H_C}(y_i)=A_{\min}(\partial C)-\varphi(1)\in\big(A_{\min}(\partial C),A_{\min}(\partial C)+\epsilon/2\big),
\]
and with index
\[
\mu_{CZ}(y_i)=n+1,\qquad\forall i\in\{1,\dots,m\}
\]
by Propositions \ref{cz=relind},  \ref{index1} and \ref{index2}. 
\end{enumerate}
Note that $\varphi\circ H_C$ is everywhere smooth, $\nabla^2(\varphi\circ H_C)$ is positive definite on $\R^{2n}\setminus\{0\}$ and $\nabla^2(\varphi\circ H_C)(0)=0$. We choose a $C^2$-small function $f:\R^{2n}\to\R$ supported in a small neighborhood of the origin such that  the origin is a critical point of $f$ and the Hessian $\nabla^2f$ is sufficiently small and positive definite near the origin. Then the function $\varphi\circ H_C+f$ has positive definite Hessian, has a unique critical point at the origin, and coincides with $\varphi\circ H_C$ outside a neighborhood of the origin. In particular, $\varphi\circ H_C+f$ is quadratically convex. Since the origin is a minimizer of $\varphi\circ H_C+f$, the Conley-Zehnder index of $z$ is
\[
\mu_{CZ}(z)=\ind(0;\varphi\circ H_C+f)+n=n,
\]  
where $\ind(0;\varphi\circ H_C+f)$ denotes the Morse index of $\varphi\circ H_C+f$ at the origin. 
Adding an additional small perturbation supported in neighborhoods of the periodic orbits $y_i$ as in \cite{bo09b} to $\varphi\circ H_C+f$, we obtain a smooth function $H:\T\times\R^{2n}\to\R$  such that:
\begin{enumerate}
\item[-] $H$ is quadratically convex;
\item[-] $H=\varphi\circ H_C$ outside the neighborhoods of the origin and of the $y_i(\T)$;
\item[-]  $H$ satisfies all the properties in the statement except possibly for \eqref{eq:kill_z}.
\end{enumerate}

In order to establish \eqref{eq:kill_z}, we use Clarke's duality. By Lemma \ref{crit} and Proposition \ref{prop:reduction}, the reduced dual functional $\psi_{H^*}:M\to\R$ has non-degenerate critical points 
\[
\pi(z),\quad\pi(y_i^-),\quad\pi(y_i^+),\qquad i\in\{1,\dots,m\}
\]
with Morse indices
\[
\ind(\pi(z),\psi_{H^*})=0,\quad \ind(\pi(y_i^-),\psi_{H^*})=1,\quad \ind(\pi(y_i^+),\psi_{H^*})=2,\qquad \forall i\in\{1,\dots,m\}.
\]
Thanks to the Theorem from the introduction and \eqref{morse=singular}, we have
\begin{equation}\label{isomorphisms}
HF_{k+n}(H)\cong HM_k(\psi_{H^*})\cong H_k(M,\{\psi_{H^*}<a\}),\qquad\forall k\in\Z
\end{equation}
for any $a<\psi_{H^*}(\pi(z))$. Since the $y_i^+$ represent a non-zero class in $FH_{n+2}(H)$,  by \eqref{isomorphisms}
\[
H_2(M,\{\psi_{H^*}<a\})\neq0.
\]
Since $M$ is diffeomorphic to $\R^{2nN}$, this implies that
\[
\{\psi_{H^*}<a\}\neq\emptyset,
\]
and therefore we have
\begin{equation}\label{eq:vanishing_H_0}
H_0(M,\{\psi_{H^*}<a\})=0.	
\end{equation}
Applying this to \eqref{isomorphisms}, we conclude that the cycle $z$ vanishes in $HF_n(H)$. Hence, $\partial^Fy_j^-=z$ for some $j\in\{1,\dots,m\}$. This completes the proof.
\end{proof}

\begin{rem}
	The vanishing result \eqref{eq:vanishing_H_0} can be seen also in a more direct way. We use the same notation as in Lemma \ref{hamiltonian}. We choose any large $\zeta>0$ such that 
	\[
	H(t,x)>\eta H_C(x)-\zeta \qquad \forall (t,x)\in\T\times \R^{2n}
	\]
	which is equivalent to 
	\[
	H^*(t,x)<(\eta H_C)^*(x)+\zeta \qquad \forall (t,x)\in\T\times \R^{2n}.
	\]
	Let $\gamma:\T\to \partial C$ be a closed characteristic with minimal action, i.e.~$A_{\min}(\partial C)=\int_\T\gamma^*\alpha_C$, and set $K:=A_{\min}(\partial C)H_C$. Then there holds
	\[
	0=\Phi_{K}(\gamma)=\Psi_{K^*}(\gamma)> \Psi_{(\eta H_C)^*}(\gamma)
	\]
	where the last inequality follows from the inequality $K<\eta H_C$ on $\R^{2n}\setminus\{0\}$. This, together with the bound
	\[
	\Psi_{H^*}(s\gamma)<\Psi_{(\eta H_C)^*}(s\gamma)+\zeta =s^2\Psi_{(\eta H_C)^*}(\gamma)+\zeta \qquad \forall s>0,
	\]
	where we used the fact that $\Psi_{(\eta H_C)^*}$ is positively 2-homogeneous, yields
	\[
	\lim_{s\to+\infty}\Psi_{H^*}(s\gamma)=-\infty.
	\]
	Therefore, the functional $\Psi_{H^*}$ is unbounded from below and, since by construction $\psi_{H^*}$ and $\Psi_{H^*}$ have the same infimum, so is 
	$\psi_{H^*}$. We conclude that $\{\psi_{H^*}<a\}$ is not empty  and \eqref{eq:vanishing_H_0} follows. 
\end{rem}

We can now prove the corollary stated in the introduction. The inequality $c_{SH}(C)\geq A_{\min}(\partial C)$ follows from the already mentioned fact that $c_{SH}(C)$ belongs to the action spectrum of $\partial C$. We prove the opposite inequality.  Fix some positive numbers 
\[
\epsilon<A_{\min}(\partial C), \qquad \eta\in (A_{\min}(\partial C),A_{\mathrm{2nd}}(\partial C)),
\]
and let $H$ be as in Lemma \ref{hamiltonian}. The Hamiltonian $H$ belongs to $\mathcal{H}(C)$ and $HF_n^{<\epsilon}(H)$ is isomorphic to $\Z_2$ and generated by $z$. Moreover the homomorphism 
\[
\sigma:HF_n^{<\epsilon}(H) \longrightarrow SH^{<\epsilon}_n(C)
\] 
in the direct limit defining $SH^{<\epsilon}_n(C)$ is an isomorphism. One way to see this is to observe that there is a cofinal subset  $\{H_\nu\}_{\nu\in\N}$ of $\mathcal H(C)$ such that $H_1=H$, and for every $\nu\in\N$, $X_{H_\nu}$ has a unique constant orbit $z$ mapping to the origin and all other 1-periodic orbits have $\Phi_{H_\nu}$-action larger than $\epsilon$. 

Consider the following commutative diagram:
\[
\begin{tikzcd}[row sep=4em, column sep=5em]
HF^{<\epsilon}_n(H) \arrow[rr,"\tau"] \arrow[d,swap,"\sigma"]
&& 
HF^{<A_{\min}(\partial C)+\epsilon}_n(H) \arrow[d,"\sigma'"] 
\\
SH^{<\epsilon}_n(C) \arrow[rr,"\tau'"] 
&&  
SH^{<A_{\min}(\partial C)+\epsilon}_n(C)\;.
\end{tikzcd}
\]
The horizontal maps are induced by the canonical inclusions, and the vertical ones are the homomorphisms into the direct limit in the definition of symplectic homology.
Since $\sigma$ is an isomorphism and $HF_n^{<A_{\min}(\partial C)+\epsilon}(H)\cong HF_n(H)$ vanishes due to Lemma \ref{hamiltonian}, the homomorphism $\tau'$ is zero. This implies that $c_{SH}(C)\leq A_{\min}(\partial C) + \epsilon$, and by the arbitrariness of $\epsilon$ we conclude that $c_{SH}(C)\leq A_{\min}(\partial C)$. This concludes the proof of the corollary stated in the introduction.

\begin{rem}
Here is another consequence of our main theorem, which can be proven in a similar fashion: Let $C$ be a smooth strongly convex domain and assume that the Reeb flow on $\partial C$ has more than one orbit with transversal Conley-Zehnder index equal to 3. Then it has a closed orbit of transversal Conley-Zehnder index 4 with action larger than the action of any closed orbit of transversal Conley-Zehnder index 3.
\end{rem}

\appendix 

\section{Appendix: Interior regularity of solutions of the Floer equation}
\label{appA}

In this appendix we prove the following interior regularity result for solutions of the Floer equation. The argument used in the proof was explained to us by Urs Fuchs.

\begin{prop}
\label{H1_regularity}
Let $U$ be an open subset of $\R\times \T$. 
Let $J$ be a uniformly bounded $\omega_0$-compatible almost complex structure on $\R^{2n}$, smoothly depending on $(s,t)\in U$. Let $H:\T\times\R^{2n}\to\R$ be a smooth Hamiltonian function such that
\[
|X_{H_t}(z)| \leq c (1+|z|) \qquad \forall (t,z)\in \T\times \R^{2n}
\]
for some $c>0$. If $u\in H^1_{\mathrm{loc}}(U,\R^{2n})$ is a weak solution of the Floer equation 
\begin{equation}
\label{lingro}
\partial_su+J(s,t,u)\big(\partial_t u - X_{H_t}(u)\big) = 0,
\end{equation}
then it is smooth.
\end{prop}

Since $u$ is a priori only in $H^1_{\mathrm{loc}}$, it may not be continuous and hence the map $(s,t) \mapsto J(s,t,u(s,t))$ may as well be not continuous. This prevents us from using standard arguments, in which one looks at a small neighbourhood of a point in $U$ and sees $J(s,t,u(s,t))$ there as a small perturbation of a constant complex structure on $\R^{2n}$. We overcome this difficulty by the following result.

\begin{lem}
\label{lemmulo}
Let $U$ be an open subset of $\R^2$, let $\{J(s,t)\}_{(s,t)\in U}$ be a bounded measurable family of $\omega_0$-compatible almost complex structures, and let  $f$ be a map in $L^p_{\mathrm{loc}}(U,\R^n)$ for some $p>2$. Then there exists a number $q>2$, depending only on $\|J\|_{\infty}$ and $p$, such that every $u\in L^p_{\mathrm{loc}}(U,\R^{2n})$ solving the linear equation
\begin{equation}
\label{linfloer}
\partial_s u + J \partial_t u = f
\end{equation}
in the distributional sense belongs to $W^{1,q}_{\mathrm{loc}}(U,\R^{2n})$.
\end{lem}

The above Lemma can be applied to solutions of the equation (\ref{lingro}) from Proposition \ref{H1_regularity}, because the map 
$(s,t) \mapsto J(s,t,u(s,t))$ is measurable and bounded and the map 
\[
f(s,t) = J(s,t,u(s,t)) X_{H_t}(u(s,t)) 
\]
belongs to $L^p_{\mathrm{loc}}(U,\R^n)$ for every $p\in (1,+\infty)$, thanks to the growth assumption on $X_H$ and to the fact that $u$ belongs to $L^p_{\mathrm{loc}}(U,\R^n)$ for every $p\in (1,+\infty)$, by the Sobolev embedding theorem. Then this lemma allows us to upgrade the regularity of $u$ to $W^{1,q}_{\mathrm{loc}}$ regularity, for some $q>2$. Based on this, standard regularity arguments (see e.g.\ \cite[Appendix B.4]{ms04}) imply that the solution $u$ of (\ref{lingro}) is smooth.

The remaining part of this appendix is devoted to the proof of Lemma \ref{lemmulo}. The argument consists of transforming the linear Floer equation (\ref{linfloer}) for a variable almost complex structure $J$ into a Beltrami equation, and then proving regularity for solutions of this equation using the standard Calderon-Zygmund estimates. The reader interested in learning more about the Beltrami equation and its regularity theory might refer to \cite{vek62} and \cite{boj10}.

In order to reduce the linear Floer equation to a Beltrami equation, we shall make use of the following well-known facts about complex structures (see e.g.\ \cite[Section 1.2]{is99}). Here, $|\cdot|$ and $\cdot$ denote the euclidean norm and scalar product on $\R^{2n}$ and $\|\cdot\|$ the induced operator norm on the space $L(\R^{2n})$ of linear endomorphisms of $\R^{2n}$.

\begin{lem}\label{lem:J_estimate}
Let $J$ be a complex structure on $\R^{2n}$. 
\begin{enumerate}[(i)]
\item Assume that there exists $\alpha>0$ such that
\[
\omega_0(Ju,u)\geq \alpha|u|^2,\qquad \forall u\in\R^{2n}.
\]
Then $\alpha\leq1$, $J+J_0$ is invertible, and the following inequalities hold
\[
\|(J+J_0)^{-1}\|\leq \frac{1}{1+\alpha}, \qquad \|(J+J_0)^{-1}(J-J_0)\|\leq \sqrt{1-\frac{4\alpha^3}{(1+\alpha)^2}}\,.
\]
\item Assume that $J$ is $\omega_0$-compatible, meaning that $(u,v) \mapsto \omega_0(Ju,v)$ is a scalar product on $\R^{2n}$. Then
\[
\omega_0(Ju,u) \geq \frac{1}{\|J\|} |u|^2 \qquad \forall u\in \R^{2n}.
\]
\end{enumerate}
\end{lem}

Note that the upper bound in the second inequality in (i) is strictly less than 1 for every $\alpha\in (0,1]$, vanishes if and only if $\alpha=1$ - which hence corresponds to the case $J=J_0$ - and tends to 1 for $\alpha\downarrow 0$.
 
\begin{proof}
(i) The invertibility of $J+J_0$ readily follows from 
\[
(\alpha+1)|u|^2 \leq \omega_0(Ju,u)+\omega_0(J_0u,u) = \omega_0\big((J+J_0)u,u\big).
\]
Moreover by the Cauchy-Schwarz inequality, we have
\[
(\alpha+1)|u|^2\leq \omega_0\big((J+J_0)u,u\big)= -J_0(J+J_0)u\cdot u\leq |(J+J_0)u||u|.
\]
Substituting $u=(J+J_0)^{-1}v$, we obtain
\[
|(J+J_0)^{-1}v|\leq \frac{1}{\alpha+1} |v|,\qquad\forall v\in\R^{2n}.
\]
This proves the first inequality in the lemma. From
\[
\alpha|Ju|^2\leq \omega_0(JJu,Ju)=\omega_0(Ju,u)\leq |Ju||u|
\]
we deduce the inequality
\[
\|J\|\leq \frac{1}{\alpha},
\]
and from $1=\|J^2\|\leq \|J\|^2\leq \alpha^{-2}$ we obtain $\alpha\leq 1$. If we set $S:=I-JJ_0=-J(J+J_0)$, where $I$ is the identity map, we find
\[
\|S\| \leq 1 + \|J\| \leq 1 + \frac{1}{\alpha}.
\]
Moreover, if we denote $W:=(J+J_0)^{-1}(J-J_0)$, a straightforward computation shows that
\[
S(I-WW^T)S^T=2(J_0J^T-JJ_0).
\]
From the estimate
\[
(J_0J^T-JJ_0)u\cdot u=-2JJ_0u\cdot u=2\, \omega_0(JJ_0u,J_0u)\geq 2\alpha|u|^2,
\]
we find
\[
\begin{split}
(I-WW^T)u\cdot u&=S^{-1}S(I-WW^T)S^T(S^T)^{-1}u\cdot u=2(J_0J^T-JJ_0)(S^T)^{-1}u\cdot (S^T)^{-1}u	\\
&\geq 4\alpha|(S^T)^{-1}u|^2\geq 4\alpha \|S^T\|^{-2}|u|^2 = 4\alpha \|S\|^{-2}|u|^2.
\end{split}
\]
Together with the above upper bound on the norm of $S$, we deduce
\[
WW^T u\cdot u \leq \left(1 - \frac{4\alpha}{\|S\|^2} \right) |u|^2 \leq \left( 1 - \frac{4\alpha^3}{(1+\alpha)^2} \right) |u|^2.
\]
We conclude that
\[
\|W\|^2 = \|W^T\|^2 = \max_{\substack{u\in \R^{2n} \\ |u|\leq 1}} WW^T u\cdot u \leq 1 - \frac{4\alpha^3}{(1+\alpha)^2},
\]
which proves the second bound in (i).

(ii) The assumption, together with the identity
\[
\omega_0(Ju,v) = Ju\cdot J_0 v = - J_0 J u\cdot v \qquad \forall u,v\in \R^{2n},
\]
implies that the endomorphism $-J_0J$ is self-adjoint and positive. Therefore, its spectrum $\sigma(-J_0J)$ is contained in the positive real axis and we have
\[
\omega_0 (Ju,u) \geq \min \sigma(-J_0 J) |u|^2 = \frac{1}{\max \sigma((-J_0 J)^{-1})} |u|^2 =  \frac{1}{\max \sigma(-JJ_0)} |u|^2 \quad \forall u\in \R^{2n}.
\]
The desired inequality now follows from the identity
\[
\max \sigma(-JJ_0) = \|-JJ_0\| = \|J\|.
\]
\end{proof}

\begin{proof}[Proof of Lemma \ref{lemmulo}.]
We identify $\R^2$ with $\C$ by mapping $(s,t)$ into $w=s+it$, and we set
\[
\partial := \partial_s-J_0\partial_t,\qquad \overline{\partial} := \partial_s+J_0\partial_t.
\] 
Moreover, we identify $\R^{2n}$ with $\C^n$ by identifying $J_0$ with the multiplication by $i$, 
and we consider the linear convolution operator
\[
T: C^\infty_c(\C,\R^{2n}) \rightarrow C^\infty(\C,\R^{2n}), \qquad 
(Tv)(z):=\frac{1}{2\pi}\int_{\C}\frac{v(w)}{z-w} \,dsdt,
\] 
where $w=s+it$. The operator $T$ commutes with partial derivatives and satisfies
\[
\overline{\partial} Tv=T \overline{\partial}v=v,
\]
see \cite[Appendix A.4]{hz94}.
Moreover by the Calderon-Zygmund inequality, see \cite[Theorem B.2.7]{ms04}, for every $p\in (1,+\infty)$ there exists $C_p>0$  such that for all $v\in C^\infty_c(\C,\R^{2n})$,
\[
\|\partial Tv\|_{L^p} = \|T \partial v\|_{L^p} \leq C_p\|v\|_{L^p}.
\]
Therefore $\partial T$ extends to a bounded linear operator on $L^p(\C,\R^{2n})$. If $p=2$, this operator is an isometry. Indeed, for $v\in C^\infty_c(\C,\R^{2n})$ we have
\[
\begin{split}
\|\partial Tv \|^2_{L^2} &= \| 2 \partial_s T v  - \overline{\partial} Tv\|_{L^2}^2 = 4 \|\partial_s T v \|_{L^2}^2 - 4 (\partial_s T v,  \overline{\partial} Tv)_{L^2} + \| \overline{\partial} Tv\|_{L^2}^2 \\ &= 4 \|\partial_s T v \|_{L^2}^2 - 4 (\partial_s T v,  \partial_s Tv + J_0 \partial_t Tv)_{L^2} + \| v\|_{L^2}^2 \\ &= - 4 (\partial_s T v,  J_0 \partial_t Tv)_{L^2} + \| v\|_{L^2}^2 =  \| v\|_{L^2}^2,
\end{split}
\] 
where the last equality follows from partial integration and the skew symmetry of $J_0$. Applying the Riesz-Thorin interpolation theorem, we also have 
\[
\|\partial T\|_{L^p} \leq \|\partial T\|_{L^2}^{1-\theta}  \|\partial T\|_{L^{p'}}^\theta =  \|\partial T\|_{L^{p'}}^\theta
\]
for all $p'\geq2$ and $\theta\in(0,1)$, where $p\geq2$ is determined by $1/p=(1-\theta)/2+\theta/p'$. By fixing $p'>2$ and letting $\theta$ go to $0$, or equivalently $p$ go to $2$, we obtain
\begin{equation}\label{interpolation}
\limsup_{p\downarrow 2} \|\partial T\|_{L^p} \leq 1.
\end{equation}

Let $u\in L^p_{\mathrm{loc}}(U,\R^{2n})$ be a solution of the linear equation (\ref{linfloer}), which we rephrase as
\begin{equation}
\label{rephra}
\big(J+J_0\big)\overline{\partial} u + \big(J-J_0\big)\partial u = 2Jf.
\end{equation}
Since $J$ is $\omega_0$-compatible and uniformly bounded, Lemma \ref{lem:J_estimate} (ii) implies that
\[
\omega_0(J(w) v,v) \geq \alpha |v|^2 \qquad \forall w \in U, \; \forall v\in \R^{2n},
\]
where $\alpha:= 1/\|J\|_{\infty}$. In particular, all the complex structures $J(w)$ satisfy the assumption of \ref{lem:J_estimate} (i) with the same constant $\alpha$. 

Multiplying the both sides of (\ref{rephra}) by $(J+J_0)^{-1}$, which exists as observed in Lemma \ref{lem:J_estimate} (i), we obtain the following Beltrami equation on $U$
\begin{equation}
\label{Floer_eq1}
\overline{\partial} u + G\partial u = g
\end{equation}
where
\[
G:=(J+J_0)^{-1}(J-J_0)
\] 
and 
\[
g:=2(J+J_0)^{-1}Jf.
\]
The map $g$ belongs to $L^p_{\mathrm{loc}}(U,\R^{2n})$. By the second inequality in Lemma \ref{lem:J_estimate} (i), 
\[
\sup_{(s,t)\in U}|G(s,t)|<1.
\]
Let $\Omega\subset U$ be an open subset  with compact closure contained in $U$. Let $\rho: \C \rightarrow \R$ be a smooth function with compact support contained in $U$ and taking the value 1 on $\Omega$. We define $v\in L^p(\C,\R^{2n})$ by
\[
v:= \rho u.
\]
Since $\Omega$ is arbitrary, it is enough to show that $v$ belongs to $W^{1,q}(\C,\R^{2n})$ for some $q>2$.

From (\ref{Floer_eq1}) we deduce that $v$ solves the following Beltrami equation on $\C$
\begin{equation}
\label{Floer_eq2}
\overline{\partial} v + G\partial v = h,
\end{equation}
where the map
\[
h:= \rho g + (\overline{\partial} \rho) u + (\partial \rho) G u
\]
belongs to $L^p(\C,\R^{2n})$, and $G$ is extended to the whole $\C$ by setting it equal to $0$ outside $U$, so that we still have
\begin{equation}\label{norm_G}
\sup_{w\in \C}|G(w)|<1.
\end{equation}
Having compact support, $v$ and $h$ belong also to $L^{p'}(\C,\R^{2n})$ for every $p'\leq p$.
Since $T\overline{\partial} = I$, equation (\ref{Floer_eq2}) can be rewritten as
\begin{equation}
\label{Floer_eq3}
(I + G \partial T) \overline{\partial} v  = h.
\end{equation}
Thanks to (\ref{interpolation}) and (\ref{norm_G}), we can find $q_0>2$ such that 
\[
\| G \partial T\|_{L^{q}} < 1
\]
for every $q\in [2,q_0]$. The above inequality implies that the operator $I + G \partial T$ is invertible on $L^q(\C,\R^{2n})$ for every $q\in [2,q_0]$.
Let $q\in (2,q_0]$ be a number not larger than $p$. Since $h$ is in $L^q(\C,\R^{2n})$, the equation (\ref{Floer_eq3}) can be restated as
\[
 \overline{\partial} v  =(I + G \partial T)^{-1} h,
\]
and shows that $\overline{\partial} v$ belongs to $L^q(\C,\R^{2n})$. By the standard Calderon-Zygmund estimates, we conclude that $v$ belongs to $W^{1,q}(\C,\R^{2n})$, as we wished to prove.
\end{proof}

\section{Appendix: Heuristics behind the definition of the isomorphism}
\label{appB}

The aim of this appendix is to show a heuristic argument that explains why the Floer complex of $H$  (that is a sort of Morse complex of the direct action functional $\Phi_H$) should be isomorphic to the Morse complex of the dual action functional $\Psi_{H^*}$ and suggests the hybrid problem that leads to the formal construction of the isomorphism $\Theta$ in our main theorem. See \cite{as15} for a similar argument applied to the Hamiltonian and Lagrangian action functionals that are associated to a fiberwise convex Hamiltonian on the cotangent bundle of a closed manifold. We shall not worry about precise regularity assumptions and pretend that we can work with smooth loops. The Hamiltonian $H\in C^{\infty}(\T \times \R^{2n})$ is assumed to satisfy the conditions of the main theorem of the Introduction: It is non-degenerate, quadratically convex and non-resonant at infinity.

We denote by $C^{\infty}_0(\T,\R^{2n})$ the vector space of smooth 1-periodic curves in $\R^{2n}$ with vanishing average and by
\[
\Pi: C^{\infty}_0(\T,\R^{2n}) \rightarrow C^{\infty}_0(\T,\R^{2n}), \qquad (\Pi x)(t) := \int_0^t x(s)\, ds - \int_\T \left(\int_0^t x(s)\, ds\right) dt,
\]
the linear operator that maps each $x\in C^{\infty}_0(\T,\R^{2n})$ into its primitive having zero average. 

Next, we notice that the convexity assumptions on $H$ imply that for every $x: \T \rightarrow \R^{2n}$ the map
\[
\R^{2n} \rightarrow \R^{2n}, \qquad z\mapsto \int_{\T} \nabla H_t(x(t)+z)\, dt,
\]
is a diffeomorphism. In particular, there is a well-defined smooth map
\[
\zeta: C^{\infty}(\T,\R^{2n}) \rightarrow \R^{2n}
\]
mapping each 1-periodic curve $x$ into the unique vector $\zeta(x)$ in $\R^{2n}$ such that
\[
\int_{\T} \nabla H_t(x(t)+\zeta(x))\, dt = 0.
\]
Notice that the 1-periodic solutions $x\in C^{\infty}(\T,\R^{2n})$ of the Hamiltonian system
\[
\dot{x}(t) = X_{H_t}(x(t)) = - J_0 \nabla H_t(x(t))
\]
satisfy $\zeta(x)=0$. 

We now introduce the following extension of the direct action functional $\Phi_H$:  
\[
\widetilde{\Phi}_H : C^{\infty}(\T,\R^{2n}) \times C^{\infty}_0(\T,\R^{2n}) \rightarrow \R, \qquad
\widetilde{\Phi}_H(x,y) := \Phi_H(x) + F(x,y),
\]
where
\[
F(x,y) := \int_{\T} \bigl( H_t(x+\zeta(x)) + H_t^*(J_0\dot{y} + \nabla H_t(x+\zeta(x))) - (x+\zeta(x))\cdot (J_0 \dot{y} + \nabla H_t(x+\zeta(x))) \bigr)\, dt.
\]
By the Fenchel inequality, the integrand in the definition of $F$ is non-negative, and vanishes if and only if 
\[
J_0 \dot{y} + \nabla H_t(x+\zeta(x)) = \nabla H_t(x+\zeta(x)),
\]
that is, if and only if $y=0$. Moreover, $F$ is strictly convex in its second variable. Therefore, for every $x\in C^{\infty}(\T,\R^{2n})$ the functional $y\mapsto F(x,y)$ has a unique critical point at $y=0$, where it takes the value zero. In particular, the critical sets of $\Phi_H$ and $\widetilde{\Phi}_H$ are related by
\[
\mathrm{crit}\, \widetilde{\Phi}_H = \mathrm{crit}\,\Phi_H \times \{0\}.
\]
The Floer complex of $H$ is the Morse complex of the functional $\Phi_H$ with respect to its negative $L^2$-gradient vector field $-\nabla^{L^2}\Phi(x)$. The vector field
\[
X_{\widetilde{\Phi}_H}(x,y) := \bigl( - \nabla^{L^2} \Phi_H(x), - \rho(x,y) \nabla_2^{H^1} F(x,y) \bigr),
\]
where $\nabla_2^{H^1}$ denotes the $H^1$-gradient with respect to the second variable and the function $\rho$ is everywhere positive and satisfies
\begin{equation}
\label{pseudograd}
\rho(x,y) \|\nabla_2^{H^1}F(x,y)\|_{H^1}^2 \geq - \langle \nabla_1^{L^2} F(x,y), \nabla^{L^2} \Phi_H(x) \rangle_{L^2},
\end{equation}
is a pseudo-gradient vector field for $\widetilde{\Phi}_H$, meaning that
\[
d\widetilde{\Phi}_H(x,y)[X_{\widetilde{\Phi}_H}(x,y)]<0 \qquad \forall (x,y)\notin \mathrm{crit}\, \widetilde{\Phi}_H .
\]
Notice that a positive function $\rho$ satisfying (\ref{pseudograd}) exists because $\nabla_2^{H_1}F(x,y)$ is non-zero for $y\neq 0$ and satisfies
\[
\|\nabla_2^{H^1}F(x,y)\|_{H^1}\geq \delta \|y\|_{H^1}
\]
for some $\delta>0$, while the right-hand side of (\ref{pseudograd})  is $O(\|y\|_{H^1}^2)$ for $y\rightarrow 0$ in $H^1$.

The Floer complex of $H$  is naturally identified with the Morse complex of $\widetilde{\Phi}_H$ with respect to $X_{\widetilde{\Phi}_H}$.  Indeed, we have
\begin{eqnarray}
\label{staphi}
W^s\bigl( (x,0); X_{\widetilde{\Phi}_H}) &=& W^s\bigl( x; - \nabla^{L^2} \Phi_H) \times C^{\infty}_0(\T,\R^{2n}), \\
\label{unstaphi}
W^u\bigl( (x,0); X_{\widetilde{\Phi}_H}) &=& W^u\bigl( x; - \nabla^{L^2} \Phi_H) \times \{0\},
\end{eqnarray}
and hence
\[
W^u\bigl( (x,0); X_{\widetilde{\Phi}_H}) \cap W^s\bigl( (y,0); X_{\widetilde{\Phi}_H}) = \Bigl( W^u\bigl( x; - \nabla^{L^2} \Phi_H) \cap W^s\bigl( y; - \nabla^{L^2} \Phi_H) \Bigr) \times \{0\},
\]
for every pair of critical points $x,y$ of $\Phi_H$. Here, the stable and unstable manifolds of the critical point $x$ of $\Phi_H$ should be understood in a loose sense:  They should be interpreted as traces at $s=0$ of solutions $u=u(s,t)$ of the Floer equation
\begin{equation}
\label{floerappB}
\partial_s u + J_t(u) (\partial_t u - X_{H_t}(u)) = 0
\end{equation}
on the positive half-cylinder $[0,+\infty)\times \T$ (for the stable manifold) or on the negative half-cylinder $(-\infty,0]\times \T$ (for the unstable manifold) that are asymptotic to $x$ for $s\rightarrow +\infty$ (for the stable manifold) or for  $s\rightarrow -\infty$ (for the unstable manifold). The intersection
\[
W^u\bigl( x; - \nabla^{L^2} \Phi_H) \cap W^s\bigl( y; - \nabla^{L^2} \Phi_H) 
\]
should be interpreted as the space of traces at $s=0$ of solutions of the Floer equation (\ref{floerappB}) on the whole cylinder $\R\times \T$ that are asymptotic to $x$ for $s\rightarrow -\infty$ and to $y$ for $s\rightarrow +\infty$.

Next we introduce the following extension of the dual action functional $\Psi_{H^*}$:
\[
\begin{split}
\widetilde{\Psi}_{H^*} : C^{\infty}_0(\T,\R^{2n}) \times C^{\infty}_0(\T,\R^{2n}) \times \R^{2n} \rightarrow \R, \\
\widetilde{\Psi}_{H^*}(x,y,z) := \Psi_{H^*}(x) +  \Omega(y) + G(x,y,z),
\end{split}
\]
where
\[
\Omega(y) := \frac{1}{2} \int_{\T} J_0 \dot{y} \cdot y\, dt, \quad 
G(x,y,z) := \int_{\T} \bigl( H_t(x+y+\zeta(x+y)) - H_t(x+y+z+ \zeta(x+y)) \bigr)\, dt.
\]
The functional $\Omega$ is a non-degenerate quadratic form on $C^{\infty}_0(\T,\R^{2n})$ with positive and negative eigenspaces
\[
\mathbb{H}^+ := \mathbb{H}^+_{1/2} \cap C^{\infty}_0(\T,\R^{2n}), \qquad \mathbb{H}^- := \mathbb{H}^-_{1/2} \cap C^{\infty}_0(\T,\R^{2n}).
\]
The functional $G(x,y,z)$ it strictly concave in the third variable $z\in \R^{2n}$ and its gradient with respect to this variable is
\[
\nabla_3 G(x,y,z) = - \int_{\T} \nabla H_t(x+y+z+\zeta(x+y))\, dt.
\]
The above formula and the definition of the map $\zeta$ show that the differential of $G(x,y,z)$ with respect to the third variable vanishes for $z=0$, and from the strict concavity we deduce that 
\[
G(x,y,z)\leq G(x,y,0) = 0 \qquad \forall (x,y,z)\in C^{\infty}_0(\T,\R^{2n}) \times C^{\infty}_0(\T,\R^{2n}) \times \R^{2n},
\]
with equality if and only if $z=0$. These facts imply that the critical sets of the functionals $\widetilde{\Psi}_{H^*}$ and $\Psi_{H^*}$ are related through the identity
\[
\mathrm{crit}\, \widetilde{\Psi}_{H^*} = \mathrm{crit}\, \Psi_{H^*} \times \{0\} \times \{0\}.
\]
Arguing as above, it is easy to construct a pseudo-gradient vector field $X_{\widetilde{\Psi}_{H^*}}$ for $\widetilde{\Psi}_{H^*}$ such that the Morse complex of $\widetilde{\Psi}_{H^*}$ with respect to it is naturally identified with the Morse complex of $\Psi_{H^*}$. Indeed, one can define
\[
X_{\widetilde{\Psi}_{H^*}}(x,y,z) := \Bigl( - \nabla^{H^1} \Psi_{H^*}(x), - \nabla^{H^{1/2}} \Omega(y), - \sigma(x,y,z) \nabla_3 G(x,y,z) \Bigr),
\]
where $\sigma(x,y,z)$ is a suitable positive function that satisfies a condition analogous to (\ref{pseudograd}). With this choice, we have
\begin{eqnarray}
\label{stapsi}
W^s\bigl( (x,0,0);X_{\widetilde{\Psi}_{H^*}}) &=& W^s\bigl( x; - \nabla^{H^1} \Psi_{H^*}) \times \mathbb{H}^+ \times \{0\}, \\
\label{unstapsi}
W^u\bigl( (x,0,0); X_{\widetilde{\Psi}_{H^*}}) &=& W^u\bigl( x; - \nabla^{H^1} \Psi_{H^*}) \times  \mathbb{H}^- \times \R^{2n},
\end{eqnarray}
for every critical point $x$ of $\Psi_{H^*}$. In particular, the intersection
\[
W^u\bigl( (x,0,0); X_{\widetilde{\Psi}_{H^*}}) \cap W^s\bigl( (y,0,0);X_{\widetilde{\Psi}_{H^*}})
\]
is given by
\[
\Bigl( W^u\bigl( x; - \nabla^{H^1} \Psi_{H^*}) \cap W^s\bigl( y; - \nabla^{H^1} \Psi_{H^*}) \Bigr) \times \{0\} \times \{0\},
\]
and the Morse complex of $\widetilde{\Psi}_{H^*}$ with respect to $X_{\widetilde{\Psi}_{H^*}}$ is naturally identified with the Morse complex of $\Psi_{H^*}$ with respect to $-\nabla^{H^1} \Psi_{H^*}$ (here we are pretending that the Morse complex of $\Psi_{H^*}$ is well-defined, without the need of performing the finite dimensional reduction that we have introduced in Section \ref{sec:Morse}). 

The final observation is that the extended functionals $\widetilde{\Phi}_H$ and $\widetilde{\Psi}_{H^*}$ are, up to a change of variables, the same functional, and hence must have isomorphic Morse complexes. Indeed, a simple computation shows that
\begin{equation}
\label{idegamma}
\widetilde{\Psi}_{H^*}= \widetilde{\Phi}_H \circ \Gamma,
\end{equation}
where the map
\[
\begin{split}
\Gamma: C^{\infty}_0(\T,\R^{2n}) \times C^{\infty}_0(\T,\R^{2n}) \times \R^{2n} \rightarrow C^{\infty}(\T,\R^{2n}) \times C^{\infty}_0(\T,\R^{2n}), \\
\Gamma(x,y,z) := \bigl( x + y + z + \zeta(x+y), x + J_0 \Pi \nabla H(x+y+\zeta(x+y)) \bigr)
\end{split}
\]
is a diffeomorphism, with inverse
\[
\Gamma^{-1}(x,y) = \bigl( y - J_0 \Pi \nabla H(x+\zeta(x)), x - \overline{x} - y + J_0 \Pi \nabla H(x+\zeta(x)), - \zeta(x) \bigr).
\]
Here, $\overline{x}\in \R^{2n}$ denotes the average of the loop $x\in C^{\infty}(\T,\R^{2n})$. 

Thanks to (\ref{idegamma}), an isomorphism from the Morse complex of $\widetilde{\Psi}_{H^*}$ - which is naturally identified with the Morse complex of $\Psi_{H^*}$ - to the Morse complex of $\widetilde{\Phi}_H$ -  which is naturally identified with the Floer complex of $H$ - can be defined by looking at the following spaces
\[
\Gamma\Bigl( W^u\bigl( (x,0,0); X_{\widetilde{\Psi}_{H^*}}\bigr) \Bigr) \cap W^s\bigl( (y,0); X_{\widetilde{\Phi}_H}),
\]
for every critical point $x$ of $\Psi_{H^*}$ and $y$ of $\Phi_H$. By the identities (\ref{staphi}) and (\ref{unstapsi}), the above space is given by the traces at $s=0$ of the solutions $u$ of the Floer equation (\ref{floerappB}) on the positive half-cylinder $[0,+\infty) \times \T$ that are asymptotic to $y$ for $s\rightarrow +\infty$ and satisfy the boundary condition
\[
u(0,\cdot) \in \pi^{-1}\bigl( W^u(x;-\nabla^{H^1} \Psi_{H^*}) \bigr)  + \mathbb{H}^-,
\]
where $\pi: C^{\infty}(\T,\R^{2n}) \rightarrow C^{\infty}_0(\T,\R^{2n})$ denotes the canonical projection $\pi(x) = x - \overline{x}$. Up to replacing the functional $\Psi_{H^*}$ with its finite dimensional reduction $\psi_{H^*}$, this is precisely the hybrid problem that we have considered in order to construct the isomorphism $\Theta$ of our main theorem.

The above heuristic considerations also show which hybrid problem one would have to consider in order to define directly an isomorphism from the Floer complex of $H$ to the Morse complex of $\Psi_{H^*}$. Indeed, the relevant space is in this case
\[
W^u\bigl( (x,0); X_{\widetilde{\Phi}_H} \bigr) \cap \Gamma\Bigl( W^s\bigl( (y,0,0); X_{\widetilde{\Psi}_{H^*}}\bigr) \Bigr)
\]
for $x\in \mathrm{crit}\, \Phi_H$ and $y\in \mathrm{crit}\, \Psi_{H^*}$, and the identities (\ref{unstaphi}) and (\ref{stapsi}) show that one should look at solutions $u$ of the Floer equation (\ref{floerappB}) on the negative half-cylinder $(-\infty,0]\times \T$ that satisfy the boundary conditions
\[
\begin{split}
\int_{\T} \nabla H_t(u(0,t))\, dt &= 0, \\
- J_0 \Pi \nabla H(u(0,\cdot)) &\in W^s(y;-\nabla^{H^1} \Psi_{H^*}), \\
u(0,\cdot) - \overline{u(0,\cdot)} +  J_0 \Pi  \nabla H(u(0,\cdot)) &\in \mathbb{H}^+.
\end{split}
\]
Since the fact that the chain map $\Theta$ is an isomorphism can be proven in a much simpler way, we shall not elaborate on the above hybrid problem any further.


\begin{thebibliography}{FCHW96}

\bibitem[Abb01]{abb01}
A.~Abbondandolo, \emph{Morse theory for {H}amiltonian systems}, Pitman Research
  Notes in Mathematics, vol. 425, Chapman \& Hall, London, 2001.
  
\bibitem[AS18]{as18}
A.~Abbondandolo and F.~Schlenk, \emph{Floer homologies, with applications},
  Jahresber. Dtsch. Math.-Ver. (online first) (2018).  

\bibitem[AS06]{as06}
A.~Abbondandolo and M.~Schwarz, \emph{{On the Floer homology of cotangent
  bundles}}, Comm. Pure Appl. Math. \textbf{59} (2006), 254--316.

\bibitem[AS09]{as09b}
A.~Abbondandolo and M.~Schwarz, \emph{A smooth pseudo-gradient for the {L}agrangian action
  functional}, Adv. Nonlinear Stud. \textbf{9} (2009), 597--623.

\bibitem[AS15]{as15}
A.~Abbondandolo and M.~Schwarz, \emph{The role of the {L}egendre transform in the study of the {F}loer
  complex of cotangent bundles}, Comm. Pure Appl. Math. \textbf{68} (2015),
  1885--1945.
  
  \bibitem[AD14]{ad14}
M.~Audin and M.~Damian, \emph{Morse theory and {F}loer homology}, Springer,
  2014.

\bibitem[Boj10]{boj10}
B.~Bojarski, \emph{On the {B}eltrami equation, once again: 54 years later},
  Ann. Acad. Sci. Fenn. Math. \textbf{35} (2010), 59--73.

\bibitem[BO09]{bo09b}
F.~Bourgeois and A.~Oancea, \emph{{Symplectic homology, autonomous
  Hamiltonians, and Morse-Bott moduli spaces}}, Duke Math. J. \textbf{146}
  (2009), 71--174.

\bibitem[Bro86]{bro86}
V.~Brousseau, \emph{L'index d'un syst{\`e}me hamiltonien lin{\'e}aire}, C. R.
  Acad. Sci. Paris S{\'e}r. I Math. \textbf{303} (1986), 351--354.

\bibitem[Bro90]{bro90}
V.~Brousseau, \emph{Espaces de {K}rein et index des syst\`emes hamiltoniens}, Ann.
  Inst. H. Poincar\'e Anal. Non Lin\'eaire \textbf{7} (1990), 525--560.
  
\bibitem[CFH95]{cfh95}
K.~Cieliebak, A.~Floer, and H.~Hofer, \emph{Symplectic homology {II}. {A}
  general construction}, Math. Z. \textbf{218} (1995), no.~1, 103--122.  
  
\bibitem[Cla79]{cla79}
F.~H. Clarke, \emph{A classical variational principle for periodic
  {H}amiltonian trajectories}, Proc. Amer. Math. Soc. \textbf{76} (1979),
  186--188.

\bibitem[Cla81]{cla81}
F.~H. Clarke, \emph{Periodic solutions to {H}amiltonian inclusions}, J. Differential
  Equations \textbf{40} (1981), 1--6.  

\bibitem[CE80]{ce80}
F.~H. Clarke and I.~Ekeland, \emph{Hamiltonian trajectories having prescribed
  minimal period}, Comm. Pure Appl. Math. \textbf{33} (1980), 103--116.

\bibitem[CE82]{ce82}
F.~H. Clarke and I.~Ekeland, \emph{Nonlinear oscillations and boundary value problems for
  {H}amiltonian systems}, Arch. Rational Mech. Anal. \textbf{78} (1982),
  315--333.

\bibitem[Eke90]{eke90}
I.~Ekeland, \emph{Convexity methods in {H}amiltonian systems}, Ergebnisse der
  Mathematik und ihrer Grenzgebiete (3), vol.~19, Springer-Verlag, Berlin,
  1990.
  
\bibitem[EH87]{eh87} 
I.~Ekeland and H.~Hofer, \emph{Convex Hamiltonian energy surfaces and their periodic trajectories},
Comm. Math. Phys. \textbf{113} (1987), 419--469.
  
\bibitem[EH89]{eh89}
I.~Ekeland and H.~Hofer, \emph{Symplectic topology and {H}amiltonian dynamics}, 
Math. Z. \textbf{200} (1989), 355--378.  
 
  
\bibitem[EH90]{eh90}
I.~Ekeland and H.~Hofer, \emph{Symplectic topology and {H}amiltonian dynamics
  {II}}, Math. Z. \textbf{203} (1990), 553--567.  

\bibitem[FCHW96]{cfhw96}
A.~Floer, K.~Cieliebak, H.~Hofer, and K.~Wysocki, \emph{Applications of
  symplectic homology {II}. {S}tability of the action spectrum}, Math. Z. 
  \textbf{223} (1996), 27--45.

\bibitem[FH94]{fh94}
A.~Floer and H.~Hofer, \emph{{Symplectic homology. I. Open sets in
  $\mathbb{C}\sp n$}}, Math. Z. \textbf{215} (1994), 37--88.

\bibitem[FHW94]{fhw94}
A.~Floer, H.~Hofer, and K.~Wysocki, \emph{Applications of symplectic homology
  {I}}, Math. Z. \textbf{217} (1994), 577--606.

\bibitem[FR78]{fr78}
E.~Fadell and P.~Rabinowitz, \emph{Generalized cohomological index theories 
for Lie group actions with an application to bifurcation questions for Hamiltonian systems},
Invent. Math. \textbf{45} (1978), 139--174.


\bibitem[GH18]{gh18}
J.~Gutt and M.~Hutchings, \emph{Symplectic capacities from positive
  {$S^1$}-equivariant symplectic homology}, Algebr. Geom. Topol \textbf{18}
  (2018), 3537--3600.

\bibitem[Hec12]{hec12}
M.~Hecht, \emph{Isomorphic chain complexes of {H}amiltonian dynamics on tori},
  Ph.D. thesis, Universit\"at Leipzig, 2012.

\bibitem[Hec13]{hec13}
M.~Hecht, \emph{Isomorphic chain complexes of {H}amiltonian dynamics on tori},
  J. fixed point theory appl. \textbf{14} (2013), 165--221.

\bibitem[Her04]{her04}
D.~Hermann, \emph{Inner and outer {H}amiltonian capacities}, Bull. Soc. Math.
  France \textbf{132} (2004), 509--541.

\bibitem[HZ90]{hz90}
H.~Hofer and E.~Zehnder, \emph{A new capacity for symplectic manifolds},
  Analysis, et cetera, Academic Press, Boston, MA, 1990, pp.~405--427.

\bibitem[HZ94]{hz94}
H.~Hofer and E.~Zehnder, \emph{Symplectic invariants and {H}amiltonian dynamics}, Birkh\"auser,
  Basel, 1994.

\bibitem[Iri14]{iri14b}
K.~Irie, \emph{Symplectic homology of disc cotangent bundles of domains in
  {E}uclidean space}, J. Symplectic Geom. \textbf{12} (2014), 511--552.
  
  \bibitem[Iri19]{iri19}
K.~Irie, \emph{Symplectic homology capacity of convex bodies and loop space homology}, {\tt arXiv:1907.09749 [math.SG]}, 2019.

\bibitem[IS99]{is99}
S.~Ivashkovich and V.~Shevchishin, \emph{Complex curves in almost-complex
  manifolds and meromorphic hulls}, {\tt arXiv:math/9912046 [math.CV]}, 1999.

\bibitem[IS00]{is00}
S.~Ivashkovich and V.~Shevchishin, \emph{Gromov compactness theorem for {$J$}-complex curves boundary},
  Int. Math. Res. Not. (IMRN) \textbf{22} (2000), 1167--1206.

\bibitem[Kan14]{Kan14}
J.~Kang, \emph{Symplectic homology of displaceable Liouville domains and leafwise intersection points},
Geom. Dedicata \textbf{170} (2014), 135--142.


\bibitem[Lon02]{lon02}
Y.~Long, \emph{Index theory for symplectic paths with applications},
  Birkh\"auser, Basel, 2002.

\bibitem[LZ90]{lz90}
Y.~Long and E.~Zehnder, \emph{Morse-theory for forced oscillations of
  asymptotically linear {H}amiltonian systems}, Stochastic processes, physics
  and geometry (Ascona and Locarno, 1988) (Teaneck, N.J.) (S.~Albeverio et~al.,
  ed.), World Sci. Publishing, 1990, pp.~528--563.

\bibitem[MS04]{ms04}
D.~McDuff and D.~Salamon, \emph{{$J$}-holomorphic curves and symplectic
  topology}, Colloquium Publications, vol.~52, American Mathematical Society,
  Providence, R.I., 2004.

\bibitem[Oan04]{oan04}
A.~Oancea, \emph{A survey of {F}loer homology for manifolds with contact type
  boundary or symplectic homology}, Symplectic geometry and {F}loer homology,
  Ensaios Mat., vol.~7, Soc. Brasil. Mat., Rio de
  Janeiro, 2004, pp.~51--91.

\bibitem[Roc70]{Roc70}
R.~T.~Rockafellar, \emph{Convex analysis}, 
Princeton Mathematical Series, No. 28 Princeton University Press, 1970.


\bibitem[Sal99]{sal99}
D.~Salamon, \emph{Lectures on {F}loer homology}, Symplectic geometry and
  topology (Y.~Eliashberg and L.~Traynor, eds.), IAS/Park City Mathematics
  Series, Amer. Math. Soc., 1999, pp.~143--225.

\bibitem[Sei08]{sei08}
P.~Seidel, \emph{A biased view of symplectic cohomology}, Current Developments
  in Mathematics, 2006, International Press, 2008, pp.~211--253.

\bibitem[Vek62]{vek62}
I.~N. Vekua, \emph{Generalized analytic functions}, Pergamon Press,
  London-Paris-Frankfurt; Addison-Wesley Publishing Co., Inc., Reading, Mass.,
  1962.

\bibitem[Vit89]{vit89b}
C.~Viterbo, \emph{Equivariant {M}orse theory for starshaped {H}amiltonian
  systems}, Trans. Amer. Math. Soc. \textbf{311} (1989), 621--655.

\bibitem[Vit99]{vit99}
C.~Viterbo, \emph{{Functors and computations in Floer homology with applications,
  I}}, Geom. Funct. Anal. \textbf{9} (1999), 985--1033.

\end{thebibliography}

\providecommand{\bysame}{\leavevmode\hbox to3em{\hrulefill}\thinspace}
\providecommand{\MR}{\relax\ifhmode\unskip\space\fi MR }
\providecommand{\MRhref}[2]{%
  \href{http://www.ams.org/mathscinet-getitem?mr=#1}{#2}
}
\providecommand{\href}[2]{#2}

\end{document}